\documentclass[reqno,11pt]{amsart}
\usepackage{amssymb,amsmath,amsthm,}
    \usepackage{a4wide}
\usepackage{amscd}
\usepackage{amsfonts}
\usepackage{amssymb}
\usepackage{latexsym}
\usepackage{color}
\usepackage{esint}
\usepackage{graphicx}
\usepackage{float}
\graphicspath{{Figures/}}
\allowdisplaybreaks
\usepackage{graphicx}
\usepackage[breaklinks=true,bookmarks=false]{hyperref}
\usepackage{cleveref}

\usepackage[makeroom]{cancel}

\newtheorem{theorem}{Theorem}

\newtheorem{definition}{Definition}
\newtheorem{lemma}{Lemma}
\newtheorem{proposition}[theorem]{Proposition}
\newtheorem{remark}{Remark}

\let\e=\varepsilon

\let\p=\partial

\let\O=\Omega

\numberwithin{equation}{section}

\let\hide\iffalse
\let\unhide\fi

\newcommand{\lfty}{L^\infty_{x,\xi}}
\newcommand{\lftyv}{L^\infty_{\xi}}

\newcommand{\R}{\mathbb{R}}

\newcommand{\be}{\begin{equation}}
\newcommand{\bm}{\begin{multline}}
\newcommand{\ee}{\end{equation}}
\newcommand{\dd}{\mathrm{d}}

\newcommand{\xb}{x_{\mathbf{b}}}
\newcommand{\tb}{t_{\mathbf{b}}}

\newcommand{\Bes}{\begin{eqnarray*}}
\newcommand{\Ees}{\end{eqnarray*}}
\newcommand{\Be}{\begin{equation} }
\newcommand{\Ee}{\end{equation}}

\def\p{\partial}

\def\O{\Omega}
\def\R{\mathbb{R}}

\def\B{\begin{equation}}
\def\E{\end{equation}}
\def\BN{\begin{eqnarray*}}
\def\EN{\end{eqnarray*}}

\begin{document}
\title{On regularity of a Kinetic Boundary layer}

 \author{Hongxu Chen}
\address{Department of Mathematics, The Chinese University of Hong Kong, Shatin, N.T., Hong Kong, email:  hongxuchen.math@gmail.com}
 
\date{\today}

\begin{abstract} We study the nonlinear steady Boltzmann equation in the half space, with phase transition and Dirichlet boundary condition. In particular, we study the regularity of the solution to the half-space problem in the situation that the gas is in contact with its condensed phase. We propose a novel kinetic weight and establish a weighted $C^1$ estimate under the spatial domain $x\in [0,\infty)$, which is unbounded and not strictly convex. Additionally, we prove the $W^{1,p}$ estimate without any weight for $p<2$.

\end{abstract}

\maketitle

\section{Introduction}
\subsection{Background}

The regularity issue for the boundary value problem of the Boltzmann equation has been a challenging problem due to the singularity of the characteristic near the boundary and the nonlocal nature of the collision operator. Significant progress has been achieved in a convex bounded domain by Guo-Kim-Tonon-Trescases in \cite{GKTT}, where the boundedness and strict convexity of the domain play crucial roles in the analysis (see \cite{chen2023geometric} for a discussion on curvature). In this paper, we are interested in the regularity estimate in a domain that lacks both boundedness and strict convexity, specifically, we consider the half-space kinetic boundary layer problem. Other domains with similar properties are discussed in Remark \ref{rmk:weight}.

We consider the phase transition problem in the kinetic theory of gas. This problem can be modeled by the steady Boltzmann equation in the half space with slab-symmetry and suitable boundary conditions, and converging to some Maxwellian equilibrium at the far field. Let $F(x,v)$ be the mass density of the gas particle at distance $x>0$ and velocity $v\in \mathbb{R}^3$. Then the half-space problem is formulated as
\begin{equation}\label{boundary_layer_equation}
\begin{cases}
   v_1\partial_x F(x,v)  = Q(F,F)(x,v), \ v\in \mathbb{R}^3, \ x>0, \\
   F(x,v)      \to  M_{1,u,1}(v), \ \ \text{ as } x\to \infty.     
\end{cases}
\end{equation}
The Maxwellian equilibrium in \eqref{boundary_layer_equation} is denoted as $M_{\rho,u,T}(v)$. It is characterized by the constant parameters $\rho \in \mathbb{R}^+$, $u \in \mathbb{R}$, and $T \in \mathbb{R}^+$, which represent the mass density, flow velocity, and temperature, respectively. The equilibrium is expressed as:

\begin{equation}
M_{\rho,u,T}(v) = \frac{\rho}{(2\pi T)^{3/2}} \exp\left(-\frac{(v_1 - u)^2 + v_2^2 + v_3^2}{2T}\right). \label{maxwell_equi}
\end{equation}

In \eqref{boundary_layer_equation}, $Q(F,F)$ is the Boltzmann collision operator. We consider the hard sphere model with an angular cut-off kernel, where $Q$ is defined as
\begin{equation}\label{Q_operator}
Q(F,F)(x,v):= \int_{\mathbb{R}^3\times \mathbb{S}^2} (F(x,v')F(x,v_*') - F(x,v)F(x,v_*)) |(v-v_*)\cdot \omega| \dd \omega \dd v_*.
\end{equation}
In \eqref{Q_operator} $v',v_*'$ represent the post-collision velocity, which can be determined through $\omega \in \mathbb{S}^2$ and the conservation of momentum and energy:
\begin{align*}
    & v' + v_*' = v+v_*, \ \ \  |v'|^2 + |v'_*|^2 =  |v|^2 + |v_*|^2, \\
    & v' = v-((v-v_*)\cdot \omega) \omega, \  v_*' = v_* + ((v-v_*)\cdot \omega)\omega .
\end{align*}

The well-posedness and asymptotic behavior of the linear half-space problem (linearized version of \eqref{boundary_layer_equation}), also known as Milne problem, have been studied in the pioneering work by Bardos-Caflisch-Nicolaenko \cite{bardos1986milne}. The Milne problem serves as a fundamental boundary layer problem with specified incoming boundary condition. This theory has found significant applications in the hydrodynamic limit problem of the kinetic equation when there exhibits a mismatch between the kinetic and fluid boundaries, as demonstrated in \cite{wu2015,ghost,wu}.

For the nonlinear half-space problem \eqref{boundary_layer_equation} with Dirichlet boundary conditions, not all Dirichlet data are admissible and the admissible conditions depend on the far field Maxwellian. When $u$ does not take the singular values $\{0,+\sqrt{5/3},-\sqrt{5/3}\}$, Ukai-Yang-Yu devised a penalization method to construct a unique solution in \cite{ukai2003,ukai2004nonlinear}, and the solution is proved to converge exponentially to $0$ as $x\to \infty$. In this study, the convergence rate and the condition on the admissible boundary data depend on the parameter $u$. When $u$ takes the singular value $u=0$, Golse studied the well-posedness and boundary admissible condition in \cite{golse2008}. 

In our paper, we focus on the case of $0<|u|\ll 1$ with the Dirichlet boundary condition, which corresponds to the phase transition in the condensation-evaporation problem. For more comprehensive details on the boundary admissible condition of these problems, we refer readers to \cite{bardos2006} and the reference therein. We also refer readers to \cite{sone2002, sone2007} for extensive numerical computation on these topics.

To begin, we denote the standard global Maxwellian by $M:=M_{1,0,1}$ with $M_{\rho,u,T}$ defined in \eqref{maxwell_equi}. We apply a change of variable $\xi : = v-(u,0,0),$
and set the Boltzmann equation $F$ \eqref{boundary_layer_equation} as a perturbation around the Maxwellian
\[F(x,v) = (M+\sqrt{M}f)(x,v-(u,0,0)).\]
The equation of the perturbation $f$, with $\xi = v-(u,0,0)$, reads
\begin{equation}\label{equation_f}
\begin{cases}
  (\xi_1+u)\p_x f(x,\xi) +\mathcal{L}f(x,\xi)   = \Gamma(f,f)  \\
  f(x,\xi)   
 \to 0 \text{ as } x\to \infty.     
\end{cases}
\end{equation}
Here $\mathcal{L}f$ is the linear Boltzmann operator, which is defined as
\begin{equation}\label{linear_operator}
 \mathcal{L}f := \nu(\xi) f - K(f)=-\frac{Q(M,\sqrt{M}f)}{\sqrt{M}}-\frac{Q(\sqrt{M}f,M)}{\sqrt{M}}.
\end{equation}
$\Gamma(f,f)$ is the nonlinear Boltzmann operator defined as
\begin{equation}\label{Gamma_f}
\Gamma(f,g) := \frac{Q(\sqrt{M}f,\sqrt{M}g) + Q(\sqrt{M}g,\sqrt{M}f)}{2\sqrt{M}}.    
\end{equation}
The properties of these Boltzmann operators are summarized in Lemma \ref{lemma:k_gamma}.

Now we discuss boundary conditions of \eqref{equation_f}. We assume that the gas is in contact with its condensed phase as in \cite{golse}. At $x=0$, we impose a Dirichlet boundary condition:
\begin{equation}\label{f_b}
f(0,\xi) = f_b(\xi) , \ \ \ \xi_1 + u>0.
\end{equation}

For other boundary conditions such as diffuse boundary condition and specular boundary condition, the well-posedness and asymptotic behavior are obtained in \cite{golse1988, coron1988}, and recently, with continuity, in \cite{huang_diffuse,huang_specular} under different functional spaces. We refer to \cite{jiang2024knudsen,jiang2025knudsen} for more recent progress on the range of Mach number. 

Recently, for $u$ closed to the singular value $0$, Bernhoff and Golse \cite{golse} proposed an elegant approach to establish the well-posedness and asymptotic stability under certain assumptions on the boundary data(also see \cite{liu2013}). In this result, the solution $f$ is proved to exhibit slab symmetry and converge exponentially as $x\to \infty$, with the convergence rate being uniform in $|u|\ll 1$. We document this well-posedness result in Theorem \ref{thm:well-pose} in Section \ref{sec:prelim}.

\hide

If the system~\eqref{equation_f}, \eqref{bc} has a unique solution that is invariant under the action of some orthogonal matrix. We further assume
\begin{equation}\label{bc_symmetry}
f_b(\xi_1,\xi_2,\xi_3) =  f_b(\xi_1,-\xi_2,-\xi_3),
\end{equation}
then $f\circ R$ is also a solution, where

By uniqueness, we have $f\circ R = f$, and thus we focus on the function space
\begin{equation}\label{function_space}
\mathcal{H} := \{\phi\in L^2(\dd \xi)|\phi \circ R = \phi\}. 
\end{equation}

\unhide

\ \\

\subsection{Main result}

While the well-posedness of \eqref{equation_f} is well understood in \cite{golse} and other literature, the regularity of such a boundary layer problem remains open. In this paper, we tackle this problem by constructing the weighted $C^1$ estimate and $W^{1,p},p<2$ estimate without weight.

It is well-known that with the presence of the boundary, the regularity of the Boltzmann equation possesses singularity due to the non-local nature of the Boltzmann operator. As demonstrated in \cite{KL,G,K,GKTT2}, in a general convex domain, the spatial derivative $\nabla_x f(x,\xi)$ generates singularity as $1/(n(\xb)\cdot \xi)$, where $\xb$ corresponds to the backward exit position defined as
\begin{align*}
    & \xb(x,\xi) = x-\tb(x,\xi)\xi, \ \tb(x,\xi) = \sup\{s\geq 0, \ x-s\xi \in \O\}.
\end{align*}

In particular, in \cite{GKTT}, Guo-Kim-Tonon-Trescases proposed a kinetic weight that can compensate the singularity and capture the convexity of domain at the same time. This innovative approach successfully overcomes the challenges posed by the non-local properties of the Boltzmann equation and establishes a weighted $C^1$ estimate. The introduction of kinetic weight has significantly advanced the understanding of the Boltzmann regularity, see \cite{CK, chen2, CKL, Ikun}.

In our specific problem, it is evident that $\partial_x f$ exhibits a singularity as $1/(\xi_1+u)$. However, due to the lack of strict convexity and boundedness in the domain, we need to introduce a new kinetic weight to address this singularity. The proposed kinetic weight is defined as follows:
\begin{definition}[Kinetic weight]
We define
\begin{equation}\label{alpha}
\tilde{\alpha}(x,\xi):=\sqrt{|\xi_1+u|^2+(c\nu_0)^2 x^2}    ,
\end{equation}
where $\nu_0$ corresponds to the lower bound of the collision frequency defined in~\eqref{frequency}. $0<c<1$ is a constant that will be specified later in \eqref{c_8}.

We introduce a cut-off function for the kinetic weight. Define
$\chi: [0,\infty) \rightarrow [0,\infty)$ which stands for a non-decreasing smooth function such that
\Be\label{chi}
\chi (s) = s \ \text{for} \ s \in [0, 1/2 ],  \ \chi(s) = 1 \ \text{for} \ s \in [ 2, \infty ), \ and \  | \chi^\prime(s) | \leq 1 \ \text{for}  \  s \in [0,\infty).
\Ee

Then we introduce a cut-off function to the kinetic weight:
\begin{equation}\label{kinetic_weight}
\alpha(x,\xi): = \chi(\tilde{\alpha}(x,\xi)).
\end{equation}
With the extra cutoff function, we directly have 
\begin{equation}\label{alpha_bdd}
\alpha(x,\xi) \leq 1, \ \ \ 1\leq \frac{1}{\alpha(x,\xi)} .  
\end{equation}    
\end{definition}

\begin{remark}

On the boundary $x=0$, clearly we have
\begin{align}
    &  \alpha(0,\xi) \leq |\xi_1+u|, \ \ \frac{1}{|\xi_1+u|} \leq \frac{1}{\alpha(0,\xi)}. \label{alpha_bdr}
\end{align} 
Thus the singularity $\frac{1}{|\xi_1+u|}$ can be compensated by $\alpha(x,\xi)$ when $|\xi_1+u|\ll 1$.

From \eqref{equation_f}, we can control $\partial_x f(x,\xi)$ using a trivial weight $\xi_1+u$. At the boundary $x=0$, $\partial_x f$ is expected to be discontinuous and possess a singularity as $\frac{1}{\xi_1+u}$. Away from the boundary, $f$ is expected to be continuous; however, the trivial weight $\xi_1+u$ does not provide information regarding this continuity. Our weight \eqref{alpha} is non-zero away from the boundary, and controlling $\alpha \partial_x f$ implies the desired estimate for $\partial_x f$ when $x>0$, as illustrated in Theorem \ref{thm:weight_C1}.

\end{remark}

\begin{remark}\label{rmk:weight}
The kinetic weight proposed in \cite{GKTT} is almost invariant along the characteristic thanks to the strict convexity and boundedness. It is important to note that even under a flat and unbounded domain, our weight \eqref{alpha} still possesses favorable properties similar to \cite{GKTT} in the following two aspects.

I. Due to the presence of the spatial variable $x$ in \eqref{alpha}, our weight $\alpha(x,\xi)$ does not remain invariant along the characteristic. In fact, the rate of change along the characteristic of these weights can be explicitly computed (see Lemma \ref{lemma:velocity} and Lemma \ref{lemma:velocity_alpha}). Importantly, it turns out that the rate of change depends on the choice of $c$. Meanwhile, the collision frequency $\nu$ from the linear Boltzmann operator $\mathcal{L}$ is bounded below by $\nu_0$ according to Lemma \ref{lemma:k_gamma}. By selecting an appropriate $c$, the growth factor $e^{2c\nu_0 s}$ in Lemma \ref{lemma:velocity_alpha} along the characteristic can be compensated by the damped factor $e^{-\nu s}$.

II. The extra spatial variable $x$ in \eqref{alpha} plays a key role in the regularity estimate of the non-local Boltzmann operator. For example, in \eqref{nonlocal_local}, the integral over $\dd \xi'$ does not blow up; instead, it becomes a function of the spatial variable $x$. This allows us to proceed the computation and focus on the $\dd s$ integral. We refer more detailed computation to Lemma \ref{lemma:NLN}.

We anticipate that this type of kinetic weight can be adapted to study the regularity issue in unbounded domains with flat boundaries. For instance, the initial boundary value problem of the Boltzmann equation in \cite{CJ1}, \cite{CKJ} with domain $\mathbb{T}^2\times \mathbb{R}^+$, and in \cite{chen2024global} with domain $\mathbb{R}^2\times (-1,1)$.

\end{remark}

In the following result, we establish the regularity estimate of the solution to the boundary layer problem \eqref{equation_f}. We denote the exponential weight in velocity $\xi$ as
\begin{equation}\label{w_weight}
\begin{split}
    & w(\xi) : = e^{\theta|\xi|^2}, \ \theta< \frac{1}{4}, \ \ \  w_{\tilde{\theta}}(\xi): = e^{\tilde{\theta}|\xi|^2} \text{ for } \tilde{\theta} \ll \theta.
\end{split}
\end{equation}
We define the grazing set as
\begin{equation}\label{grazing}
\mathcal{D}: = \{(0,\xi)|\xi\in \mathbb{R}^3, \ \ \xi_1+u=0\}.
\end{equation}
We denote the inner product and orthongonal matrix as
\begin{align}
    & \langle f \rangle := \int_{\mathbb{R}^3} f(\xi) \dd \xi,  \label{inner_product}
\end{align}
\begin{equation} \label{matrix}
\mathcal{R} := 	\begin{pmatrix} 
	1 & 0 & 0 \\
	0 & -1 & 0\\
	0 & 0 & -1 \\
	\end{pmatrix} .
\end{equation}

\begin{theorem}\label{thm:weight_C1}
Assume $0<|u|<r$ for some $r\ll 1$. Assume that for some $\e \ll 1$, the boundary data $f_b$ \eqref{f_b} satisfy
\begin{equation}\label{f_b_continuous}
f_b\circ \mathcal{R} = f_b, \ f_b \in C(\mathbb{R}^3) \text{ and } \Vert w f_b\Vert_{L^\infty_{\xi}} \leq \e, \ \ \mathcal{R} \text{ defined in } \eqref{matrix},
\end{equation}
and
\begin{equation}\label{extra_assumption}
\langle (\xi_1+u)Y_1[u] \mathfrak{R}_{u,\gamma}[f_b] \rangle = \langle (\xi_1+u)Y_2[u] \mathfrak{R}_{u,\gamma}[f_b]\rangle = 0.
\end{equation}
Here $Y_1[u]$, $Y_2[u]$, and $\mathfrak{R}_{u,\gamma}$ are defined in Appendix \ref{appendix}.

We further assume that \eqref{equation_f} with boundary condition \eqref{f_b} has a unique solution $f$ such that for some $C>0$ and $\gamma \ll 1$,
\begin{align}
    & f\in C(\mathbb{R}^+ \times \mathbb{R}^3 \backslash \mathcal{D}) ,\ \Vert w f(x,\xi)\Vert_{L^\infty_\xi} \leq C \e e^{-\gamma x}.  \label{assume_soln} 
\end{align}
Then for some $C_\infty>0$, the unique solution $f$ satisfies the weighted $C^1$ estimate:
\begin{equation}\label{weighted_C1}
\Vert w_{\tilde{\theta}}\alpha \p_x f\Vert_{L^\infty_{x,\xi}} \leq C_\infty\e e^{-\gamma x} .
\end{equation}
For $1\leq p<2$ and $\gamma_0 < \gamma$, the solution $f$ also satisfies the $W^{1,p}$ and $H^1_{loc}$ estimate:
\begin{align}\label{W1p}
    &\Vert w_{\tilde{\theta}/2}e^{\gamma_0 x}\p_x f\Vert_{L^{p}_{x,\xi}} \leq C_p \e  \text{ for some } C_p>0 ,
\end{align}
\begin{align}\label{local_H1}
    &  \int_{\delta}^\infty \int_{\mathbb{R}^3} w_{\tilde{\theta}}(\xi) e^{2\gamma_0 x}|\p_x f|^2 \dd \xi \dd x \leq C_\delta \text{ for any } \delta>0 \text{ and some } C_\delta>0.
\end{align}

\end{theorem}

\begin{remark}
The continuity and $w$-weighted $L^\infty$ estimate assumption \eqref{assume_soln} will be justified in Theorem \ref{thm:continuity}.

Bernhoff and Golse construct a unique solution $f$ using polynomial weight $(1+|\xi|^3)$(see Theorem \ref{thm:well-pose}) in \cite{golse}. In this study, they focus on the well-posedness analysis of a penalized problem, which subsequently establishes the well-posedness of the original problem \eqref{equation_f} by eliminating the penalization. To remove the penalization, additional assumptions \eqref{extra_assumption} are introduced. The assumptions \eqref{extra_assumption} also imply that the incoming boundary forms a manifold of co-dimension two for given $u$, as stated in \cite{golse}.

In our paper, we do not incorporate these additional conditions in our analysis. For precise definitions of the functions $Y_1[u]$, $Y_2[u]$, and $\mathfrak{R}_{u,\gamma}$ in \eqref{extra_assumption}, we refer readers to the Appendix \ref{appendix} and also to \cite{golse}.
\end{remark}

\begin{remark}
For the $W^{1,p}$ estimate with $p<2$, the exponential weight $w_{\tilde{\theta}}(\xi)$ and exponential decay $e^{-\gamma x}$ in the weighted $C^1$ estimate \eqref{weighted_C1} guarantee the integrability as $x,|\xi|\to \infty$. For a general convex, smooth, and bounded domain, \cite{CK_2023} establishes a $W^{1,p}$ estimate for $p<3$ by employing control over the backward exit time $\tb \lesssim |n(\xb)\cdot \xi|/|\xi|^2$. This control allows for a gain in local integrability with respect to the singularity $1/(n(\xb)\cdot \xi)$. We notice that in our paper, we do not have the benefit of such an additional control since our domain is unbounded and not strictly convex.

\end{remark}

\begin{remark}
The $H^1_{loc}$ estimate \eqref{local_H1} coincides with the regularity of the Milne problem, as stated in Theorem 3.3 in \cite{bardos1986milne}:
\begin{align*}
    & \int_{\delta}^\infty \int_{\mathbb{R}^3} (1+|\xi|)e^{2\gamma x} |\p_x f|^2 \dd \xi \dd x \leq C_\delta.
\end{align*}
Bardos-Caflisch-Nicolaenko derived this upper bound through an $L^2$ energy estimate to the equation of $\p_x f$.  In our Theorem \ref{thm:weight_C1}, we deal with the nonlinear problem, and we derive \eqref{local_H1} through direct computation using the weighted $C^1$ estimate \eqref{weighted_C1}. In this computation, the extra weight in \eqref{weighted_C1} generates extra singularity $\frac{1}{|\alpha(x,\xi)|^2}$ in integration, which is controllable away from $x=0$, as demonstrated in Section \ref{sec:proof_strategy}.

\end{remark}

The assumption \eqref{assume_soln} is justified in the following theorem.

\begin{theorem}\label{thm:continuity}
There exists $\e$, $r$, $\gamma$ and $C$ such that if the boundary data $f_b$ \eqref{f_b} satisfies
\begin{equation*}
f_b\circ \mathcal{R} = f_b, \ f_b \in C(\mathbb{R}^3) \text{ and } \Vert w f_b\Vert_{L^\infty_{\xi}} \leq \e, \ \ \mathcal{R} \text{ defined in } \eqref{matrix},
\end{equation*}
\begin{equation*}
\langle (\xi_1+u)Y_1[u] \mathfrak{R}_{u,\gamma}[f_b] \rangle = \langle (\xi_1+u)Y_2[u] \mathfrak{R}_{u,\gamma}[f_b]\rangle = 0.
\end{equation*}
for each $0<|u|\leq r$. Then \eqref{equation_f} with boundary condition \eqref{f_b} has a unique solution $f$ that is continuous away from the grazing set: $f\in C(\mathbb{R}^+ \times \mathbb{R}^3 \backslash \mathcal{D})$, and
\begin{equation*}
f(x,\mathcal{R}\xi) = f(x,\xi) \text{ for a.e } (x,\xi)\in \mathbb{R}_+\times \mathbb{R}^3 .
\end{equation*}
Moreover, $f$ satisfies the uniform-in-$u$ decay estimate
\begin{equation*}
 \Vert w f(x,\xi)\Vert_{L^\infty_\xi} \leq C \e e^{-\gamma x}.
\end{equation*}

\end{theorem}
\begin{remark}\label{remark:w_weight}
The $w$-weighted $L^\infty$ estimate \eqref{assume_soln} is crucial in constructing the weighted $C^1$ regularity \eqref{weighted_C1} in Theorem \ref{thm:weight_C1}. As stated in \eqref{Gamma_est} in Lemma \ref{lemma:k_gamma}, the nonlinear operator $\Gamma$ shares a similar structure with the linear operator $Kf$ \eqref{Kf} by employing the weight $w$. This allows us to control the contribution of both $K$ and $\Gamma$ in the weighted $C^1$ estimate using a uniform strategy.

\end{remark}

\ \\

\subsection{Difficulty and main idea}\label{sec:proof_strategy} 

\ 

\textbf{1. Continuity proof by penalization.} To begin, we justify the continuity and $w$ weighted $L^\infty$ estimate assumption \eqref{assume_soln} in Theorem \ref{thm:continuity}. Following the approach in \cite{golse}, we introduce a linearized penalized problem, as described in Proposition \ref{prop:pen_well_pose}. The construction of the penalized operator is based on solving the eigen-value problem \eqref{eigen_prob}. In \cite{golse}, the corresponding eigen-function $\phi_u$ is proved to be bounded in $L^\infty$ with polynomial weight $(1+|\xi|)^s$. To prove $\Vert wg\Vert_{L^\infty_{x,\xi}}$, the first step is to control the eigen-function $\phi_u$ in $L^\infty$ with exponential weight. We use a modified linear Boltzmann operator,
\begin{align*}
    \mathcal{L}_\theta \phi_u  & = -\frac{Q(M,\sqrt{M}e^{-\theta |\xi|^2} \phi_u)+Q(\sqrt{M}e^{-\theta |\xi|^2} \phi_u, M)}{\sqrt{M}e^{-\theta |\xi|^2}}   
  = \nu(\xi) \phi_u - K_\theta \phi_u.
\end{align*}
The modified linear operator $K_\theta$ shares a structure similar to the original operator $K$ (see Lemma \ref{lemma:k_theta}). The eigen-value problem associated with the modified operator can be solved in the same spirit. We obtain the eigen-function to the original problem as $e^{-\theta |\xi|^2} \phi_u$ with the desired weight(see Lemma \ref{lemma:eigen_continuous} for detail). 

Subsequently, we employ a fixed-point argument in the $w$-weighted $L^\infty$ space to establish the well-posedness and the continuity, for the linearized penalized problem in Proposition \ref{prop:pen_continuous}, and after, for the nonlinear penalized problem, in Proposition \ref{prop:nonlinear_continuous}.

\textbf{2. Regularity}

\textbf{2-a. Weighted $C^1$ estimate.} The construction of the derivative involves fixed-point argument to the penalized problem \eqref{pen_prob} with an a-priori weighted $C^1$ estimate. In the following, we illustrate the main difficulty and idea in the a-priori estimate(Lemma \ref{lemma:aprori_linear}).

The difficulty of the singularity $\frac{1}{\xi_1+u}$ is addressed by the proposed kinetic weight $\alpha(x,\xi)$. To control its rate of change along the characteristic(Lemma \ref{lemma:velocity_alpha}), we set $c=1/8$ and thus this growth rate can be controlled by the collision frequency $\nu$ in \eqref{linear_operator}.

The derivative along the characteristic is given by the piece-wise formula \eqref{p_x_1}-\eqref{p_x_6}. The \textit{main difficulty} comes from the contribution of the linear operator $K(\p_x g)$ in \eqref{p_x_4}. By employing the Grad estimate \eqref{Kf} to $K(\p_x g)$, we see a $\dd \xi'$ integration in this term as
\begin{align*}
    & \int^t_0 \dd s e^{-\nu(t-s)} \int_{\mathbb{R}^3} \mathbf{k}(\xi,\xi') \p_x g(x-(\xi_1+u)(t-s),\xi') \dd \xi'. 
\end{align*}
From a standard technique, we iterate along the characteristic with velocity $\xi'$ again and obtain a double Duhamel's formula(in \eqref{kk}):
\begin{align*}
    & \int_0^t \dd s e^{-\nu(t-s)} \int_{\mathbb{R}^3} \mathbf{k}(\xi,\xi') \int^s_0 \dd s' e^{-\nu(s-s')} \int_{\mathbb{R}^3}\mathbf{k}(\xi',\xi'') \p_x g(x-(\xi_1+u)(t-s) - (\xi_1'+u)(s-s'),\xi'').
\end{align*}

To control this term, we expect to utilize the estimate on $\Vert g\Vert_\infty$ or gain a smallness contribution from the time integration. To achieve our goal, we split the $\dd s'$ integral into two regions: $s-s'\leq \e$ and $s-s'>\e$. When $s-s'>\e$, we observe the following change of variable
\begin{align*}
    & \frac{\p_{\xi'_1} [g(y-(\xi'_1+u)(s-s'),\xi'')]}{-(s-s')}    =  \p_x g(y-(\xi'_1+u)(s-s'),\xi''). 
\end{align*}
With the lower bound of $s-s'$, we convert the $x$-derivative $\p_x g$ into a velocity $\xi_1$ derivative $\p_{\xi_1}[g]$. Via an integration by part, we remove the derivative in $\p_{\xi_1} [g]$ and control the $g$  using the $L^\infty$ estimate mentioned in the previous paragraphs(Theorem \ref{thm:continuity}). On the other hand, the case $s-s'<\e$ corresponds to a small time contribution, we need to incorporate the nonlocal operator $K(\p_x g)$ with the kinetic weight $\alpha$. By introducing the kinetic weight $\alpha$ in the Duhamel's formula and isolating $\Vert \alpha \p_x g\Vert_{L^\infty_\xi}$, we arrive at the following integral: 
\begin{align}
    &\int_{t-\e}^t \dd s e^{-\nu(\xi)(t-s)}\int_{\mathbb{R}^3} \dd \xi' \frac{e^{-C|\xi-\xi'|^2}}{|\xi-\xi'|} \frac{1}{\alpha(x-(t-s)(\xi_1+u),\xi')}.\label{nonlocal_local}
\end{align}
This type of estimate was first studied in \cite{GKTT} when the domain is bounded and strictly convex. In our case, the integrability in $\dd \xi'$ arises from the additional $x^2$ term in \eqref{alpha}. The $\dd \xi'$ integration can be explicitly computed as a function of $\xi,s$ (see \eqref{ln_bdd_eta} and \eqref{ln_bdd_1} for detail): 
\begin{align*}
    & \int_{t-\e}^t \dd s e^{-\nu(\xi)(t-s)} \ln \Big(1+ \Big|\frac{1}{x-(\xi_1+u)(t-s)}\Big|^2 \Big). 
\end{align*}
We aim to control this integral by $\frac{1}{\alpha(x,v)}$ to close the weighted $C^1$ estimate $\Vert \alpha \p_x g\Vert_\infty.$ Despite the lack of strict convexity and boundedness, our setting has one beneficial aspect - the spatial variable is one-dimensional. Along the characteristic, the spatial variable can be explicitly computed as $x - (\xi_1 + u)(t - s)$. We tackle this integration by carefully comparing the scale of $\frac{1}{x - (\xi_1 + u)(t - s)}$ with the scale of $\frac{1}{\alpha(x,v)}$(approximately $\frac{1}{\sqrt{|\xi_1 + u|^2 + x^2}}$ in \eqref{alpha}). We refer the detailed analysis and conclusion of this integral to Lemma \ref{lemma:NLN}.

The contribution of the nonlinear operator \eqref{p_x_6} can be controlled in the same spirit. The corresponding integral reads
\begin{align*}
    & \int^t_0 \dd s e^{-\nu(t-s)}  \p_x\Gamma(g,g)(x-(\xi_1+u)(t-s),\xi) \dd \xi \\
    &\lesssim \Vert wg\Vert_{L^\infty_{x,\xi}} \int_0^t \dd e^{-\nu(t-s)} \int_{\mathbb{R}^3} \mathbf{k}(\xi,\xi') |\p_x g(x-(\xi_1+u)(t-s),\xi')| \dd \xi'.
\end{align*}
Here the $w$ weight plays an important role to have the kernel $\mathbf{k}(\xi,\xi')$ in the integration, as emphasized in Remark \ref{remark:w_weight}. The smallness of $\Vert wg\Vert_{L^\infty_{x,\xi}}$ plays the same role as the small time contribution discussed in the previous paragraph. In fact, the large time integration includes an additional growing factor in the time scale $t$, as stated in Lemma \ref{lemma:NLN}. Consequently, we set $t$ to be large but fixed, and absorb it using the smallness of $\Vert wg\Vert_{L^\infty_{x,\xi}}$.

\textbf{2-b. $W^{1,p}$ estimate.} The $W^{1,p}$ estimate without weight in \eqref{W1p} is derived after obtaining the weighted $C^1$ estimate. In the previous paragraph we establish the weighted $C^1$ estimate to the penalized problem $g$. The solution $f$ to \eqref{equation_f} is set to be (see Section \ref{sec:proof_contin} for detail)
\[f(x,\xi) = e^{-\gamma x} g(x,\xi).\]
The exponential decay factor $e^{-\gamma x}$ and the velocity weight in $\Vert w_{\tilde{\theta}}(\xi)\alpha \p_x g\Vert_{L^\infty_{x,\xi}}$ provide the integrability as $x,|\xi|\to \infty$. In \eqref{W1p}, the constraint of $p<2$ arises from the local integrability of the term $\frac{1}{\alpha^p(x,\xi)}$. The local integral of $\frac{1}{\alpha^p(x,\xi)}\sim \frac{1}{(|\xi_1+u|^2 + x^2)^{p/2}}$ corresponds to a two-dimensional integration involving the variables $\xi_1$ and $x$. Therefore, this integral over $x\in (0,1)$ is bounded for $p<2$, and the integral over $x\in (\delta,1)$ is bounded by some constant $C_\delta$ for $\delta>0$ and $p=2$. This justifies the $H^1_{loc}$ estimate in \eqref{local_H1}.

\ \\

\textbf{Outline.} In Section \ref{sec:prelim}, we provide the properties of the kinetic weight and Boltzmann operator as preliminary material. In Section \ref{sec:continuity}, we introduce the penalized problem and analyze its continuity as well as the exponential weighted $L^\infty$ estimate. This analysis allows us to justify the assumption \eqref{assume_soln} and conclude Theorem \ref{thm:continuity} after removing the penalization. In Section \ref{sec:regularity}, we apply the properties of kinetic weight to conclude Theorem \ref{thm:weight_C1} by establishing the weighted $C^1$ and $W^{1,p}$ estimate for $p<2$.

\ \\

\section{Preliminary}\label{sec:prelim}

First, we document the well-posedness result of \eqref{equation_f} in \cite{golse}.

\begin{theorem}[Theorem 2.1 in \cite{golse}]\label{thm:well-pose}
There exists $\e$, $r$, $\gamma$ and $C$ such that if the boundary data $f_b$ \eqref{f_b} satisfies
\begin{equation}\label{bdr_small}
f_b\circ \mathcal{R} = f_b \text{ and } \Vert (1+|\xi|)^3 f_b\Vert_{L^\infty_{\xi}} \leq \e,
\end{equation}
for each $0<|u|\leq r$. Then \eqref{equation_f} with boundary condition \eqref{f_b} has a unique solution that satisfies
\begin{equation*}
f(x,\mathcal{R}\xi) = f(x,\xi) \text{ for a.e } (x,\xi)\in \mathbb{R}_+\times \mathbb{R}^3 
\end{equation*}
and the uniform decay estimate
\begin{equation*}
 \Vert (1+|\xi|)^3 f(x,\xi)\Vert_{L^\infty_\xi} \leq C \e e^{-\gamma x}
\end{equation*}
if and only if the boundary data $f_b$ satisfies further conditions
\begin{equation*}
\langle (\xi_1+u)Y_1[u] \mathfrak{R}_{u,\gamma}[f_b] \rangle = \langle (\xi_1+u)Y_2[u] \mathfrak{R}_{u,\gamma}[f_b]\rangle = 0.
\end{equation*}
Here $Y_1[u]$, $Y_2[u]$, and $\mathfrak{R}_{u,\gamma}$ are defined in Appendix \ref{appendix}.
\end{theorem}

\subsection{Properties of the Boltzmann operator in \eqref{linear_operator} and \eqref{Gamma_f}.}

\begin{lemma}\label{lemma:k_gamma}
The linear Boltzmann operator $K(f)$ in~\eqref{equation_f} is given by
\begin{align}
& Kf(x,\xi)=\int_{\mathbb{R}^3}\mathbf{k}(\xi,\xi')f(x,\xi')\dd \xi'. \label{Kf}  
\end{align}

The kernel $\mathbf{k}(\xi,\xi')=\mathbf{k}_1(\xi,\xi')+\mathbf{k}_2(\xi,\xi')$ is given by the Grad estimate \cite{R}:
\begin{align}
\mathbf{k}_1(\xi,\xi')   &= C_{\mathbf{k}_1}|\xi-\xi'|e^{-\frac{|\xi|^2+|\xi'|^2}{4}}, \label{k1}\\
 \mathbf{k}_2(\xi,\xi')  & = C_{\mathbf{k}_2}\frac{1}{|\xi-\xi'|}e^{-\frac{1}{8}|\xi-\xi'|^{2}-\frac{1}{8}
\frac{(|\xi|^{2}-|\xi'|^{2})^{2}}{|\xi-\xi'|^{2}}}. \label{k2}
\end{align}

The kernel satisfies
\begin{equation*}
 |\mathbf{k}  (\xi,\xi')| \lesssim \mathbf{k}^\varrho (\xi,\xi'), \  \ \mathbf{k}^\varrho (\xi,\xi') := e^{- \varrho |\xi-\xi'|^2}/ |\xi-\xi'|.    
\end{equation*}

For $\nu$ and $\Gamma$ in~\eqref{equation_f}, we have
\begin{equation}\label{nablav nu}
\nu(\xi) \gtrsim 1+|\xi|,\quad |\nabla_\xi \nu(\xi)| \lesssim 1,
\end{equation}

\begin{equation}\label{Gamma bounded}
\begin{split}
 \Big\Vert \frac{w}{1+|\xi|}\Gamma(f,g)(x,\xi) \Big\Vert_{L^\infty_{x,\xi}}    & \lesssim \Vert wf\Vert_{L^\infty_{x,\xi}}\times \Vert wg\Vert_{L^\infty_{x,\xi}},
\end{split}
\end{equation}

\begin{equation}\label{Gamma_est}
\begin{split}
 | \nabla_x \Gamma(f,g)(x,\xi) | &\lesssim [1+|\xi|] \| wf \|_{\lfty}
  | \nabla_x g(x,\xi) |  + \Vert wf\Vert_{\lfty}
  \int_{\R^3} \mathbf{k}(\xi,\xi')
 |\nabla_x g(x,\xi')|
  \dd \xi'  \\
  & + \Vert wg\Vert_{\lfty} \int_{\mathbb{R}^3} \mathbf{k}(\xi,\xi') |\nabla_x f(x,\xi')| \dd \xi'.
\end{split}
\end{equation}

\end{lemma}

\begin{proof}
The proof of \eqref{k1} -- \eqref{nablav nu} can be found in \cite{R}. The proof of \eqref{Gamma bounded} can be found in \cite{G}. 

We only prove \eqref{Gamma_est}. We take $x$-derivative to have
\begin{align*}
   |\nabla_x \Gamma(f,g)| & = |\Gamma(\nabla_x f,g) + \Gamma(f,\nabla_x g) |  \notag \\
   &\leq \Vert wf\Vert_{L^\infty_{x,\xi}} \Big|\Gamma(e^{-\theta |\cdot|^2},\nabla_x g) \Big| + \Vert wg\Vert_{L^\infty_{x,\xi}}\Big|\Gamma(\nabla_x f, e^{-\theta |\cdot |^2})  \Big|.
\end{align*}
The $\Gamma$ terms above have the same structure as the linear operator $\mathcal{L}$ in \eqref{linear_operator} with replacing $\mu$ by a different exponent $\sqrt{\mu}e^{-\theta |\xi|^2}$. Then \eqref{Gamma_est} follows the expression of $Kf$ in \eqref{k1} and \eqref{k2} with replacing $1/8$ by another coefficient. For ease of notation, we keep the same $\mathbf{k}(\xi,\xi')$ in \eqref{Gamma_est}.

\end{proof}

\begin{lemma}

In the linear operator~\eqref{linear_operator}, the collision frequency $\nu(\xi)$ satisfies
\begin{equation}\label{frequency}
\nu_0 [1+|\xi|] \leq  \nu(\xi) \leq \nu_1 [1+|\xi|].
\end{equation}

The kernel of $\mathcal{L}$ is
\begin{equation*}
\text{Ker}L = \text{Span}(X_+,X_-,X_0,\xi_2\sqrt{M} ,\xi_3 \sqrt{M}),
\end{equation*}
where
\begin{equation}\label{basis}
X_{\pm} = \frac{1}{\sqrt{30}}(|\xi|^2+\sqrt{15}\xi_1)\sqrt{M} ,  \ \ X_0 = \frac{1}{\sqrt{10}}(|\xi|^2-5) \sqrt{M}.
\end{equation}
Denote
\[\langle \phi \rangle := \int_{\mathbb{R}^3}\phi(\xi)\dd \xi.\]
The family $(X_+,X_0,X_-,\xi_2,\xi_3)$ is orthonormal in $L^2(\mathbb{R}^3;\dd \xi)$, and is orthogonal for the bilinear form $(f,g)\to \langle \xi_1 fg\rangle$, with
\[\langle \xi_1 X_{\pm}^2\rangle = \pm \sqrt{\frac{5}{3}}, \ \ \langle \xi_1 X_0^2\rangle = \langle \xi_1 X_2^2\rangle = \langle \xi_1 \xi_3^2\rangle = 0.\]
\hide
Then $\mathcal{H}\cap \text{Ker}L = \text{Span}(X_+,X_0,X_-)$.
\unhide
\end{lemma}

\begin{proof}
The proof is standard and we omit it.
\end{proof}

\begin{lemma}\label{lemma:k_theta}
Let $0\leq \theta < \frac{1}{4}$, denote $\mathbf{k}_\theta(\xi,\xi') := \mathbf{k}(\xi,\xi') \frac{e^{\theta |\xi|^2}}{e^{\theta |\xi'|^2}}.$ Then we have
\begin{equation}\label{k_theta}
\int_{\mathbb{R}^3}  \mathbf{k}(\xi,\xi') \frac{e^{\theta |\xi|^2}}{e^{\theta |\xi'|^2}}  \dd \xi'  \lesssim \frac{C}{1+|\xi|}.
\end{equation}
And there exists $C_\theta > 0$
\begin{equation}\label{k_theta_bdd}
\mathbf{k}_\theta (\xi,\xi')\lesssim  \frac{e^{-C_\theta |\xi-\xi'|^2}}{|\xi-\xi'|} .
\end{equation}
\hide
Moreover, for $N\gg 1$, we have
\begin{equation}\label{k_N_upper_bdd}
\mathbf{k}_\theta(\xi,\xi') \mathbf{1}_{|\xi-\xi'|> \frac{1}{N}} \leq C_N,
\end{equation}
and
\begin{equation}\label{K_N_small}
\int_{|\xi'|>N \text{ or } |\xi-\xi'|\leq \frac{1}{N}} \mathbf{k}_\theta(\xi,\xi') \dd \xi' \lesssim \frac{1}{N} \leq o(1).
\end{equation}
\unhide

The derivative on $\mathbf{k}(\xi,\xi')$ shares similar property: for $0<\tilde{\theta}\ll \theta,$ 
\begin{equation}\label{nabla_k_theta}
\vert \nabla_\xi \mathbf{k}(\xi,\xi')\vert \frac{e^{\tilde{\theta}\vert \xi\vert^2}}{e^{\tilde{\theta}\vert \xi'\vert^2}} \lesssim \frac{[1+|\xi|^2]e^{-C_\theta |\xi-\xi'|^2}}{\vert \xi-\xi'\vert^2}    .
\end{equation}

\end{lemma}

\begin{proof}
The proof of \eqref{k_theta} and \eqref{k_theta_bdd} can be found in \cite{G}. We provide the proof of \eqref{nabla_k_theta}.

Taking the derivative to \eqref{k1} and \eqref{k2} we have
\begin{align*}
\nabla_v \mathbf{k}_1(\xi,\xi')    & = \frac{\xi-\xi'}{\vert \xi-\xi'\vert}\mathbf{k}_1(\xi,\xi') - \xi \mathbf{k}_1(\xi,\xi'),
\end{align*}
\begin{align*}
\nabla_v \mathbf{k}_2(\xi,\xi')    &  = \frac{\xi-\xi'}{\vert \xi-\xi'\vert^2} \mathbf{k}_2(\xi,\xi') \\
&- \mathbf{k}_2(\xi,\xi')\Big[\frac{\xi-\xi'}{4} + \frac{\xi(\vert \xi'\vert^2-\vert \xi\vert^2)\vert \xi-\xi'\vert^2 - (\vert \xi'\vert^2-\vert \xi\vert^2)^2  (\xi-\xi') }{4\vert \xi-\xi'\vert^4} \Big] .
\end{align*}
Since $|\xi'|^2 - |\xi|^2    = |\xi'-\xi|^2 + 2\xi\cdot (\xi'-\xi) $, we have
\begin{align*}
   & \big| |\xi'|^2 - |\xi|^2\big| \lesssim |\xi'-\xi|^2 + |\xi||\xi'-\xi| , \\
   & \big| |\xi'|^2 - |\xi|^2\big|^2 \lesssim |\xi'-\xi|^4 + |\xi|^2 |\xi'-\xi|^2.
\end{align*}
This leads to
\begin{align*}
\Big\vert\nabla_\xi \mathbf{k}(\xi,\xi') \frac{e^{\tilde{\theta}\vert \xi\vert^2}}{e^{\tilde{\theta}\vert \xi'\vert^2}} \Big\vert    &\lesssim  \big[\frac{1+|\xi'|^2+|\xi|^2}{\vert \xi-\xi'\vert}+\vert \xi-\xi'\vert+[1+|\xi|]\big] [\mathbf{k}_1(\xi,\xi')+\mathbf{k}_2(\xi,\xi') ]   \frac{e^{\tilde{\theta}\vert \xi\vert^2}}{e^{\tilde{\theta}\vert \xi'\vert^2}}  \\
    &\lesssim \big[\frac{1 +|\xi|^2}{\vert \xi-\xi'\vert}+\vert \xi-\xi'\vert+[1+|\xi|]\big] \mathbf{k}^{C_\theta}(\xi,\xi') \\
    & =\big[\frac{1 +|\xi|^2}{\vert \xi-\xi'\vert}+\vert \xi-\xi'\vert+[1+|\xi|]\big] \frac{e^{-C_\theta|\xi-\xi'|^2}}{|\xi-\xi'|} \lesssim \frac{1+|\xi|^2}{|\xi-\xi'|} \frac{e^{-c|\xi-\xi'|^2}}{|\xi-\xi'|}
\end{align*}
for some $c<C_\theta$. In the second line we have applied \eqref{k_theta}. In the last line we used that for $c<C_\theta$, we have
\begin{align*}
    &  e^{-C_\theta|\xi-\xi'|^2} \lesssim \frac{e^{-c|\xi-\xi'|^2}}{|\xi-\xi'|}, \ e^{-C_\theta|\xi-\xi'|^2} \lesssim \frac{e^{-c|\xi-\xi'|^2}}{|\xi-\xi'|^2}.
\end{align*}
For ease of notation, we conclude \eqref{nabla_k_theta} with coefficient $C_\theta$.

\end{proof}

\ \\

\subsection{Properties of kinetic weight $\alpha$ in \eqref{alpha}.}\label{sec:prelim_weight}

A key property of the kinetic weight~\eqref{alpha} is that $\tilde{\alpha}$ grows exponentially fast with a factor $c$ along the characteristic. We take 
\begin{equation}\label{c_8}
c:=\frac{1}{8},    
\end{equation} 
so that the exponential growth can be controlled by a faster exponential decay from the collision frequency $\nu(\xi)$ in~\eqref{frequency}.
\begin{lemma}[Velocity Lemma]\label{lemma:velocity}
Along the characteristic, whenever $x\geq 0$ and $ x-(\xi_1+u)s\geq 0$, for $\tilde{\alpha}(x,\xi)$ defined in \eqref{alpha}, we have
\begin{equation}\label{velocity_lemma}
   e^{-c\nu_0s/2}\tilde{\alpha}(x-s(\xi_1+u),\xi) \leq \tilde{\alpha}(x,\xi) \leq e^{c\nu_0s/2} \tilde{\alpha}(x-s(\xi_1+u),\xi).
\end{equation}
\end{lemma}

\begin{proof}
We take spatial derivative to $\tilde{\alpha}(x,\xi)$ and have
\begin{align}
 & |(\xi_1+u) \p_x \tilde{\alpha}(x,\xi)| \notag\\
 & =\Big|(\xi_1+u)  \frac{(c\nu_0)^2 x}{\tilde{\alpha}(x,\xi)}    \Big| =  \frac{|(c\nu_0)^{1/2}(\xi_1+u) (c\nu_0)^{3/2}x|}{\tilde{\alpha}(x,\xi)}\notag\\
  &\leq \frac{(c\nu_0)(\xi_1+u)^2+ (c\nu_0)^{3} x^2}{2\tilde{\alpha}(x,\xi)} = \frac{c\nu_0 (\tilde{\alpha}(x,\xi))^2}{2\tilde{\alpha}(x,\xi)} =\frac{c\nu_0\tilde{\alpha}(x,\xi)}{2}.\label{derivative_alpha}
\end{align}
Since
\begin{align*}
  \frac{\dd }{\dd s}\tilde{\alpha}(x-s(\xi_1+u),\xi)  & = -(\xi_1+u)\p_x \tilde{\alpha}(x-s(\xi_1+u),\xi) ,
\end{align*}
by \eqref{derivative_alpha}, we have 
\begin{align*}
  \frac{\dd }{\dd s}\tilde{\alpha}(x-s(\xi_1+u),\xi)  &   \leq \frac{c\nu_0\tilde{\alpha}(x-s(\xi_1+u),\xi)}{2}.
\end{align*}
By Gronwall's inequality, we conclude~\eqref{velocity_lemma}.

\end{proof}

With the extra cut-off function, the weight \eqref{kinetic_weight} shares similar property as demonstrated in the following lemma:
\begin{lemma}\label{lemma:velocity_alpha}
Along the characteristic, when $x\geq 0$ and $x-(\xi_1+u)s\geq 0$, for $\alpha(x,\xi)$ defined in \eqref{kinetic_weight}, we have
\begin{equation}\label{almost_invariant}
  e^{-2c\nu_0 s}\alpha(x-s(\xi_1+u),\xi) \leq \alpha(x,\xi) \leq e^{2c\nu_0 s} \alpha(x-s(\xi_1+u),\xi).
\end{equation}

\end{lemma}

\begin{proof}
First we prove that $\chi(s)$ in~\eqref{chi} has the following property:
\begin{equation}\label{chi_property}
 s\chi'(s)\leq 4\chi(s).   
\end{equation}
By~\eqref{chi}, when $s\geq 2$, $s\chi'(s)=0\leq 4\chi(s)$. When $s\leq \frac{1}{2}$, we have $\chi(s)=s,$ and thus $s\chi'(s)=s \leq 4\chi(s)=4s$. When $\frac{1}{2}<s<2$, since $\chi'(s)\leq 1$, we have $s\chi'(s)\leq s<2=4\chi(1/2)\leq 4\chi(s)$. Then we conclude~\eqref{chi_property}.

Then we compute
\begin{align*}
   &|(\xi_1+u)\p_x \alpha(x,\xi)| \\
   &  = |(\xi_1+u) \chi'(\tilde{\alpha}(x,\xi))\p_x \tilde{\alpha}(x,\xi)|  \\
   & =\chi'(\tilde{\alpha}(x,\xi)) |(\xi_1+u)\p_x \tilde{\alpha}(x,\xi)| \leq   \frac{c\nu_0}{2} \chi'(\tilde{\alpha}(x,\xi))\tilde{\alpha}(x,\xi) \\
   &  \leq  2c\nu_0 \chi(\tilde{\alpha}(x,\xi)) =   2c\nu_0 \alpha(x,\xi).
\end{align*}
In the third line, we applied the computation in~\eqref{derivative_alpha}. In the last line, we used~\eqref{chi_property}. Then by 
\[\frac{\dd }{\dd s} \alpha(x-s(\xi_1+u),\xi) = -(\xi_1+u) \p_x \alpha(x-s(\xi_1+u),\xi)\]
and Gronwall's inequality, we conclude~\eqref{almost_invariant}.

\end{proof}

The next lemma addresses the integral \eqref{nonlocal_local} mentioned in Section \ref{sec:proof_strategy}, which consists of the weight $1/\alpha$ and $K(f)$ in Lemma \ref{lemma:k_theta}.
\begin{lemma}\label{lemma:NLN}
Let $t\gg 1$, $T\leq t$ and $x-T(\xi_1+u)\geq 0$. For $T>1$, we have
\begin{align}
    & \int^t_{t-T} \dd s e^{-\nu(\xi)(t-s)/2} \int_{\mathbb{R}^3} \dd \xi' \frac{e^{-C|\xi-\xi'|^2}}{|\xi-\xi'|} \frac{1}{\alpha(x-(t-s)(\xi_1+u),\xi')}    \lesssim \frac{t}{\alpha(x,\xi)}. \label{nln_large}
\end{align}   
For $T\leq 1$, we have
\begin{align}
     & \int^t_{t-T} \dd s e^{-\nu(\xi)(t-s)/2} \int_{\mathbb{R}^3} \dd \xi' \frac{e^{-C|\xi-\xi'|^2}}{|\xi-\xi'|} \frac{1}{\alpha(x-(t-s)(\xi_1+u),\xi')}    \lesssim \frac{\sqrt{T}+T\ln(t)}{\alpha(x,\xi)}. \label{nln_small}
\end{align}

In result, for $\e\ll 1$ such that $\e\ln(t)\ll 1$ and $x-\e(\xi_1+u)\geq 0$,
\begin{align}
    &  \int^t_{t-\e} \dd s e^{-\nu(\xi)(t-s)/2} \int_{\mathbb{R}^3} \dd \xi' \frac{e^{-C|\xi-\xi'|^2}}{|\xi-\xi'|} \frac{1}{\alpha(x-(t-s)(\xi_1+u),\xi')}    \lesssim \frac{\sqrt{\e}+\e \ln(t)}{\alpha(x,\xi)}. \label{nln_e}
\end{align}

\end{lemma}

\begin{remark}
We note that there is a time-growing factor $t$ in the upper bound. In the steady problem, we will fix $t$ to be large and fixed, and this factor will be damped by either $e^{-\nu(\xi)t}$ or $\gamma$(e.g. see \eqref{pen_lin_op} and \eqref{K_123456_bdd}).
\end{remark}

\begin{proof}
We only consider the case $\alpha(x-(t-s)(\xi_1+u),\xi') = \tilde{\alpha}(x-(t-s)(\xi_1+u),\xi')$. For the other case, from~\eqref{chi}, we have $\alpha(x-(t-s)(\xi_1+u),\xi')\geq \frac{1}{2}$, this leads to
\begin{align*}
  &\int_{t-T}^t \dd s e^{-\nu(\xi)(t-s)/2} \int_{\mathbb{R}^3} \dd \xi'
 \mathbf{1}_{\alpha\neq \tilde{\alpha}}\frac{e^{-C|\xi-\xi'|^2}}{|\xi-\xi'|\alpha(x-(t-s)(\xi_1+u),\xi')}   \\
  & \lesssim \int_{t-T}^t \dd s e^{-\nu(\xi)(t-s)/2} \int_{\mathbb{R}^3} \dd \xi'  \frac{e^{-C|\xi-\xi'|}}{|\xi-\xi'|} \lesssim \min\{1,T\} \lesssim   \frac{\min\{1,T\}}{\alpha(x,\xi)}.
\end{align*}
The last inequality follows from~\eqref{alpha_bdd}.

Then we focus on
\begin{equation*}
\int_{t-T}^t \dd s e^{-\nu(\xi)(t-s)/2} \int_{\mathbb{R}^3} \dd \xi' \frac{e^{-C|\xi-\xi'|^2}}{|\xi-\xi'|} \frac{1}{\sqrt{(\xi_1'+u)^2 + (c\nu_0)^2(x-(t-s)(\xi_1+u))^2             }}.
\end{equation*}

\textbf{Step 1: integral over $\dd \xi'$.}

First we compute the $\dd \xi'$ integral. We use a notation 
\begin{equation}\label{eta}
\eta := \xi+(u,0,0).     
\end{equation}
We apply a change of variable $\xi' + (u,0,0) \to \xi'$. Denote $\xi_{\parallel} = (0,\xi_2,\xi_3)$, then the $\dd \xi'$ integral reads
\begin{align}
    &\int_{\mathbb{R}^3} \frac{e^{-C|\xi-\xi'|^2}}{|\xi-\xi'|} \frac{1}{\sqrt{(\xi_1'+u)^2 + (c\nu_0)^2(x-(\xi_1+u)(t-s))^2}} \dd \xi'   \notag\\
    & =  \int_{\mathbb{R}^3} \frac{e^{-C|\eta-\xi'|^2}}{|\eta-\xi'|} \frac{1}{\sqrt{|\xi_1'|^2 + (c\nu_0)^2(x-(\xi_1+u)(t-s))^2}} \dd \xi'  \notag\\
    & \lesssim \iint \frac{e^{-C|\eta_\parallel - \xi'_\parallel|^2}}{|\eta_\parallel - \xi'_\parallel |} \int_{\mathbb{R}} \frac{e^{-C|\eta_1-\xi'_1|^2}}{\sqrt{|\xi'_1|^2 + (c\nu_0)^2 (x-(\xi_1+u)(t-s))^2  }}      \dd \xi_1'  \notag \\
    & \lesssim \int_{\mathbb{R}} \frac{e^{-C|\eta_1-\xi'_1|^2}}{\sqrt{|\xi'_1|^2 + (c\nu_0)^2 (x-(\xi_1+u)(t-s))^2  }}      \dd \xi_1'  \notag\\
    & = \int_{\mathbb{R}}  \mathbf{1}_{|\xi_1'|\leq 2|\eta_1|} + \mathbf{1}_{|\xi_1'| > 2|\eta_1|}  \dd \xi_1'  \notag\\
    & \lesssim   \int_0^{2|\eta_1|} \frac{1}{\sqrt{|\xi'_1|^2 + (c\nu_0)^2 (x-(\xi_1+u)(t-s))^2  }}   \dd \xi_1' \notag \\
    & + \int_{\mathbb{R}} \frac{e^{-C|\xi_1'|^2/4}}{\sqrt{|\xi'_1|^2 + (c\nu_0)^2 (x-(\xi_1+u)(t-s))^2  }} \dd \xi_1'  \notag\\
    & \lesssim 1+\Big[\int_0^1 + \int_0^{2|\eta_1|}\Big]\frac{1}{\sqrt{|\xi'_1|^2 + (c\nu_0)^2 (x-(\xi_1+u)(t-s))^2  }} \dd \xi_1'. \label{xi_integral}
\end{align}
In the second last line we used that for $|\xi_1'|>2|\eta_1|$,
\begin{align*}
    &  |\eta_1-\xi_1'|\geq |\xi_1'|-|\eta_1|\geq |\xi_1'|-\frac{|\xi_1'|}{2}=\frac{|\xi_1'|}{2}.
\end{align*}
In the last line we used
\begin{align*}
    & \int_1^\infty  \frac{e^{-C|\xi_1'|^2/4}}{\sqrt{|\xi'_1|^2 + (c\nu_0)^2 (x-(\xi_1+u)(t-s))^2  }}\lesssim 1.
\end{align*}
The integral in~\eqref{xi_integral} reads
\begin{align}
    & \int_0^{2|\eta_1|} \frac{1}{\sqrt{|\xi'_1|^2 + (c\nu_0)^2 (x-(\xi_1+u)(t-s))^2  }} \dd \xi_1' \notag\\
    & = \ln\Big(\sqrt{|\xi_1'|^2 +  (c\nu_0)^2 (x-(\xi_1+u)(t-s))^2} + |\xi_1'|\Big)\Big|_{0}^{2|\eta_1|} \notag\\
    & = \ln \Big(\sqrt{4|\eta_1|^2 + (c\nu_0)^2 (x-(\xi_1+u)(t-s))^2}+2|\eta_1| \Big) \notag\\
    &- \ln\Big(\sqrt{(c\nu_0)^2 (x-(\xi_1+u)(t-s))^2} \Big) \notag\\
    & = \ln \Big(\sqrt{1+\Big|\frac{2\eta_1}{(c\nu_0)(x-(\xi_1+u)(t-s))}\Big|^2} + \Big| \frac{2\eta_1}{(c\nu_0)(x-(\xi_1+u)(t-s))}\Big| \Big). \label{ln_bdd_eta}
\end{align}
Similarly
\begin{align}
    &\int_0^{1} \frac{1}{\sqrt{|\xi'_1|^2 + (c\nu_0)^2 (x-(\xi_1+u)(t-s))^2  }} \dd \xi_1' \notag\\
    & = \ln \Big(\sqrt{1+\Big|\frac{1}{(c\nu_0)(x-(\xi_1+u)(t-s))}\Big|^2} + \Big| \frac{1}{(c\nu_0)(x-(\xi_1+u)(t-s))}\Big| \Big) \label{ln_bdd_1}.
\end{align}

\textbf{Step 2: integral over $\dd s$}.

Then we compute the $\dd s$ integral:
\begin{align*}
    &  \int_{t-T}^t \dd s e^{-\nu(\xi)(t-s)/2} \eqref{xi_integral} \lesssim \int_{t-T}^t \dd s e^{-\nu(\xi)(t-s)/2} [1+\eqref{ln_bdd_eta}+\eqref{ln_bdd_1}]. 
\end{align*}
The contribution of the constant is bounded as
\begin{align*}
    & \int_{t-T}^t \dd s e^{-\nu(\xi)(t-s)/2} \lesssim \min\{1,T\} \lesssim \frac{\min\{1,T\}}{\alpha(x,\xi)},
\end{align*}
where we have used~\eqref{alpha_bdd}.

Then we consider the contribution of~\eqref{ln_bdd_eta} and~\eqref{ln_bdd_1}. We split the discussion into $\xi_1+u<0$ and $\xi_1+u>0$.

\textbf{Step 2-1: case of $\xi_1+u<0$.} We have $x\leq x-(\xi_1+u)(t-s)$.

\textit{Contribution of~\eqref{ln_bdd_eta}.}

If $|\eta_1|/(c\nu_0 x) \leq 1$, then $|\eta_1|/[c\nu_0(x-(\xi_1+u)(t-s))]\leq 1$ for all $t-T\leq s\leq t$, this leads to \eqref{ln_bdd_eta} $\lesssim 1$ and thus
\begin{align*}
    & \int_{t-T}^t \dd s \mathbf{1}_{|\eta_1|/(c\nu_0x)\leq 1}e^{-\nu(\xi)(t-s)/2} \eqref{ln_bdd_eta} \lesssim \min\{1,T\} \leq \frac{\min\{1,T\}}{\alpha(x,\xi)}.
\end{align*}
Then we consider $|\eta_1|/(c\nu_0 x)>1$. From~\eqref{eta}, we have $\eta_1 = \xi_1 + u$. Then
\begin{align*}
  \alpha(x,\xi) &\leq \sqrt{(\xi_1+u)^2 + (c\nu_0)^2 x^2} \leq |\xi_1+u| + c\nu_0 x < 2|\xi_1+u| .
\end{align*}
Thus we have
\begin{align}
   \frac{1}{|\xi_1+u|} & \leq \frac{2}{\alpha(x,\xi)}. \label{xi_bdd}
\end{align}
We bound \eqref{ln_bdd_eta} by
\begin{align}
  \eqref{ln_bdd_eta}  &  \lesssim  1+ \ln\Big| \frac{\eta_1}{(c\nu_0)(x-(\xi_1+u)(t-s))}\Big|        \notag\\
  & \lesssim 1+ |\ln |\eta_1|| + |\ln |x-(\xi_1+u)(t-s)||. \label{ln_eta_upp_bdd}
\end{align}

The contribution of $|\ln|\eta_1||$ can be bounded as
\begin{align}
   \int_{t-T}^t e^{-\nu(\xi)(t-s)/2} |\ln |\eta_1|| \dd s  &\lesssim \min\{1,T\}\frac{|\ln |\eta_1||}{[1+|\xi|]} = \mathbf{1}_{|\xi_1+u|\geq 1} + \mathbf{1}_{|\xi_1+u|<1} \notag\\
    & \lesssim \min\{1,T\}\frac{|\ln (2|\xi|+2)|}{[1+|\xi|]} +  \frac{\min\{1,T\}}{|\xi_1+u|} \lesssim \frac{\min\{1,T\}}{\alpha(x,\xi)}. \label{ln_eta}
\end{align}
Here for $|\xi_1+u|<1$ we used $\ln |\xi_1+u| \leq \frac{1}{|\xi_1+u|}$. For the other case we used $|u|\ll 1$ and $|\xi_1+u|\geq 1$ to derive $|\xi_1+u|<2|\xi_1|$, and thus 
\begin{align*}
    &  \ln(|\xi_1+u|)\leq \ln(2|\xi_1|) \leq \ln(2|\xi|+2).
\end{align*}

Then we compute the contribution of the last term of~\eqref{ln_eta_upp_bdd}, which reads
\begin{align}
    & \int_{t-T}^t e^{-\nu(\xi)(t-s)/2} |\ln |x-(\xi_1+u)(t-s)|| \dd s \notag\\
    & = \int_{t-T}^t \mathbf{1}_{x-(\xi_1+u)(t-s)\geq 1} + \mathbf{1}_{x-(\xi_1+u)(t-s) < 1} \dd s \notag\\
    & \leq \int_{t-T}^t \mathbf{1}_{x-(\xi_1+u)(t-s)\geq 1}e^{-\nu(\xi)(t-s)/2} \ln ( -(\xi_1+u)/(c\nu_0) - (\xi_1+u)(t-s) ) \dd s \notag\\
    & + \int_{t-T}^t \mathbf{1}_{x-(\xi_1+u)(t-s)< 1} e^{-\nu(\xi)(t-s)/2} |\ln(-(\xi_1+u)(t-s))|\dd s \notag\\
    & \lesssim \int_{t-T}^t e^{-\nu(\xi)(t-s)/2} [|\ln |\xi_1+u||  + |\ln (t-s)| +|\ln (t-s+\frac{1}{c\nu_0} ) | ] \dd s \label{three_terms}\\
    &  \lesssim \min\{1,T\}\frac{|\ln |\xi_1+u||}{[1+|\xi|]} + \frac{\min\{1,|T|+|T\ln(T)|\}}{\alpha(x,\xi)} \lesssim \frac{\min\{1,|T|+|T\ln(T)|\}}{\alpha(x,\xi)}. \notag
\end{align}
In the third line, we used $x< \frac{|\eta_1|}{c\nu_0} = -\frac{(\xi_1+u)}{c\nu_0}$ for $\xi_1+u<0$. In the last line, we applied the same computation~\eqref{ln_eta} for the first term in~\eqref{three_terms}. For the rest terms in~\eqref{three_terms}, we used the following computation:
\begin{align}
  & \int_{t-T}^t e^{-\nu(\xi)(t-s)/2} [|\ln(t-s)| + |\ln(t-s+\frac{1}{c\nu_0})|]\dd s   = \int_{t-T}^t \mathbf{1}_{t-s\geq 1} + \mathbf{1}_{t-s\leq 1} \dd s \notag\\
  &  \lesssim  \mathbf{1}_{T\geq 1}\int_{t-T}^{t-1}  e^{-\nu(\xi)(t-s)/2} |t-s+\frac{1}{c\nu_0}| \dd s +\int_0^{\min\{1,T\}}  |\ln(s)| + |\ln(s+\frac{1}{c\nu_0})| \dd s  \notag\\
  & \lesssim \mathbf{1}_{T\geq 1} [1+ \frac{1}{c\nu_0}]+ \int_0^{\min\{1,T\}} |\ln(s)| \dd s   \notag \\
  &\lesssim \min\{1,T\} + \min\{1,|T|+|T\ln(T)|\} \notag\\
  &\lesssim  \min\{1,|T|+|T\ln(T)|\} \leq \frac{\min\{1,|T|+|T\ln(T)|\}}{\alpha(x,\xi)}. \label{integrate_s}
\end{align}
We conclude
\begin{align}
    &\int_{t-T}^t e^{-\nu(\xi)(t-s)/2} \eqref{ln_bdd_eta} \dd s\lesssim \frac{\min\{1,|T|+|T\ln(T)|\}}{\alpha(x,\xi)}. \label{step2_1_1_conclusion} 
\end{align}

\textit{Contribution of~\eqref{ln_bdd_1}.} 

If $1/(c\nu_0x)\leq 1$, we have \eqref{ln_bdd_1}$\lesssim 1$ and thus
\begin{align*}
    & \int_{t-T}^t \dd s e^{-\nu(\xi)(t-s)/2} 
  \mathbf{1}_{1/(c\nu_0x)\leq 1}\eqref{ln_bdd_1} \dd s \lesssim \frac{\min\{1,T\}}{\alpha(x,\xi)}.
\end{align*}
Then we suppose $c\nu_0 x \leq 1$. If $c\nu_0 x < |\eta_1|$, we bound \eqref{ln_bdd_1} similarly as \eqref{ln_eta_upp_bdd}:
\begin{align}
   \eqref{ln_bdd_1} & \lesssim  1 + |\ln(x-(\xi_1+u)(t-s))|. \label{ln_1_bdd}
\end{align}
We follow the same computation as in \eqref{three_terms} to conclude
\begin{align*}
    & \int_{t-T}^t \dd s e^{-\nu(\xi)(t-s)/2} 
\mathbf{1}_{1/(c\nu_0x)\leq 1,c\nu_0 x<|\eta_1|} \eqref{ln_bdd_1} \dd s\lesssim \frac{\min\{1,|T|+|T\ln(T)|\}}{\alpha(x,\xi)}.
\end{align*}
If $c\nu_0 x \geq |\eta_1|$, we have
\begin{align*}
   \alpha(x,\xi) & \leq \sqrt{|\xi_1+u|^2 + (c\nu_0)^2x^2} \leq 2c\nu_0 x,
\end{align*}
which leads to
\begin{align}
  \frac{1}{c\nu_0 x}  & \leq \frac{2}{\alpha(x,\xi)}. \label{x_bdd_alpha}
\end{align}

Then we use the bound \eqref{ln_1_bdd} and further compute
\begin{align*}
    &\int_{t-T}^t \dd s e^{-\nu(\xi)(t-s)/2} |\ln(x-(\xi_1+u)(t-s))|  = \int_{t-T}^t  [\mathbf{1}_{x-(\xi_1+u)(t-s)\leq 1} + \mathbf{1}_{x-(\xi_1+u)(t-s)\geq 1}]\\
    & \leq \int_{t-T}^t \mathbf{1}_{x-(\xi_1+u)(t-s)\leq 1} e^{-\nu(\xi)(t-s)/2}|\ln(x)| \dd s \\
    & +\int_{t-T}^t \mathbf{1}_{x-(\xi_1+u)(t-s)>1} e^{-\nu(\xi)(t-s)/2}\ln(x+c\nu_0 x(t-s)) \dd s\\
    &\lesssim |\ln(x)| \int_{t-T}^t e^{-\nu(\xi)(t-s)/2}[1+ \ln(1+c\nu_0(t-s))] \dd s\\
    &\lesssim |\ln(x)| \min\{1,|T|+|T\ln(T)|\}\lesssim \frac{\min\{1,|T|+|T\ln(T)|\}}{|x|} \\
    &\lesssim \frac{\min\{1,|T|+|T\ln(T)|\}}{\alpha(x,\xi)}.
\end{align*}
In the second last line, for the first inequality we applied the same computation~\eqref{integrate_s}, for the second inequality we used $x\leq \frac{1}{c\nu_0}$. In the last inequality, we used \eqref{x_bdd_alpha}.

We conclude
\begin{align}
    &\int_{t-T}^t e^{-\nu(\xi)(t-s)/2} \eqref{ln_bdd_1} \dd s \lesssim \frac{\min\{1,|T|+|T\ln(T)|\}}{\alpha(x,\xi)}. \label{step_2_1_2_conclusion}
\end{align}

\textbf{Step 2-2: case of $\xi_1+u>0$}. In such case $x-(\xi_1+u)(t-s)<x$. 

\textit{Contribution of~\eqref{ln_bdd_eta}.}

When $|\eta_1|/(c\nu_0 x)\geq 1$, we have $1/|\xi_1+u|\lesssim 1/\alpha(x,\xi)$. Then we use the same bound~\eqref{ln_eta_upp_bdd} for~\eqref{ln_bdd_eta}, where the contribution of the first two terms are independent of $x$, and thus can be bounded using the same computation in~\eqref{ln_eta}. We only need to compute the contribution of $|\ln(x-(\xi_1+u)(t-s))|$. We split the integral into $x-(\xi_1+u)(t-s)\geq 1$ and $x-(\xi_1+u)(t-s)\leq 1$. For the first case, we have
\begin{align}
    & \int_{t-T}^t \mathbf{1}_{x-(\xi_1+u)(t-s)\geq 1} e^{-\nu(\xi)(t-s)/2} |\ln(x-(\xi_1+u)(t-s)) | \dd s \notag\\
    & \leq \int_{t-T}^t  \mathbf{1}_{x-(\xi_1+u)(t-s)\geq 1} e^{-\nu(\xi)(t-s)/2} |\ln((\xi_1+u)/(c\nu_0)-(\xi_1+u)(t-s))|\dd s \notag\\
    &\lesssim \int_{t-T}^t e^{-\nu(\xi)(t-s)/2} |\ln(\xi_1+u)| \dd s + \int_{t-T}^t \mathbf{1}_{1/(c\nu_0)-(t-s)\geq 0} e^{-\nu(\xi)(t-s)/2} |\ln(1/(c\nu_0)-(t-s))|\dd s \notag\\
    & \lesssim \min\{1,T\}\frac{|\ln(\xi_1+u)|}{[1+|\xi|]} + \int_0^{\min\{1/(c\nu_0),T\}} |\ln(s)| \dd s \notag \\
    &\lesssim \min\{1,T\}\frac{|\ln(\xi_1+u)|}{[1+|\xi|]} + \min\{1,|T|+|T\ln(T)|\} \lesssim \frac{\min\{1,|T|+|T\ln(T)|\}}{\alpha(x,\xi)}. \label{ln_xi_1}
\end{align}
In the last inequality we used the same computation in~\eqref{ln_eta}.

For the second case $x-(\xi_1+u)(t-s)\leq 1$, without loss of generality, we assume $x-(\xi_1+u)T<1$, otherwise, the integration vanishes. We use a change of variable $y=x-(\xi_1+u)(t-s)$ with $\dd y = (\xi_1+u)\dd s$ and $x-(\xi_1+u)T\leq y\leq \min\{1,x\}$, then
\begin{align}
    &  \int_{t-T}^t \mathbf{1}_{x-(\xi_1+u
    )(t-s)\leq 1} e^{-\nu(\xi)(t-s)/2} |\ln(x-(\xi_1+u)(t-s))| \dd s \notag\\
    & \leq \int_{x-(\xi_1+u)T}^{\min\{1,x\}} e^{-\nu(\xi) \frac{x-y}{\xi_1+u}/2}|\ln(y)| \frac{1}{\xi_1+u} \dd y \leq \frac{1}{\xi_1+u}\int^{\min\{1,x\}}_{x-(\xi_1+u)T} |\ln(y)| \dd y \notag \\
    &\leq \frac{1}{\xi_1+u} \int_{0}^{\min\{1,(\xi_1+u)T\}}  |\ln(y)| \dd y   \label{third}\\
    &\lesssim  \frac{1}{\xi_1+u} \min\{1, |(\xi_1+u)T | + | (\xi_1+u)T \ln((\xi_1+u)T)|\}  \label{fourth}\\
    &\lesssim \frac{\min\{1,\sqrt{T}\}}{\alpha(x,\xi)}. \label{ln_xi_2}
\end{align}
In the third line, we use the fact that the integral domain is on $0<y\leq 1$, and the length of the integral domain is bounded by $\min\{1,x\} - [x-(\xi_1+u)T]\lesssim \min\{1,(\xi_1+u)T\}$. In the last line, in the case of $(\xi_1+u)T > 1$, we have $\frac{1}{\xi_1+u} < T$, and thus from \eqref{xi_bdd},
\begin{align*}
    & \frac{1}{\xi_1+u} \lesssim \min\{T,\frac{1}{\alpha(x,\xi)}\} \lesssim  \frac{\min\{1,T\}}{\alpha(x,\xi)}.
\end{align*}
In the case of $(\xi_1+u)T \leq 1$, for $T>1$, by \eqref{xi_bdd} we bound \eqref{third}$\lesssim \frac{1}{\xi_1+u} \lesssim \frac{\min\{1,T\}}{\alpha(x,\xi)}.$ For $T\leq 1$, from $(\xi_1+u)T\leq 1$ we have $|\ln((\xi_1+u)T)| \lesssim \frac{1}{\sqrt{(\xi_1+u)T}}$ and
\begin{align}
    &\frac{(\xi_1+u)T + |(\xi_1+u)T\ln((\xi_1+u)T)|}{\xi_1+u}  \notag\\
    & =T + |T\ln((\xi_1+u)T)| = \mathbf{1}_{(\xi_1+u)\geq 1} + \mathbf{1}_{(\xi_1+u)< 1} \notag\\
    & \lesssim T + \mathbf{1}_{\xi_1+u\geq 1} \frac{T}{\sqrt{(\xi_1+u)T}} + \mathbf{1}_{\xi_1+u<1} [T\ln(T)+T|\ln(\xi_1+u)|] \label{xi+u_T_leq_1}\\
    & \lesssim T + \sqrt{T} + |T\ln(T)| + \frac{T}{\xi_1+u} \lesssim \frac{\sqrt{T}}{\alpha(x,\xi)} \lesssim \frac{\min\{1,\sqrt{T}\}}{\alpha(x,\xi)}. \notag
\end{align}
In the last line, we used \eqref{xi_bdd} and $T\leq 1$.

Then we focus on the scenario that $|\eta_1|/(c\nu_0 x) \leq 1$, if $|\eta_1|\leq (c\nu_0)(x-(\xi_1+u)(t-s))$ for all $t-T\leq s\leq t$, then we have \eqref{ln_bdd_eta}$\lesssim 1$ and thus
\begin{align*}
  \int_{t-T}^t \mathbf{1}_{|\eta_1|\leq (c\nu_0)(x-(\xi_1+u)(t-s))} \eqref{ln_bdd_eta} \dd s  & \lesssim  \frac{\min\{1,T\}}{\alpha(x,\xi)}.
\end{align*}

If $|\eta_1|> (c\nu_0)(x-(\xi_1+u)(t-s))$ for some $t-T<s<t$, then there is a unique $t'$ such that $|\eta_1| = (c\nu_0)(x-(\xi_1+u)(t-t'))$. Denote $x' = x-(\xi_1+u)(t-t')$. Then 
\begin{align*}
    &   \frac{|\eta_1|}{(c\nu_0)(x'-(\xi_1+u)(t'-s))}>1 \text{ for all } t-T\leq s<t',
\end{align*}
\begin{align*}
    &  \frac{|\eta_1|}{(c\nu_0)(x-(\xi_1+u)(t-s))}\leq 1 \text{ for all } t'\leq s\leq t .
\end{align*}

By the observation above, we split the $s$-integral into two parts. The first part is bounded as
\begin{align*}
    &\int_{t'}^t e^{-\nu(\xi)(t-s)/2}\eqref{ln_bdd_eta}\dd s \lesssim \int_{t'}^t e^{-\nu(\xi)(t-s)/2} \dd s  \lesssim \frac{\min\{1,T\}}{\alpha(x,\xi)}.
\end{align*}
For the second part, we have $\frac{\eta_1}{(c\nu_0)x'}\geq 1$, and thus
\begin{align*}
    & \frac{1}{|\xi_1+u|}\lesssim \frac{1}{\alpha(x',\xi)}.
\end{align*}
Then the integral for the second part reads
\begin{align*}
    & \int_{t-T}^{t'}e^{-\nu(\xi)(t-s)/2}\eqref{ln_bdd_eta}\dd s  = e^{-\nu(\xi)(t-t')/2}\int_{t-T}^{t'} e^{-\nu(\xi)(t'-s)/2}\eqref{ln_bdd_eta} \dd s .
\end{align*}
Note that
\begin{align*}
   \eqref{ln_bdd_eta} & = \ln \Big(\sqrt{1+\Big|\frac{2\eta_1}{(c\nu_0)(x'-(\xi_1+u)(t'-s))}\Big|^2} + \Big| \frac{2\eta_1}{(c\nu_0)(x'-(\xi_1+u)(t'-s))}\Big| \Big).
\end{align*}
Following the same computation as~\eqref{ln_xi_1}, \eqref{ln_xi_2} for the case $|\eta_1|/(c\nu_0 x)>1$, we replace $x$ by $x'$ and conclude that
\begin{align*}
    &\int_{t-T}^{t'} e^{-\nu(\xi)(t-s)/2} \eqref{ln_bdd_eta}\dd s \lesssim e^{-\nu(\xi)(t-t')/2}\frac{\min\{1,\sqrt{T}\}}{\alpha(x',\xi)} \\
    & \lesssim e^{-\nu_0(t-t')/2} e^{2c\nu_0(t-t')}\frac{\min\{1,\sqrt{T}\}}{\alpha(x,\xi)}\leq \frac{\min\{1,\sqrt{T}\}}{\alpha(x,\xi)}.
\end{align*}
In the second line, we have used Lemma \ref{lemma:velocity_alpha} with $2c = \frac{1}{4}$.

Combining with~\eqref{ln_xi_1} and~\eqref{ln_xi_2}, we conclude that
\begin{align}
    & \int_{t-T}^{t} e^{-\nu(\xi)(t-s)/2} \eqref{ln_bdd_eta}\dd s \lesssim \frac{\min\{1,\sqrt{T}\}}{\alpha(x,\xi)}. \label{step2_2_1_conclusion}
\end{align}

\textit{Contribution of~\eqref{ln_bdd_1}.} 

If $1/(x-(\xi_1+u)(t-s))\leq 1$ for all $s$, then \eqref{ln_bdd_1}$\lesssim 1$ so that the contribution of~\eqref{ln_bdd_1} is bounded by $\min\{1,T\}/\alpha(x,\xi)$. Thus we only consider the integral over $x-(\xi_1+u)(t-s)\leq 1$:
\begin{align}
    & \int_{t-T}^t \mathbf{1}_{x-(\xi_1+u)(t-s)\leq 1} e^{-\nu(\xi)(t-s)/2} |\ln(x-(\xi_1+u)(t-s))| \dd s \label{s_integral}.
\end{align}
In such a case, we have $x-(\xi_1+u)T < 1$. We apply a change of variable $y=x-(\xi_1+u)(t-s)$, with $x-(\xi_1+u)T\leq y\leq \min\{x,1\}$. Then
\begin{align}
  \eqref{s_integral}  &  \leq \int_{x-(\xi_1+u)T}^{\min\{x,1\}} e^{-\nu(\xi)\frac{x-y}{\xi_1+u}/2} |\ln(y)|\frac{1}{\xi_1+u} \dd y  .\label{cov_y}
\end{align}
If $x>2$, we apply~\eqref{fourth} to have
\begin{align*}
   \eqref{s_integral} & \leq e^{-\frac{\nu(\xi)}{2(\xi_1+u)}} \frac{1}{\xi_1+u} \int_{x-(\xi_1+u)T}^{\min\{x,1\}} |\ln(y)| \dd y   \\
   &\lesssim  e^{-\frac{\nu(\xi)}{2(\xi_1+u)}} \frac{1}{\xi_1+u} \int_0^{\min\{1,(\xi_1+u)T\}} |\ln(y)| \dd y  \\
   & \lesssim e^{-\frac{\nu(\xi)}{2(\xi_1+u)}} \frac{\min\{1,(\xi_1+u)T + |(\xi_1+u)T \ln ((\xi_1+u)T)|\}}{\xi_1+u}  .
\end{align*}
For $(\xi_1+u)T \geq  1$ we have $\frac{1}{\xi_1+u}\leq T$, from $\frac{e^{-\frac{\nu(\xi)}{2(\xi_1+u)}}}{\xi_1+u} \lesssim 1$, we further have
\begin{align*}
   \frac{e^{-\frac{\nu(\xi)}{2(\xi_1+u)}}}{\xi_1+u} & \lesssim \min\{1,e^{-\frac{\nu(\xi)}{2(\xi_1+u)}}T\}\lesssim \frac{\min\{1,T\}}{\alpha(x,\xi)}.
\end{align*}
For $(\xi_1+u)T< 1$ and $T>1$, from $\frac{e^{-\frac{\nu(\xi)}{2(\xi_1+u)}}}{\xi_1+u} \lesssim 1$ we have
\begin{align*}
   \eqref{s_integral}\mathbf{1}_{T>1} & \lesssim \frac{e^{-\frac{\nu(\xi)}{2(\xi_1+u)}}}{\xi_1+u} \mathbf{1}_{T>1} \lesssim \frac{\min\{1,T\}}{\alpha(x,\xi)}\mathbf{1}_{T>1}. 
\end{align*}
For $(\xi_1+u)T < 1$ and $T\leq 1$, from $\frac{e^{-\frac{\nu(\xi)}{2(\xi_1+u)}}}{\sqrt{\xi_1+u}} \lesssim 1$ we have
\begin{align*}
    & e^{-\frac{\nu(\xi)}{2(\xi_1+u)}} \frac{(\xi_1+u)T+|(\xi_1+u)T\ln((\xi_1+u)T)|}{\xi_1+u}  \\
    & \lesssim e^{-\frac{\nu(\xi)}{2(\xi_1+u)}} [T+\frac{T}{\sqrt{(\xi_1+u)T}}] \lesssim \sqrt{T} \lesssim \frac{\min\{1,\sqrt{T}\}}{\alpha(x,\xi)}.
\end{align*}

Then we consider $x\leq 2$. We further discuss two cases. The first case is $|\eta_1|/(c\nu_0 x)\leq 1/(2tc\nu_0)$, which implies $x>2t|\eta_1|$. Then
\begin{align*}
   \alpha(x,\xi) &\leq \sqrt{(\xi_1+u)^2 + (c\nu_0)^2 x^2} \leq (\xi_1+u) + c\nu_0 x  \\
   & \leq \frac{x}{2t} + c\nu_0 x \lesssim x,
\end{align*}
here we used $t\gg 1$. Thus we derive $\frac{1}{x}\lesssim 1/\alpha(x,\xi).$ 

Since $0\leq s\leq t$, we have 
\begin{align*}
    x-(\xi_1+u)(t-s)& = \frac{x}{2} + \frac{x}{2} -(\xi_1+u)(t-s)  \\
    & \geq \frac{x}{2} + t|\eta_1| - t|\eta_1| = \frac{x}{2}.
\end{align*}
Then we have
\begin{align*}
    & \int_{t-T}^t e^{-\nu(\xi)(t-s)/2} |\ln(x-(\xi_1+u)(t-s))| \dd s \\
    &\leq \min\{1,T\}\max\{|\ln(x)|,|\ln(x/2)|\}\lesssim \frac{\min\{1,T\}}{x} \leq \frac{\min\{1,T\}}{\alpha(x,\xi)},
\end{align*}
where we have used $x\leq 2$.

The other case is $|\eta_1|/(c\nu_0 x)> 1/(2t c\nu_0)$, which implies $2t|\eta_1|>x$. We have
\begin{align*}
   \alpha(x,\xi) &\leq (\xi_1+u) + (c\nu_0)x \leq (\xi_1+u)(1+2t c\nu_0)  ,
\end{align*}
which leads to
\begin{align*}
  \frac{1}{\xi_1+u}  & \leq \frac{1+2t c\nu_0}{\alpha(x,\xi)}. 
\end{align*}
In this case we apply~\eqref{cov_y} with~\eqref{third} and~\eqref{fourth} to have
\begin{align*}
   \eqref{cov_y} & \leq \eqref{third}\leq \frac{1}{\xi_1+u} \min\{1, |(\xi_1+u)T | + | (\xi_1+u)T \ln((\xi_1+u)T)|\} .
\end{align*}

When $(\xi_1+u)T > 1$, we have $\frac{1}{\xi_1+u}<T$ and
\begin{align*}
   \eqref{cov_y} & \lesssim \frac{1}{\xi_1+u}  \lesssim \min\{T,\frac{t}{\alpha(x,\xi)}\} \lesssim \frac{\min\{T,t\}}{\alpha(x,\xi)}  .
\end{align*}
When $(\xi_1+u)T \leq 1$, for $T>1$ we use~\eqref{third} to bound
\begin{align*}
  \eqref{cov_y}\mathbf{1}_{T>1}  & \lesssim \frac{\mathbf{1}_{T>1}}{\xi_1+u} \lesssim \frac{\mathbf{1}_{T>1} t}{\alpha(x,\xi)}.   
\end{align*}
For $T<1$, we follow the computation in~\eqref{xi+u_T_leq_1} to have
\begin{align*}
    &\eqref{cov_y} \lesssim T+\sqrt{T} + |T\ln(T)| + T \mathbf{1}_{\xi_1+u\leq 1}|\ln(\xi_1+u)| \\
    & \lesssim \sqrt{T} + T \mathbf{1}_{\frac{1}{\xi_1+u}\geq 1} \ln(\frac{1}{\xi_1+u}) \\
    &\lesssim \sqrt{T}  + T \ln(\frac{1+2tc\nu_0}{\alpha(x,\xi)}) \lesssim \sqrt{T} +T\ln(t) + T|\ln(\alpha(x,\xi))|  \\
   & \lesssim \frac{\sqrt{T} +T\ln(t) }{\alpha(x,\xi)}.
\end{align*}
In the last inequality we used $\alpha(x,\xi)\lesssim 1$ to have $|\ln(\alpha(x,\xi))|\lesssim \frac{1}{\alpha(x,\xi)}$.

We conclude that
\begin{align}
    & \int_{t-T}^t e^{-\nu(\xi)(t-s)/2} \eqref{ln_bdd_1} \dd s \lesssim \frac{\mathbf{1}_{T>1}t + \mathbf{1}_{T\leq 1}[\sqrt{T}+T\ln(t)]}{\alpha(x,\xi)}. \label{step2_2_2_conclusion}
\end{align}

\textbf{Step 3: conclusion}

In summary, in the case of $\xi_1+u<0$ in Step 2-1 and the contribution of~\eqref{ln_bdd_eta} in the case of $\xi_1+u>0$ in Step 2-2, we collect~\eqref{step2_1_1_conclusion}, \eqref{step_2_1_2_conclusion} and~\eqref{step2_2_1_conclusion} to have
\begin{align*}
    \int_{t-T}^t e^{-\nu(\xi)(t-s)/2} \eqref{ln_bdd_eta} + \mathbf{1}_{\xi_1+u<0} \eqref{ln_bdd_1}  \dd s&  \lesssim \frac{\min\{1,\sqrt{T}\}}{\alpha(x,\xi)},
\end{align*}
which satisfies~\eqref{nln_large} for $T>1$ and \eqref{nln_small} for $T\leq 1$.

For the contribution of~\eqref{ln_bdd_1} in the case of $\xi_1+u>0$ in Step 2-2, we use~\eqref{step2_2_2_conclusion} to have
\begin{align*}
 \int_{t-T}^t  \mathbf{1}_{\xi_1+u>0} \eqref{ln_bdd_1}  \dd s&  \lesssim \frac{\mathbf{1}_{T>1}t + \mathbf{1}_{T\leq 1}[\sqrt{T}+T\ln(t)]}{\alpha(x,\xi)},
\end{align*}
which also satisfies~\eqref{nln_large} and~\eqref{nln_small}.

\end{proof}

The proof of Lemma \ref{lemma:NLN} directly implies the following result:
\begin{lemma}\label{lemma:NLN_inner}
\begin{equation}\label{NLN_inner}
\int^t_0 \dd s e^{-\nu(\xi)(t-s)} \int_{\mathbb{R}^3} \dd \xi' \frac{e^{-C|\xi'|^2}}{\alpha(x-(\xi_1+u)(t-s),\xi')} \lesssim \frac{t}{\alpha(x,\xi)},    
\end{equation}

\begin{equation}\label{NLN_two}
\int^t_0 \dd s e^{-\nu(\xi)(t-s)}\int_{\mathbb{R}^3} \dd \xi' w^{-1}(\xi') \int_{\mathbb{R}^3} \dd \xi'' \frac{e^{-C|\xi'-\xi''|^2}}{|\xi'-\xi''|\alpha(x-(\xi_1+u)(t-s),\xi'')} \dd \xi'' \lesssim \frac{t}{\alpha(x,\xi)}.
\end{equation}

\end{lemma}

\begin{proof}

\textit{Proof of \eqref{NLN_inner}.} The integral over $\dd \xi'$ is bounded as
\begin{align*}
    &  1 + \int_0^1 \frac{1}{\sqrt{|\xi'_1|^2 + (c\nu_0)^2(x-(\xi_1+u)(t-s))^2}} \dd \xi_1' = 1 + \eqref{ln_bdd_1}.
\end{align*}
Here, we applied the same computation as \eqref{xi_integral}. Then follow Step 2 in the proof of Lemma \ref{lemma:NLN} to conclude the lemma.

\textit{Proof of \eqref{NLN_inner}.} Again following the computation of \eqref{xi_integral}, with $\eta = \xi'+(u,0,0)$, the $\dd \xi' \dd \xi''$ integral is bounded as
\begin{align*}
    & \int_{\mathbb{R}^3} e^{-\theta |\eta|^2} \dd \eta \int_{\mathbb{R}} \frac{e^{-C|\eta_1-\xi_1''|^2}}{\sqrt{|\xi_1''|^2 + (c\nu_0)^2 (x-(\xi_1+u)(t-s))^2}} \dd \xi_1'' \\
    & \lesssim \int_{\mathbb{R}} \frac{e^{-C_1|\xi_1''|^2}}{\sqrt{|\xi_1''|^2 + (c\nu_0)^2 (x-(\xi_1+u)(t-s))^2}}  \dd \xi_1'' \lesssim 1 + \eqref{ln_bdd_1}.
\end{align*}
Here we choose $C_1 = \frac{\e}{2}C$ for some $\e > 0$ such that $\theta > \e C$ and 
\begin{align*}
- \theta |\eta_1|^2 - C|\eta_1-\xi''_1|^2   & \leq -\theta |\eta_1|^2 -\e C|\eta_1-\xi''_1|^2 \leq (-\theta-\e C) |\eta_1|^2  + 2\e C \eta_1 \xi'' - \e C |\xi''|^2 \\
 & \leq (\theta - \e C + 2\e C) - (\e C - \frac{\e C}{2}) |\xi''|^2.
\end{align*}

\end{proof}

\hide

The next lemma is a direct consequence of Lemma \ref{lemma:NLN}:
\begin{lemma}\label{lemma:p_x_K_bdd}
\begin{align*}
 \alpha(x,\xi) \int^t_{t-T} \dd s e^{-\nu(\xi)(t-s)/2}  \p_x K f(x,\xi)  & \lesssim \Vert \alpha \p_x f\Vert_{L^\infty_{x,\xi}}.   
\end{align*}
\end{lemma}

\unhide

\ \\

\section{Continuity and exponential decay in $\xi$.}\label{sec:continuity}
In this section, we conclude justify the assumption \eqref{assume_soln} and conclude Theorem \ref{thm:continuity}.

We start from proving the continuity and $w$-weighted $L^\infty$ estimate of the linearized penalized problem \eqref{pen_prob} in Proposition \ref{prop:pen_continuous}, then we will move onto the nonlinear penalized problem \eqref{pen_nonlin_prob} in Proposition \ref{prop:nonlinear_continuous}. At the end of this section, we will make use of the penalized problem to recover the solution the boundary layer problem \eqref{equation_f}.

To define the penalized problem \eqref{pen_prob}, we denote $\prod_+$ as the orthogonal projection on $X_+$:
\begin{equation*}
\prod_+ g = \langle gX_+ \rangle X_+.
\end{equation*}
Here, $X_+$ is defined in \eqref{basis}, and the inner product is defined in \eqref{inner_product}.

Denote $\phi_u$ to be the eigen-function of the following eigen-value problem:
\begin{equation}\label{eigen_prob}
\begin{cases}
& \mathcal{L} \phi_u = \tau_u (\xi_1 + u)\phi_u ,   \\
&\langle (\xi_1+u)\phi_u^2 \rangle = -u.
\end{cases}
\end{equation}

\ \\

\subsection{Continuity and $w$-weighted $L^\infty$ estimate of the eigen-value problem \eqref{eigen_prob}.}

For $u\neq 0$, we define 
\begin{equation}\label{p_g}
\psi_u := \frac{\phi_u - \phi_0}{u}.
\end{equation}

For $u$ near $0$, it was shown in \cite{golse} that there exists solution $\phi_u \in L^2_\xi \cap L^{\infty,s}_\xi$ to \eqref{eigen_prob}:
\begin{proposition}[\textbf{Proposition 3.1 in \cite{golse}}]\label{prop:eigen_exists}
There exists $r>0$ and real analytic function $u\to \tau_u\in \mathbb{R}^3$ with $|\tau_u|\sim u$ and real analytic map $u\to \phi_u  \in \mathcal{H}\cap dom(\mathcal{L})$ with $0<|u|<r$ such that $\phi_u$ satisfies \eqref{eigen_prob}. Furthermore, there exists a positive constant $C_s$ such that for each $s\geq 0$, $\phi_u$ satisfies
\begin{equation}\label{phi_u_linfty}
\Vert (1+|\xi|)^s \phi_u\Vert_{\lftyv} \leq C_s
\end{equation}
for all $s\geq 0$ uniformly in $u\in (-r,0)\cup (0,r)$.

\end{proposition}
 
\begin{remark}
In \cite{golse}, the linearization reads $F = M+Mf$, and their statement regarding \eqref{eigen_prob} reads as following. The eigen-function of the following eigen-value problem 
\begin{equation*}
\begin{cases}
&  \frac{-Q(M\phi_u,M) - Q(M,M\phi_u)}{M} = \tau_u (\xi_1 + u)\phi_u    \\
&\int_{\mathbb{R}^3} (\xi_1+u)\phi_u^2 M \dd \xi = -u
\end{cases}
\end{equation*}
satisfies
\begin{align*}
    & \Vert (1+|\xi|)^s \phi_u \sqrt{M}\Vert_{L^\infty_\xi} \leq C_s.
\end{align*}
By substituting $\phi_u \sqrt{M}$ by $\phi_u$, this statement is exactly Proposition \ref{prop:eigen_exists}.

\end{remark}

First we show that the eigen-function in \eqref{eigen_prob} is continuous and decays exponentially in $\xi\in \mathbb{R}^3$.

\begin{lemma}\label{lemma:eigen_continuous}
For the eigen-function $\phi_u$ in Proposition \ref{prop:eigen_exists}, we have
\begin{equation}\label{eigen_continuous}
\phi_u \in C(\mathbb{R}^3).
\end{equation}
Recall the exponential weight in \eqref{w_weight}, the eigen-function in \eqref{eigen_prob} further satisfies
\begin{equation}\label{eigen_expo}
\Vert w \phi_u \Vert_{L^\infty_{\xi}} < \infty.
\end{equation}

\end{lemma}

\begin{proof}
\hide
In the proof of Proposition 3.1 in \cite{golse}, the eigen-function $\phi_u$ was finally chosen to be
\begin{align*}
    \phi_u := \frac{\phi^0_{z(u)}}{\sqrt{-\langle (\xi_1+u) (\phi^0_{z(u)})^2\rangle / u}},
\end{align*}
where $\phi^0_{z(u)}$ satisfies
\begin{align}
    \phi^0_{z(u)} = \frac{1}{\nu - z(u)(\xi_1+u)} K \phi^0_{z(u)}. \label{phi_0_eqn}
\end{align}
Here $z(u)$ a holomorphic map and satisfies
\begin{align*}
    & \nu - z(u)(\xi_1+u)\geq \frac{1}{2}\nu_- (1+|\xi|)>0.
\end{align*}

\begin{align*}
    &   \Vert \phi^0_{z(u)} \Vert_{\mathcal{H}} = 1, \ \ \Vert \sqrt{M}\phi_{z(u)}^0\Vert_{L^{\infty,s}} \leq C_s.
\end{align*}
To prove the lemma, it suffices to justify the continuity of $\phi^0_{z(u)}$. We use contradiction argument to show that $\phi^0_{z(u)}$ in \eqref{phi_0_eqn} is continuous. 
\unhide

\textit{Proof of \eqref{eigen_continuous}}. From \eqref{eigen_prob}, the solution $\phi_u$ satisfies
\begin{align}
  \phi_u  &   =  \frac{1}{\nu(\xi) - \tau_u (\xi_1+u)} K\phi_u  . \label{phi_K}
\end{align}
Since $|u|<r$ for some small $r$, on RHS we have $\nu(\xi) - \tau_u (\xi_1+u)> \frac{\nu(\xi)}{2}$. We use the contradiction argument to show that $\phi_u$ is continuous.

Suppose $\phi_u$ is not continuous at some $\xi = \xi^0$, then at $\xi=\xi^0$, RHS of \eqref{phi_K} reads
\begin{align*}
    & \frac{1}{\nu(\xi^0) - \tau_u(\xi^0_1 + u)} \int_{\xi'\in\mathbb{R}^3} \mathbf{k}(\xi^0,\xi') \phi_u(\xi') \dd \xi'.
\end{align*}
The first term is continuous in $\xi^0$ due to $\tau_u \ll 1$. For the second term, note that from \eqref{phi_u_linfty}, $\phi_u \in L^\infty_\xi$, the second term is differentiable at $\xi^0$ from \eqref{nabla_k_theta}:
\begin{align*}
    & \Big| \int_{\xi'\in \mathbb{R}^3} \nabla_{\xi} \mathbf{k}(\xi^0, \xi') \phi_u (\xi') \dd \xi' \Big|\lesssim  [1+ \xi_0]^2 \int_{\mathbb{R}^3}     \frac{e^{-C|\xi'-\xi^0|^2}}{|\xi'-\xi^0|^2} \dd \xi' <\infty.
\end{align*}
Since both terms are continuous, we conclude \eqref{eigen_continuous} by contradiction.

\textit{Proof of \eqref{eigen_expo}}.  For $\theta<\frac{1}{4}$, we consider a variant of perturbation around the Maxwellian as
\begin{align*}
    & F(x,\xi) = M + \sqrt{M} e^{-\theta |\xi|^2} f.
\end{align*}
Then we denote the corresponding linear Boltzmann operator as
\begin{align}
  \mathcal{L}_\theta f  & = -\frac{Q(M,\sqrt{M}e^{-\theta |\xi|^2} f)+Q(\sqrt{M}e^{-\theta |\xi|^2} f, M)}{\sqrt{M}e^{-\theta |\xi|^2}}   
  = \nu(\xi) f - K_\theta f.  \label{L_1}
\end{align}
With the extra weight $e^{\theta|\xi|^2}$ and $\theta < \frac{1}{4}$, $K_\theta f$ can be expressed as
\begin{align*}
   & K_\theta f(\xi) = \int_{\mathbb{R}^3} \mathbf{k}_\theta(\xi,\xi') f(\xi') \dd \xi' = \int_{\mathbb{R}^3}  \mathbf{k}(\xi,\xi') \frac{e^{\theta|\xi|^2}}{e^{\theta |\xi'|^2}} f(\xi') \dd \xi'  , 
\end{align*}
here $\mathbf{k}_\theta(\xi,\xi')$ is given in Lemma \ref{lemma:k_theta}, which has the form in \eqref{k1}, \eqref{k2} with different coefficients in the exponent( due to \eqref{k_theta} and see Lemma 3 in \cite{G} for the proof).

Therefore, $\mathcal{L}_\theta$ is still self-adjoint, and $K_\theta$ is still a bounded operator from $L^2_{\xi}$ to $L^{\infty,1/2}_\xi$, and from $L^{\infty,s}_\xi$ to $L^{\infty,s+1}_\xi$. Here
\[L^{\infty,s}(\mathbb{R}^3): = \{\phi \in L^\infty(\mathbb{R}^3)|(1+|\xi|)^s \phi \in L^\infty(\mathbb{R}^3)\}.\]
The kernel of $\mathcal{L}_\theta$ is given by $(1,\xi_i,|\xi|^2)\sqrt{M}e^{\theta |\xi|^2} \in L^2_{\xi}$. Then applying the same argument in Proposition 3.1 in \cite{golse}, for $|u|\ll 1$, there exists $\phi_\theta \in L^2_{\xi} \cap L^{\infty,s}_\xi$ for any $s>0$ to the following eigen-value problem
\begin{equation*}
\begin{cases}
    & \mathcal{L}_\theta \phi_\theta = \tau_u (\xi_1+u) \phi_\theta  ,\\
    &  \int_{\mathbb{R}^3} (\xi_1+u) |\phi_\theta|^2   \dd \xi = -u.
\end{cases}
\end{equation*}
Then for some constant $C=C(\theta,u)$, $\phi = e^{-\theta|\xi|^2} \phi_\theta$ is an eigen-function of the following problem
\begin{equation*}
\begin{cases}
    &  \mathcal{L} \phi = \tau_u (\xi_1+u) \phi, \\
    &  \langle (\xi_1+u) |\phi|^2\rangle = C(\theta,u).
\end{cases}
\end{equation*}
Through rescaling, we conclude that the eigen-function $\phi_u$ in \eqref{eigen_prob} satisfies \eqref{eigen_expo}.

\hide

We first prove the a priori estimate. We rearrange \eqref{eigen_prob} into
\begin{align*}
    &    [\nu  - \tau_u(\xi_1+u)] \phi_u  = K\phi_u.    
\end{align*}
We split the cases into $\mathbf{1}_{|\xi_1|\geq N} + \mathbf{1}_{|\xi_1| < N}$. Then for $\tau_u \ll \nu$ we have
\begin{align*}
  \Big\Vert \frac{\nu}{2} w\phi_u  \Big\Vert_{L^\infty_{|\xi|>N}}  &  \leq  \int_{\mathbb{R}^3} \mathbf{k}(\xi,\xi') \frac{w(\xi)}{w(\xi')} w(\xi') \phi_u(\xi') \dd \xi' \leq \frac{\Vert w\phi_u\Vert_{L^\infty_{\xi}}}{N}.
\end{align*}
On the other hand, we have
\begin{align*}
  \Big\Vert \frac{\nu}{2} w \phi_u \Big\Vert_{L^\infty_{|\xi|<N}}  &  \leq C_N \Vert \phi_u\Vert_{L^\infty_{\xi}}.
\end{align*}
Hence we conclude that
\begin{align*}
  \Vert w\phi_u\Vert_{L^\infty_{\xi}}  &  \leq C_N + \frac{1}{N} \Vert w\phi_u\Vert_{L^\infty_{\xi}}.  
\end{align*}

\unhide
    
\end{proof}

In the following lemma, we bound $\psi_u$ in the $w$-weighted $L^\infty$ norm. Here we recall the definition of $\psi_u$ in \eqref{p_g}.
\begin{lemma}\label{Lemma: w_psi_infty}
For the eigen-function $\phi_u$ in \eqref{eigen_prob} in Proposition \ref{prop:eigen_exists}, we further have
\begin{equation}\label{w_psi_bdd}
\Vert w\psi_u \Vert_{L^\infty_\xi} = \big\Vert w \frac{\phi_u - \phi_0}{u} \big\Vert_{L^\infty_\xi} \lesssim 1,
\end{equation}
where the inequality in \eqref{w_psi_bdd} does not depend on $u$.
\end{lemma}

\begin{proof}
Following the proof of Proposition 3.1 in \cite{golse}, for $z$ near $0$, there exists $\phi_z$ and $\lambda(z) = u(z)z$ which are analytic in $z$ such that
\begin{align*}
    & T(z)\phi(z):= (\mathcal{L}-z\xi_1) \phi(z) = \lambda (z) \phi(z).
\end{align*}
Denoting $\dot{\phi}_z$ as the derivative with respect to $z$, then we have
\begin{align*}
    &  \mathcal{L} \dot{\phi}_z - z\xi_1 \dot{\phi}_z - \xi_1 \phi_z = \lambda (z) \dot{\phi}_z + \dot{\lambda}(z) \phi_z,
\end{align*}
which is equivalent to
\begin{align*}
    &    [\nu(\xi) - z(\xi_1+u)] \dot{\phi}_z = K \dot{\phi}_z + (\dot{\lambda}(z) + \xi_1) \phi_z.
\end{align*}
Since $\phi_z$ is analytic in $z$ (also see Proposition \ref{prop:eigen_exists}), we have $\Vert \dot{\phi}_z\Vert_{L^2_\xi}\lesssim 1$. Combining with $\Vert [1+\xi]\phi_z \Vert_{L^\infty_\xi} \lesssim 1$ from Proposition \ref{prop:eigen_exists}, and the fact that $K$ is bounded from $L^2_\xi$ to $L^{\infty,1/2}_\xi$, we conclude that
\begin{align}
    & \Vert \dot{\phi}_z\Vert_{L^\infty_\xi} \lesssim \big\Vert \frac{K\dot{\phi}_z}{\nu(\xi)} \big\Vert_{L^\infty_\xi} + \Vert [1+\xi] \phi_z \Vert_{L^\infty_\xi} \lesssim 1. \label{dot_phi_z_bdd}
\end{align}
From the Taylor's theorem, $\frac{\phi_u-\phi_0}{u} = \dot{\phi}_{\bar{u}}$ for some $\bar{u}\in [0,u]$.

To prove \eqref{w_psi_bdd}, we apply the same argument in the proof of Lemma \ref{lemma:eigen_continuous}. With the modified linear operator $\mathcal{L}_\theta$ defined in \eqref{L_1}, we apply the same computation of \eqref{dot_phi_z_bdd} to the eigen-value problem $[\mathcal{L}_\theta-z\xi_1] \phi_\theta(z) = \lambda(z)\phi_\theta(z)$ and deduce that
\begin{align*}
    &   \Vert \dot{\phi}_\theta(z)\Vert_{L^\infty_\xi} \lesssim 1.
\end{align*}
Then we conclude \eqref{w_psi_bdd} from the Taylor's theorem and the fact that $\dot{\phi}(z) = e^{-\theta |\xi|^2} \dot{\phi}_\theta(z) $.

\end{proof}

\ \\

\subsection{Continuity and $w$-weighted $L^\infty$ estimate of the linearized penalized problem}

With $\psi_u$ defined in \eqref{p_g}, we define
\begin{equation*}
\mathbf{p}_u g =  -\langle (\xi_1+u)\psi_u g \rangle \phi_u , \ \ \mathbf{P}_u g = -\langle \psi_u g\rangle (\xi_1+u) \phi_u.
\end{equation*}
Now we define the linearized penalized problem in the following proposition.

\begin{proposition}[Proposition 5.3 and Proposition 5.6 in \cite{golse}]\label{prop:pen_well_pose}
Define the linearized penalized collision operator as
\begin{equation}\label{pen_lin_op}
\mathcal{L}^p g : = \mathcal{L}g + 2\gamma \prod_+((\xi_1+u)g) + 2\gamma \mathbf{p}_u g - \gamma (\xi_1 + u)g.
\end{equation}
Here $\gamma$ is a small constant $\gamma\ll 1$.

Let $Q\in L^2_{x,\xi}$. There exists a unique solution $g$ to the following linearized penalized problem
\begin{equation}\label{pen_prob}
\begin{cases}
     &  (\xi_1+u)\p_x g + \mathcal{L}^p g = Q, \ x>0, \ \xi\in \mathbb{R}^3 \\
     & g(0,\xi) = g_b(\xi), \ \xi_1 + u >0.
\end{cases}
\end{equation}
Moreover, this solution satisfies the $L^2$-estimate
\begin{equation}\label{l2_estimate}
\Vert \nu g\Vert_{L^2_{x,\xi}} \leq C_2\big[\Vert Q\Vert_{L^2_{x,\xi}} + \Vert \nu g_b\Vert_{L^2_{\xi}} \big].
\end{equation}

If we further assume $Q$ and $g_b$ satisfy
\begin{align*}
    & (1+|\xi|)^3  g_b \in L^\infty_{\xi}, \ \ (1+|\xi|)^2  Q \in L^\infty_{x,\xi}, \ \ Q(x,\cdot )\perp \ker \mathcal{L},
\end{align*}
then the solution further satisfies the $L^\infty$-estimate
\begin{align}
    & \Vert (1+|\xi|)^3  g\Vert_{L^\infty_{x,\xi}}  \leq C \big[\Vert (1+|\xi|)^3  g_b\Vert_{L^\infty_{\xi}} + \Vert (1+|\xi|)^2  Q \Vert_{L^\infty_{x,\xi}} + \Vert g\Vert_{L^2_{x,\xi}} \big].  \label{linfty_estimate}
\end{align}

\end{proposition}
\begin{remark}
Here we note that the $L^\infty$ estimate in~\eqref{linfty_estimate} is controlled by the $L^2$ estimate $\Vert g\Vert_{L^2_{x,\xi}}$, and thus it can be further bounded using \eqref{l2_estimate}:
\begin{align}
    & \Vert (1+|\xi|)^3  g\Vert_{L^\infty_{x,\xi}}  \leq C \big[\Vert (1+|\xi|)^3  g_b\Vert_{L^\infty_{\xi}} + \Vert e^{\delta x}(1+|\xi|)^2  Q \Vert_{L^\infty_{x,\xi}}  \big]. \notag
\end{align}

\end{remark}

Denote 
\begin{align}
  -K^p g & = -Kg   + 2\gamma \prod_+((\xi_1+u)g) + 2 \gamma \mathbf{p}_u g, \label{Kp}
\end{align}
so that $\mathcal{L}^p g = [\nu(\xi)-\gamma(\xi_1+u)] g - K^p g$. With the $w$-weighted estimate for the eigen-function problem in Lemma \ref{lemma:eigen_continuous}, we have the following property for $K^p$:
\begin{lemma}\label{lemma:Kp}
\begin{equation}\label{Kp_bdd_l2}
\Vert K^p g\Vert_{L^2_{x,\xi}} \lesssim \Vert g\Vert_{L^2_{x,\xi}}.
\end{equation}

\begin{equation}\label{Kp_bdd}
    \Vert w(\xi)K^p g(\xi) \Vert_{L^\infty_{x,\xi}} \lesssim \Vert wg\Vert_{L^\infty_{x,\xi}}.
\end{equation}
\end{lemma}

\begin{proof}
\textit{Proof of \eqref{Kp_bdd_l2}.} Clearly $\Vert Kf\Vert_{L^2_{x,\xi}} \lesssim \Vert f\Vert_{L^2_{x,\xi}}$. For the rest, the contribution of $\prod_{+}(\xi_1+u)g$ is bounded as
\begin{align*}
  \Vert \chi_+ \langle (\xi_1+u)g,\chi_+\rangle\Vert_{L^2_{x,\xi}}^2   & \lesssim  \iint_{\mathbb{R}^+ \times \mathbb{R}^3} e^{-\frac{|\xi|^2}{4}} \big(\int_{\mathbb{R}^3} |g(x,\xi')|^2 \dd \xi'\big) \big(\int_{\mathbb{R}^3} (\xi_1'+u)^2 e^{-\frac{|\xi'|^2}{4}} \dd \xi' \big)   \dd x \dd \xi \\
    & \lesssim  \iint_{\mathbb{R}^+ \times \mathbb{R}^3}  |g(x,\xi')|^2 \dd x \dd \xi' = \Vert g\Vert^2_{L^2_{x,\xi}}   .
\end{align*}
The contribution of $\mathbf{p}_u g$ is bounded as
\begin{align*}
    &  \Vert \phi_u \langle (\xi_1+u)\psi_u g\rangle\Vert_{L^2_{x,\xi}}^2 \notag \\
    &\lesssim \iint_{\mathbb{R}^+ \times \mathbb{R}^3} \Vert w\phi_u\Vert_{L^\infty_\xi}^2 w^{-2}(\xi) \big(\int_{\mathbb{R}^3} |g(x,\xi')|^2 \dd \xi'\big)  \Vert w\psi_u \Vert_{L^\infty_{x,\xi}}^2 \big(\int_{\mathbb{R}^3} (\xi'_1+u)^2 w^{-2}(\xi') \dd \xi'\big)   \dd x \dd \xi \\
    & \lesssim \iint_{\mathbb{R}^+ \times \mathbb{R}^3}  |g(x,\xi')|^2  \dd x \dd \xi'  \Vert g\Vert_{L^2_{x,\xi}}^2.
\end{align*}
Here we have applied H\"older inequality to $\langle (\xi_1+u)\psi_u g \rangle$ and used Lemma \ref{lemma:eigen_continuous} and Lemma \ref{Lemma: w_psi_infty}.

\textit{Proof of \eqref{Kp_bdd}.}
The contribution of $K$ in $K^p$ can be controlled using Lemma \ref{lemma:k_theta}:
\begin{align*}
    &\Vert wKg\Vert_{L^\infty_{x,\xi}} \leq \Vert wg\Vert_{L^\infty_{x,\xi}} \int_{\mathbb{R}^3} \mathbf{k}(\xi,\xi') \frac{w(\xi)}{w(\xi')} \dd \xi' \lesssim \Vert wg\Vert_{L^\infty_{x,\xi}}.
\end{align*}

The contribution of $\prod_{+}(\xi_1+u)g$ is bounded as
\begin{align}
    &   w(\xi)\chi_+ \langle (\xi_1+u)g,\chi_+\rangle \lesssim \Vert wg\Vert_{L^\infty_{x,\xi}} \langle (\xi_1+u)e^{-\theta |\xi|^2},\chi_+\rangle \lesssim \Vert wg\Vert_{L^\infty_{x,\xi}}. \label{prod_bdd}
\end{align}

The contribution of $\mathbf{p}_u g$ can be bounded using Proposition \ref{prop:eigen_exists} and Lemma \ref{lemma:eigen_continuous}:
\begin{align}
    &  |w(\xi)\phi_u(\xi) \langle (\xi_1+u)\psi_u g \rangle| \leq \Vert w\phi_u\Vert_{L^\infty_{x,\xi}}  \Vert wg\Vert_{L^\infty_{x,\xi}} |\langle (\xi_1+u) \psi_u e^{-\theta |\xi|^2}\rangle|  \notag\\
    & \lesssim \Vert w\phi_u\Vert_{L^\infty_{x,\xi}}  \Vert wg\Vert_{L^\infty_{x,\xi}}  \Vert \psi_u \Vert_{L^2_\xi}\lesssim   \Vert wg\Vert_{L^\infty_{x,\xi}}. \label{p_g_bdd}
\end{align}
Then we conclude the lemma.

\end{proof}

In the following proposition, we improve the result in Proposition \ref{prop:pen_well_pose} by establishing the continuity and $w$-weighted $L^\infty$ estimate.

\begin{proposition}\label{prop:pen_continuous}
Suppose all conditions in Proposition \ref{prop:pen_well_pose} are satisfied. We assume $g_b(\xi)\in C(\mathbb{R}^3)$, and $Q$ is continuous away from $\mathcal{D}$ defined in \eqref{grazing}, then the unique solution $g$ of~\eqref{pen_prob} is continuous away from $\mathcal{D}$. If we further assume $\Vert wg_b\Vert_{L^\infty_{\xi}}<\infty$ and $\big\Vert \frac{w}{[1+|\xi|]}Q \big\Vert_{L^\infty_{x,\xi}}<\infty$, the solution in Proposition \ref{prop:pen_well_pose} satisfies
\begin{equation}\label{weighted_lfty}
\Vert wg\Vert_{L^\infty_{x,\xi}} \leq C\big[ \Vert wg_b\Vert_{L^\infty_{\xi}} + \big\Vert \frac{w}{[1+|\xi|]}Q \big\Vert_{L^\infty_{x,\xi}} + \Vert g\Vert_{L^2_{x,\xi}} \big].
\end{equation}

\end{proposition}

\begin{remark}
By the $L^2$ bound \eqref{l2_estimate}, for some $\delta\ll 1$, \eqref{weighted_lfty} can be further bounded as
\begin{equation}\label{weighed_infty}
\Vert wg\Vert_{L^\infty_{x,\xi}} \leq C\big[\Vert wg_b\Vert_{L^\infty_\xi} + \big\Vert e^{\delta x} \frac{w}{[1+|\xi|]} Q\big\Vert_{L^\infty_{x,\xi}} \big].
\end{equation}
\end{remark}

\begin{proof}
We consider a variant of~\eqref{pen_prob}:
\begin{equation}\label{lambda_eqn}
\begin{cases}
    &(\xi_1+u)\p_x g + [\nu(\xi)-\gamma(\xi_1+u)] g - \lambda K^p g = Q \\
    &g(0,\xi) = g_b(\xi), \ \xi_1+u>0.
\end{cases}
\end{equation}
It is straightforward to apply the same argument in \cite{golse} to show that \eqref{lambda_eqn} is well-posed and satisfies both estimates \eqref{l2_estimate} and \eqref{linfty_estimate} uniformly in $\lambda$. To obtain the $L^\infty_\xi$ estimate with exponential weight $w(\xi)$, we will prove the a-priori estimate in Step 1, and in Step 2, we will use fixed point argument to justify that the solution $g$ indeed satisfies \eqref{weighted_lfty} and is continuous away from $\mathcal{D}$.

\textit{Step 1. A-priori estimate.} In this step we prove the following statement: suppose the solution to \eqref{lambda_eqn} satisfies $\Vert wg\Vert_{L^\infty_{x,\xi}}<\infty$, then we have
\begin{equation}\label{aprior_wg}
\Vert wg\Vert_{L^\infty_{x,\xi}} \lesssim \Vert wg_b\Vert_{L^\infty_\xi} + \big\Vert \frac{w}{[1+|\xi|]}Q \big\Vert_{L^\infty_{x,\xi}} + \Vert g\Vert_{L^2_{x,\xi}} .
\end{equation}
Here the inequality does not depend on $\lambda$.

We denote $G(\xi): = \Vert g(x,\xi)\Vert_{L^\infty_x}$. Applying the Duhamel's principle and taking $\sup$ in $x$, we have
\begin{align}
  G(\xi)  & \leq |g_b(\xi)| + \frac{|K^p G(\xi)|}{\nu(\xi)-\gamma |\xi_1+u|} + \frac{\Vert Q(x,\xi)\Vert_{L^\infty_x}}{\nu(\xi)-\gamma |\xi_1+u|}. \label{G_eqn}
\end{align}
To estimate $wG(\xi)$, we first estimate $w|K^p G|$. We apply Lemma \ref{lemma:k_theta} to compute the contribution of $K$ in $K^p$ as
\begin{align}
    &w(\xi) K G(\xi) = w(\xi) \int_{\mathbb{R}^3} \mathbf{k}(\xi,\xi') G(\xi') \dd \xi' \notag\\
    & = \int_{|\xi'|>N \text{ or }|\xi'-\xi|<\frac{1}{N}} \mathbf{k}_\theta(\xi,\xi') w(\xi')G(\xi') \dd \xi' + \int_{|\xi'|\leq N \text{ and }|\xi'-\xi| \geq \frac{1}{N}} \mathbf{k}_\theta(\xi,\xi') w(\xi')G(\xi') \dd \xi'  \notag\\
    & \leq o(1) \Vert wG\Vert_{L^\infty_{\xi}} + \int_{|\xi'|\leq N} G(\xi') \dd \xi' \lesssim o(1) \Vert wG\Vert_{L^\infty_{\xi}} + C_N \Vert G\Vert_{L^2_{\xi}}. \label{G_K}
\end{align}
In the last line, in the first inequality, we used $\mathbf{k}_\theta(\xi,\xi') w(\xi') \lesssim_N 1$ when $|\xi'|\leq N, \ |\xi'-\xi|\geq \frac{1}{N}$; when $|\xi'-\xi|<\frac{1}{N}$, we directly have $\int_{|\xi'-\xi|<\frac{1}{N}} \mathbf{k}_\theta(\xi,\xi')\dd \xi' \lesssim o(1)$; when $|\xi'|>N$, we further split the case into $|\xi|<\frac{N}{2}$ and $|\xi|>\frac{N}{2}$, for the first case we have $|\xi-\xi'|>\frac{N}{2}$, and thus $\int_{|\xi'-\xi|>\frac{N}{2}} \mathbf{k}_\theta(\xi,\xi')\dd \xi' \lesssim o(1)$; for the second case we have $\mathbf{1}_{|\xi|>\frac{N}{2}}\int_{\mathbb{R}^3} \mathbf{k}_\theta(\xi,\xi')\dd \xi' \lesssim o(1)$ from \eqref{k_theta}. In the last inequality, we have used the H\"older inequality in the bounded space $|\xi'|\leq N$.

The contribution of $\gamma \prod_+ ((\xi_1+u)g)$ and $\gamma\mathbf{p}_u g$ can be controlled using \eqref{prod_bdd} and \eqref{p_g_bdd}, with the extra constant $\gamma \ll 1$:
\begin{align}
    & \gamma  \Vert   wG\Vert_{L^\infty_\xi}  \leq o(1)\Vert wG\Vert_{L^\infty_\xi}. \label{G_lambda}
\end{align}

Combining \eqref{G_eqn}, \eqref{G_K} and \eqref{G_lambda}, we conclude that
\begin{align}
  \Vert wG\Vert_{L^\infty_\xi}  &  \lesssim \Vert wg_b\Vert_{L^\infty_\xi}+\Vert G\Vert_{L_\xi^2} + \big\Vert \frac{w}{[1+|\xi|]} Q \big\Vert_{L^\infty_{x,\xi}}   . \label{wg_first}
\end{align}

The $L^2(\dd \xi;L^\infty_x)$ estimate $\Vert G\Vert_{L^2_\xi}$ follows from (5.19) in \cite{golse}:
\begin{align*}
    & \Vert G\Vert_{L^2_\xi} \leq 2\Vert g_b\Vert_{L^2_\xi} + C\Vert Q\Vert_{L^2(\dd \xi;L^\infty_x)} + \Vert g\Vert_{L^2_{x,\xi}},
\end{align*}
where we can further control the RHS as
\begin{align}
    & \Vert G\Vert_{L^2_\xi} \lesssim  \Vert w g_b\Vert_{L^\infty_\xi} + \big\Vert \frac{w}{[1+|\xi|]} Q \big\Vert_{L^\infty_{x,\xi}} + \Vert g\Vert_{L^2_{x,\xi}}  . \label{l2linfty}
\end{align}

Combining \eqref{l2linfty} and \eqref{wg_first}, we conclude \eqref{aprior_wg}.

\textit{Step 2. Fixed point argument.}

For fixed boundary data $g_b(\xi)$, we denote $\mathcal{L}_\lambda^{-1}$ to be the solution operator associated with~\eqref{lambda_eqn}, i.e, the solution to \eqref{lambda_eqn} is $g = \mathcal{L}_{\lambda}^{-1}(Q)$.

We define a Banach space as
\begin{equation}\label{banach}
\mathcal{X}:= \Big\{ g(x,\xi): \Vert w  g \Vert_{L^\infty_{x,\xi}}<\infty, \ \Vert g\Vert_{L^2_{x,\xi}}<\infty, \ \ g\in C(\mathbb{R}^+ \times \mathbb{R}^3\backslash \mathcal{D})    \Big\},
\end{equation}
with the associated norm
\begin{equation*}
\Vert g\Vert_{\mathcal{X}} : = \Vert w  g \Vert_{\lfty} + \Vert g\Vert_{L^2_{x,\xi}}.
\end{equation*}

We start from $\lambda = 0$ and define an operator as
\begin{equation*}
T_\lambda g = \mathcal{L}_0^{-1}(\lambda K^p g + Q).
\end{equation*}

When $g\in \mathcal{X}$, from the assumption, we have that both $g,Q\in C(\mathbb{R}^+ \times \mathbb{R}^3 \backslash \mathcal{D})$. The continuity of $K^p g$ follows from the fact that $Kg$ is continuous and $\phi_u,\chi_+$ are continuous functions from Lemma \ref{lemma:eigen_continuous}. Since $g_b\in C(\mathbb{R}^3)$, without the contribution of $K^p$ acting on $T_\lambda g$, we conclude that $T_\lambda g= \mathcal{L}_0^{-1}(\lambda K^p g + Q) \in C(\mathbb{R}^+ \times \mathbb{R}^3 \backslash \mathcal{D})$.

Again, without the contribution of $K^p T_\lambda g$, $\Vert w T_\lambda g\Vert_{L^\infty_{x,\xi}}$ can be directly controlled by \eqref{wg_first}, which is bounded. Thus we can apply the a-priori estimate \eqref{aprior_wg} with the $L^2$ estimate \eqref{l2_estimate} to have
\begin{align}
 &\Vert T_\lambda g\Vert_{L^2_{x,\xi}} +  \Vert wT_\lambda g\Vert_{L^\infty_{x,\xi}} \notag \\
 & \lesssim  \Vert wg_b\Vert_{L^\infty_\xi} + \Vert  \lambda K^p g + Q \Vert_{L^2_{x,\xi}} + \big\Vert \frac{w}{[1+|\xi|]}Q \big\Vert_{L^\infty_{x,\xi}}  + \lambda \Vert w K^p g \Vert_{L^\infty_{x,\xi}} + \Vert T_\lambda g\Vert_{L^2_{x,\xi}} \notag\\
   &  \lesssim \Vert wg_b\Vert_{L^\infty_\xi} + \big\Vert \frac{w}{[1+|\xi|]}Q \big\Vert_{L^\infty_{x,\xi}} + \lambda\Vert g\Vert_{L^2_{x,\xi}} + \lambda \Vert w g \Vert_{L^\infty_{x,\xi}} + \Vert Q\Vert_{L^2_{x,\xi}} .    \label{chi_bdd}
\end{align}
Here we have used Lemma \ref{lemma:Kp}. From the assumption that $\Vert Q\Vert_{L^2_{x,\xi}}, \Vert \frac{w}{[1+|\xi|]} Q \Vert_{L^\infty_{x,\xi}},\Vert wg_b\Vert_{L^\infty_\xi},\Vert g\Vert_{\mathcal{X}}<\infty$, we conclude that $T_\lambda g \in \mathcal{X}.$

Given $g_1,g_2\in \mathcal{X}$, then $g_T:=T_\lambda (g_1-g_2)$ satisfies
\begin{equation*}
\begin{cases}
    & (\xi_1 + u) \p_x g_T + [\nu(\xi) - \gamma(\xi_1+u)]g_T = -\lambda K^p (g_1-g_2)   \\
    & g_T(0,\xi) = 0, \ \xi_1+u>0.
\end{cases}
\end{equation*}

Again we apply the a-priori estimate \eqref{aprior_wg}. With  \eqref{l2_estimate} and Lemma \ref{lemma:Kp}, we have
\begin{align}
  \Vert g_T\Vert_{\mathcal{X}}  &  = \Vert w  g_T\Vert_{\lfty} + \Vert  g_T \Vert_{L^2_{x,\xi}}  \notag \\
  & \lesssim  \lambda \Vert K^p(g_1-g_2)\Vert_{L^2_{x,\xi}} + \lambda \Vert w  K^p(g_1-g_2) \Vert_{\lfty} + \Vert g_T\Vert_{L^2_{x,\xi}}  \notag\\
  & \lesssim \lambda \Vert g_1-g_2\Vert_{L^2_{x,\xi}}+ \lambda \Vert w(g_1-g_2) \Vert_{\lfty} + \lambda  \Vert K^p(g_1-g_2)\Vert_{L^2_{x,\xi}}  \notag\\
  & \leq C\lambda \Vert g_1-g_2\Vert_{\mathcal{X}}. \label{contraction_X}
\end{align}
Then we choose $\lambda_* = \lambda_*(C)$ to be small enough such that $\lambda_* C \leq \frac{1}{2}$, then for $0\leq \lambda\leq \lambda_*$, from the Banach fixed point theorem we conclude that $T_\lambda$ has a fixed point, i.e, there exists $g\in \mathcal{X}$ such that $T_\lambda g=g$, which is equivalent to $g = \mathcal{L}_0^{-1}(\lambda K g + Q) $, and 
\begin{align*}
    & (\xi_1+u)\p_x g + [\nu(\xi)-\gamma(\xi_1+u)] g  - \lambda K^p g = Q.
\end{align*}
Therefore, the solution to~\eqref{lambda_eqn} is continuous away from $\mathcal{D}$ and satisfies $\Vert g\Vert_{\mathcal{X}}<\infty$ for $0\leq \lambda \leq \lambda_*$.

Next we define
\begin{equation*}
T_{\lambda_*+\lambda} g= \mathcal{L}^{-1}_{\lambda_*}(\lambda K g + Q).
\end{equation*}
Since the estimates \eqref{l2_estimate} and \eqref{aprior_wg} are uniform in $\lambda$, we can apply a similar fixed-point argument in~\eqref{contraction_X} to conclude that there exists a fixed point $T_{\lambda+\lambda_*} g= g$ for $g\in \mathcal{X}$. Step by step, we can conclude that $\mathcal{L}_1^{-1}(Q)$, the solution to~\eqref{pen_prob}, is in $\mathcal{X}$ and thus is continuous away from $\mathcal{D}$.

\end{proof}

\ \\

\subsection{Continuity and $w$-weighted $L^\infty$ estimate of the nonlinear penalized problem}

With the continuity and $w$-weighted $L^\infty_{x,\xi}$ estimate of the eigenfunction and the linearized penalized problem in Lemma \ref{lemma:eigen_continuous} and Proposition \ref{prop:pen_continuous}, we are ready to prove the continuity and $w$-weighted $L^\infty_{x,\xi}$ estimate of the nonlinear penalized problem. The nonlinear problem is given by
\begin{equation}\label{pen_nonlin_prob}
\begin{cases}
    &(\xi_1+u)\p_x g + \mathcal{L}^p g = e^{-\gamma x} (\mathbf{I}-\mathbf{P}_u)\Gamma(g-h\phi_u, g-h\phi_u)  , \\
    &h(x) = -e^{-\gamma x} \int_0^\infty e^{(\tau_u - 2\gamma)z} \langle \psi_u \Gamma(g-h\phi_u,g-h\phi_u)\rangle(x+z) \dd z,   \\
    & g(0,\xi) = f_b(\xi) + h(0)\phi_u(\xi), \ \ \xi_1+u>0.
\end{cases}
\end{equation}

The well-posedness of the nonlinear problem is already established in \cite{golse}:

\begin{proposition}[Proposition 6.1 in \cite{golse}]\label{prop:nonlinear_wellpose}
Suppose the boundary data $f_b(\xi)$ satisfy \eqref{bdr_small}. There exists a unique solution $(g,h)$ to \eqref{pen_nonlin_prob} such that
\[g(x,\mathcal{R}\xi) = g(x,\xi)\]
and for $\e \ll 1$,
\[\Vert (1+|\xi|)^3 g\Vert_{L^\infty_{x,\xi}} + \Vert h\Vert_{L^\infty_x}\leq C\e.\]

\end{proposition}

In the following proposition, we construct the continuity and $w$-weighted $L^\infty_{x,\xi}$ estimate of the nonlinear problem.

\begin{proposition}\label{prop:nonlinear_continuous}
Suppose the boundary data \eqref{bdr_small} further satisfies \eqref{f_b_continuous}. Then the solution in Proposition \ref{prop:nonlinear_wellpose} satisfies
\begin{equation*}
h(x) \in C(\mathbb{R}^+), \ \ g(x,\xi) \in C(\mathbb{R}^+ \times \mathbb{R}^3 \backslash \mathcal{D}), \ \ \Vert wg\Vert_{L^\infty_{x,\xi}} + \Vert h\Vert_{L^\infty_x} \leq C\e.
\end{equation*}

\end{proposition}

\begin{proof}
For some $C>0$ specified later, we denote a Banach space as
\begin{align}
  \mathcal{X}_1:=  &  \big\{(g,h) :  \Vert w g \Vert_{L^\infty_{x,\xi}} + \Vert h\Vert_{L^\infty_x} \leq 2C\e , \ g(x,\mathcal{R}\xi) = g(x,\xi),   \notag \\
  & \ \ \text{ and } h\in C(\mathbb{R}^3), \ g\in C(\mathbb{R}^+ \times \mathbb{R}^3 \backslash \mathcal{D}) \big\},  \label{nonlinear_banach}
\end{align}
with norm given by
\begin{equation*}
    \Vert (g,h)\Vert_{\mathcal{X}_1} := \Vert w g\Vert_{L^\infty_{x,\xi}} + \Vert h\Vert_{L^\infty_x}.
\end{equation*}
Given $(\tilde{g}, \tilde{h})\in \mathcal{X}_1$, we focus on the following linear problem:
\begin{equation}\label{pen_nonlin_prob_linear}
\begin{cases}
    &(\xi_1+u)\p_x g + \mathcal{L}^p g = e^{-\gamma x} (\mathbf\mathbf{P}_u)\Gamma(\tilde{g}-\tilde{h}\phi_u, \tilde{g}-\tilde{h}\phi_u)  , \\
    &h(x) = -e^{-\gamma x} \int_0^\infty e^{(\tau_u - 2\gamma)z} \langle \psi_u \Gamma(\tilde{g}-\tilde{h}\phi_u,\tilde{g}-\tilde{h}\phi_u)\rangle(x+z) \dd z,   \\
    & g(0,\xi) = f_b(\xi) + h(0)\phi_u(\xi), \ \ \xi_1+u>0.
\end{cases}
\end{equation}
The well-posedness of the above system is given by Proposition \ref{prop:pen_well_pose}, with $w$-weighted estimate given in Proposition \ref{prop:pen_continuous}. Then we denote $\mathcal{S}$ as the map from $(\tilde{g},\tilde{h})$ to the solution $(g,h)$: $\mathcal{S}((\tilde{g},\tilde{h})) = (g,h)$.

By the continuity of $\tilde{h}$ and $\tilde{g}$ from $(\tilde{g},\tilde{h})\in \mathcal{X}_1$, we conclude the continuity of $h(x)$ from the continuity of $\phi_u$ in Lemma \ref{lemma:eigen_continuous}.

Since $f_b$ is continuous from \eqref{f_b_continuous} and $\phi_u$ is continuous in Lemma \ref{lemma:eigen_continuous}, $g(0,\xi)$ is continuous. Then the continuity of $g$ follows from Proposition \ref{prop:pen_continuous}, with the continuity of $\phi_u$, $g(0,\xi)$, and the assumption that $(\tilde{g},\tilde{h})\in \mathcal{X}_1$.

Applying \eqref{Gamma bounded}, we have the estimate for the nonlinear term $\Gamma$ as
\begin{align}
    & \Big\Vert \frac{w}{[1+|\xi|]} \Gamma(\tilde{g}-\tilde{h}\phi_u, \tilde{g}-\tilde{h}\phi_u) \Big\Vert_{L^\infty_{x,\xi}}    \lesssim  \Vert w(\tilde{g} - \tilde{h}\phi_u)\Vert_{L^\infty_{x,\xi}}^2 \notag\\
    & \lesssim  \Vert w\tilde{g}\Vert^2_{L^\infty_{x,\xi}} + \Vert w\phi_u\Vert_{L^\infty_{x,\xi}}^2 \Vert \tilde{h}\Vert_{L^\infty_x}^2 \lesssim  \Vert (\tilde{g},\tilde{h})\Vert_{\mathcal{X}_1}^2.    \label{gamma_bdd}
\end{align}
In the last inequality we have applied Lemma \ref{lemma:eigen_continuous}.

Thus we bound $h(x)$ in \eqref{pen_nonlin_prob_linear} as
\begin{align}
   \Vert h\Vert_{L^\infty_x} & \lesssim  \Big\Vert \frac{w}{[1+|\xi|]} \Gamma\Big\Vert_{L^\infty_{x,\xi}} \Big\langle \psi_u \frac{[1+|\xi|]}{w} \Big\rangle \lesssim \Vert \psi_u\Vert_{L^2_\xi} \Vert (\tilde{g},\tilde{h})\Vert_{\mathcal{X}_1}^2 \lesssim \Vert (\tilde{g},\tilde{h})\Vert_{\mathcal{X}_1}^2. \label{h_bdd}
\end{align}
Here we have used Proposition \ref{prop:eigen_exists}.

For the boundary condition, we have
\begin{align*}
    & g_b = f_b + h(0)\phi_u,
\end{align*}
then from Lemma \ref{lemma:eigen_continuous} and \eqref{h_bdd}, we have
\begin{align}
  &  \Vert wg_b\Vert_{L^\infty_{\xi}} \leq \Vert wf_b\Vert_{L^\infty_\xi} + \Vert w\phi_u\Vert_{L^\infty_\xi} \Vert h\Vert_{L^\infty_x} \lesssim \Vert wf_b\Vert_{L^\infty_\xi} + \Vert (\tilde{g},\tilde{h}) \Vert_{\mathcal{X}_1}^2.    \label{gb_bdd}
\end{align}

Now we apply \eqref{weighed_infty} to $g$ in \eqref{pen_nonlin_prob_linear}, and we take $\delta=\gamma$ to have
\begin{align}
  \Vert wg\Vert_{L^\infty_{x,\xi}}  & \lesssim  \Vert wg_b\Vert_{L^\infty_\xi}  + \Big\Vert e^{(\delta-\gamma) x} \frac{w}{[1+|\xi|]}(\mathbf{I}-\mathbf{P}_u)\Gamma(\tilde{g}-\tilde{h}\phi_u,\tilde{g}-\tilde{h}\phi_u) \Big\Vert_{L^\infty_{x,\xi}}  \notag\\
  & \lesssim  \Vert wf_b\Vert_{L^\infty_\xi} +  \Vert (\tilde{g},\tilde{h})\Vert_{\mathcal{X}_1}^2. \label{g_bdd}
\end{align}
In the second line we have used \eqref{gb_bdd} and \eqref{gamma_bdd}, and applied the following computation for $\mathbf{P}_u$:
\begin{align}
    &\Big\Vert \frac{w}{[1+|\xi|]} \langle \psi_u \Gamma(\tilde{g}-\tilde{h}\phi_u,\tilde{g}-\tilde{h}\phi_u) \rangle (\xi_1+u)\phi_u \Big\Vert_{L^\infty_{x,\xi}}    \notag\\
    & \lesssim  \Vert w\phi_u \Vert_{L^\infty_\xi} \Big\Vert \frac{w}{[1+|\xi|]}\Gamma\Big\Vert_{L^\infty_{x,\xi}} \Big\langle \psi \frac{[1+|\xi|]}{w(\xi)}\Big\rangle \lesssim  \Vert w\phi_u \Vert_{L^\infty_\xi} \Vert \psi\Vert_{L^2_{\xi}} \Big\Vert \frac{w}{[1+|\xi|]}\Gamma\Big\Vert_{L^\infty_{x,\xi}} \lesssim \Big\Vert \frac{w}{[1+|\xi|]}\Gamma\Big\Vert_{L^\infty_{x,\xi}}. \label{P_u_gamma_infty}
\end{align}

Combining \eqref{h_bdd} and \eqref{g_bdd}, we conclude that for some $C>0$,
\begin{align*}
  \Vert (g,h)\Vert_{\mathcal{X}_1}  & \leq C\Vert wf_b\Vert_{L^\infty_\xi} + C\Vert (\tilde{g},\tilde{h})\Vert_{\mathcal{X}_1}^2 .
\end{align*}
Then we let $\e$ in \eqref{f_b_continuous} be small enough such that $4C^2\e<1$. Then if $\Vert 
(\tilde{g},\tilde{h}) \Vert_{\mathcal{X}_1}\leq 2C\e$, we have
\begin{align}
    & \Vert (g,h)\Vert_{\mathcal{X}_1} \leq C\e + 4C^3 \e^2 \leq 2C\e. \label{g_h_bdd}
\end{align}

Combining the continuity of $g$, $h$ and \eqref{g_h_bdd}, we conclude that $(g,h)\in \mathcal{X}_1$.

It remains to prove the contraction property. Given $(\tilde{g}_1,\tilde{h}_1),(\tilde{g}_2,\tilde{h}_2)\in \mathcal{X}_1$ as the source terms in \eqref{pen_nonlin_prob_linear}, we denote the corresponding solutions as $(g_1,h_1),(g_2,h_2)$. The equations for $g_1-g_2,h_1-h_2$ read
\begin{equation}\label{nonlinear_difference}
\begin{cases}
    &  (\xi_1+u)\p_x (g_1-g_2) + \mathcal{L}^p (g_1-g_2) = e^{-\gamma x} (\mathbf{I}-\mathbf{P}_u)\Sigma, \\
    &\Sigma =  \Gamma(\tilde{g}_1-\tilde{g}_2-(\tilde{h}_1-\tilde{h}_2)\phi_u, \tilde{g}_1 - \tilde{h}_1\phi_u)+ \Gamma(\tilde{g}_2 - \tilde{h}_2 \phi_u, \tilde{g}_1 - \tilde{g}_2 - (\tilde{h}_1-\tilde{h}_2)\phi_u) , \\
    & (h_1-h_2)(x) = -e^{-\gamma x} \int_0^\infty e^{(\tau_u-2\gamma)z} \langle \psi_u \Sigma\rangle(x+z) \dd z, \\
    & (g_1-g_2)(0,\xi) = (h_1-h_2)(0) \phi_u(\xi),  \ \ \xi_1+u>0.     
\end{cases}
\end{equation}

Applying \eqref{Gamma bounded} and the same computation in \eqref{gamma_bdd}, the source term is bounded as
\begin{align}
  \Big\Vert \frac{w}{[1+|\xi|]}\Sigma\Big\Vert_{L^\infty_{x,\xi}}  & \lesssim  \Vert w[\tilde{g}_1-\tilde{g}_2-(\tilde{h}_1-\tilde{h}_2)\phi_u ]\Vert_{L^\infty_{x,\xi}} \big[ \Vert w(\tilde{g}_1-\tilde{h}_1 \phi_u)\Vert_{\lfty} + \Vert w(\tilde{g}_2-\tilde{h}_2 \phi_u)\Vert_{\lfty}  \big] \notag\\
  &\lesssim [\Vert (\tilde{g}_1,\tilde{h}_1)\Vert_{\mathcal{X}_1} + \Vert (\tilde{g}_2,\tilde{h}_2)\Vert_{\mathcal{X}_1}]\Vert (\tilde{g}_1-\tilde{g}_2,\tilde{h}_1-\tilde{h}_2)\Vert_{\mathcal{X}_1}  \lesssim \e \Vert (\tilde{g}_1-\tilde{g}_2,\tilde{h}_1-\tilde{h}_2)\Vert_{\mathcal{X}_1} . \notag
\end{align}
Then we apply the same computation in \eqref{h_bdd} to bound $h_1-h_2$ as
\begin{align}
  \Vert h_1-h_2\Vert_{L^\infty_x}  & \lesssim  \Big\Vert \frac{w}{[1+|\xi|]}\Sigma\Big\Vert_{L^\infty_{x,\xi}} \lesssim \e \Vert (\tilde{g}_1-\tilde{g}_2,\tilde{h}_1-\tilde{h}_2)\Vert_{\mathcal{X}_1} .  \label{h_12_bdd}
\end{align}

We denote the incoming boundary as $(g_1-g_2)_b$, then it is bounded similarly as \eqref{h_12_bdd}:
\begin{align*}
  \Vert w(g_1-g_2)_b\Vert_{L^\infty_\xi}  &  \lesssim \Vert h_1-h_2\Vert_{L^\infty_x} \lesssim \e\Vert (\tilde{g}_1-\tilde{g}_2, \tilde{h}_1-\tilde{h}_2)\Vert_{\mathcal{X}_1}.  
\end{align*}

Now we apply \eqref{weighed_infty} with $\delta = \gamma$ and similar computation in \eqref{g_bdd} to bound $g_1-g_2$ as
\begin{align}
  \Vert w(g_1-g_2)\Vert_{L^\infty_{x,\xi}}  & \lesssim  \Vert w(g_1-g_2)_b\Vert_{L^\infty_{\xi}} +  \Big\Vert \frac{w}{[1+|\xi|]}(\mathbf{I}-\mathbf{P}_u)\Sigma \Big\Vert_{L^\infty_{x,\xi}} \lesssim \e\Vert (\tilde{g}_1-\tilde{g}_2, \tilde{h}_1-\tilde{h}_2)\Vert_{\mathcal{X}_1}. \label{g_12_bdd}
\end{align}

Combining \eqref{h_12_bdd} and \eqref{g_12_bdd} we conclude that for some $C_1>0$,
\begin{align}
  \Vert (g_1-g_2,h_1-h_2) \Vert_{\mathcal{X}_1}   &  \leq C_1  \e \Vert (\tilde{g}_1-\tilde{g}_2,\tilde{h}_1-\tilde{h}_2)\Vert_{\mathcal{X}_1}.    \label{chi_1_contraction}
\end{align}
Last we take $\e$ to be small such that $C_1\e < 1$. We conclude the proposition from the Banach fixed-point theorem.

\end{proof}

\ \\

\subsection{Proof of Theorem \ref{thm:continuity}}\label{sec:proof_contin}

Finally, the solution of the boundary layer problem \eqref{equation_f} is constructed by assuming further condition \eqref{extra_assumption} on the boundary data. The solution reads
\begin{equation*}
f(x,\xi) = e^{-\gamma x}g(x,\xi) - e^{-\gamma x}h(x)\phi_u(\xi),
\end{equation*}
where $g,h$ are the unique solution constructed in Proposition \ref{prop:nonlinear_wellpose}. With the continuity of $g,h,\phi_u$ and $w$-weighted estimate of $g,\phi$ established in Lemma \ref{lemma:eigen_continuous} and Proposition \ref{prop:nonlinear_continuous}, we conclude Theorem \ref{thm:continuity}.

\ \\

\section{Weighted $C^1$ estimate and $W^{1,p}$ estimate for $p<2$ without weight}\label{sec:regularity}
In the section, we will conclude Theorem \ref{thm:weight_C1} by applying the weight $\alpha$ in \eqref{alpha} with its property in Section \ref{sec:prelim_weight}.

We begin with proving the existence of derivative to the damped transport equation
\begin{equation}\label{linear_eqn}
\begin{cases}
   & (\xi_1+u)\p_x f(x,\xi) + \nu f(x,\xi) = Q  \\
   & f(0,\xi)|_{\xi_1+u>0} = f_b(\xi).
\end{cases}
\end{equation}

\hide
By method of characteristic, the solution to \eqref{linear_eqn} can be expressed as
\begin{equation}\label{characteristic_g}
\begin{split}
    f(x,\xi) & = [\mathbf{1}_{\frac{x}{\xi_1+u}>t}+\mathbf{1}_{\xi_1+u<0}]e^{-\nu(\xi) t} f(x-(\xi_1+u)t,\xi)  \\
   & \  +\mathbf{1}_{0<\frac{x}{\xi_1+u}\leq t} e^{-\nu(\xi) \frac{x}{\xi_1+u}} f(0,\xi)\\
   &\ + \int^t_{\max\{0,\mathbf{1}_{\xi_1+u>0}(t-\frac{x}{\xi_1+u})\}} \dd s e^{-\nu(\xi)(t-s)} Q(x-(\xi_1+u)(t-s),\xi)  .    
\end{split}
\end{equation}
Here we choose a large and fixed $t\gg 1$.

We consider the following expansion:
\begin{equation}\label{p_x_expansion}
\begin{split}
 \p_x f(x,\xi)&  =  -\mathbf{1}_{x< t(\xi_1+u)}  \frac{\nu(\xi)}{\xi_1+u} e^{-\nu(\xi) \frac{x}{\xi_1+u}} f(0,\xi)       \\
  & +\mathbf{1}_{x> t(\xi_1+u)} e^{-\nu(\xi) t} \p_x f (x-(\xi_1+u)t,\xi) \\
  & + \int^t_{\max\{0,\mathbf{1}_{\xi_1+u>0}(t-\frac{x}{\xi_1+u})\}} \dd s e^{-\nu(\xi)(t-s)} \p_x Q(x-(\xi_1+u)(t-s),\xi) \\
  & -\mathbf{1}_{x< t(\xi_1+u)} \frac{1}{\xi_1+u} e^{-\nu(\xi) \frac{x}{\xi_1+u}} Q(0,\xi).    
\end{split}
\end{equation}

\unhide

\begin{lemma}\label{lemma:weak_deri}
Suppose $|Q|<\infty$ and $\p_x Q$ exists. Then $\p_x f(x,\xi)$ exists.
\end{lemma}

\begin{proof}
First we consider the case of $\xi_1+u=0$. Then $f(x,\xi)$ can be expressed as
\begin{align*}
    & f(x,\xi) = e^{-\nu(\xi)t}f(x,\xi) + \int^t_0 \dd s e^{-\nu(\xi)(t-s)} Q(x,\xi).
\end{align*}
The difference quotient reads
\begin{align*}
    &  \frac{f(x+\e,\xi)-f(x,\xi)}{\e} = e^{-\nu(\xi)t} \frac{f(x+\e,\xi)-f(x,\xi)}{\e} + \int^t_0 \dd s e^{-\nu(\xi)(t-s)} \frac{Q(x+\e,\xi)-Q(x,\xi)}{\e},
\end{align*}
thus
\begin{align*}
    (1-e^{-\nu(\xi)t})\frac{f(x+\e,\xi)-f(x,\xi)}{\e} = \int^t_0 \dd s e^{-\nu(\xi)(t-s)} \frac{Q(x+\e,\xi)-Q(x,\xi)}{\e}.
\end{align*}
From the assumption that $\p_x Q$ exists, we can pass $\e$ to $0$ to conclude that $\p_x f(x,\xi)$ exists for $\xi_1+u=0$.

Next, we consider the case of $\xi_1+u\neq 0$. Then from the equation \eqref{linear_eqn}, we have $|\p_x f| = \Big|\frac{Q-\nu f}{\xi_1+u}\Big| < \infty $ from $\Vert wf\Vert_{\lfty}<\infty$. Thus $\p_x f(x,\xi)$ exists.

\hide
Next we consider the case of $\xi_1 + u < 0$. By the $L^\infty$ estimate of $f$ in Theorem \ref{thm:continuity}, we can also express $f$ as
\begin{align*}
    f(x,\xi) \mathbf{1}_{\xi_1+u < 0} & = \int^t_{-\infty} \dd s e^{-\nu(t-s)} Q(x-(\xi_1+u)(t-s),\xi).
\end{align*}
Clearly, we have 
\begin{align*}
 \p_x  f(x,\xi)  \mathbf{1}_{\xi_1+u < 0} & = \int_{-\infty}^t \dd s e^{-\nu(t-s)} \p_x Q(x-(\xi_1+u)(t-s),\xi).   \\
\end{align*}
Thus, $\p_x f$ exists when $\xi_1+u < 0$. Since we can also write
\begin{align*}
   f(x,\xi) \mathbf{1}_{\xi_1+u < 0}&  = \mathbf{1}_{\xi_1+u < 0} \big[ e^{-\nu(\xi) t} f(x-(\xi_1+u)t,\xi) + \int^t_0 \dd s e^{-\nu(\xi)(t-s)} Q(x-(\xi_1+u)(t-s),\xi)   \big],
\end{align*}
the existence of the $\p_x f$ leads to another representation of $\p_x f:$
\begin{equation}\label{xi_u_leq_0}
\begin{split}
  \p_x f(x,\xi) \mathbf{1}_{x_1+u\leq 0} &  = \mathbf{1}_{\xi_1+u\leq 0} e^{-\nu(\xi) t} \p_x f(x-(\xi_1+u)t,\xi)   \\
  & + \mathbf{1}_{\xi_1+u\leq 0} \int_0^t \dd s e^{-\nu(\xi)(t-s)} \p_x Q(x-(\xi_1+u)(t-s),\xi).
\end{split}
\end{equation}

Then we consider the case of $\xi_1+u > 0$. Then we express $f$ as
\begin{align*}
  f(x,\xi)\mathbf{1}_{\xi_1+u>0}  &  =    e^{-\nu(\xi) \frac{x}{\xi_1+u}} f(0,\xi) \\
  & + \int^t_{t-\frac{x}{\xi_1+u}} \dd s e^{-\nu(\xi)(t-s)} Q(x-(\xi_1+u)(t-s),\xi).
\end{align*}
Then $\p_x f$ exists as
\begin{align*}
  \p_x f(x,\xi)  &  = -\mathbf{1}_{\xi_1+u > 0} \big[ \frac{\nu(\xi)}{\xi_1+u} f(0,\xi) + \frac{1}{\xi_1+u} e^{-\nu(\xi)\frac{x}{\xi_1+u}} Q(0,\xi) \big] \\
  & + \int^t_{t-\frac{x}{\xi_1+u}} \dd s e^{-\nu(\xi)(t-s)} \p_x Q(x-(\xi_1+u)(t-s),\xi).
\end{align*}

Then when $x>t(\xi_1+u)$, we can also express $f$ as
\begin{align*}
  f(x,\xi) \mathbf{1}_{x>t(\xi_1+u)>0}  &  =  \mathbf{1}_{x>t(\xi_1+u)>0} e^{-\nu(\xi) t} f(x-t(\xi_1+u),\xi)   \\
  & +  \mathbf{1}_{x>t(\xi_1+u)>0} \int_{0}^t \dd s e^{-\nu(\xi)(t-s)} Q(x-(\xi_1+u)(t-s),\xi).
\end{align*}

Since $\p_x f$ exists, we can also represent $\p_x f$ as
\begin{align*}
  \p_x f(x,\xi)  \mathbf{1}_{x>t(\xi_1+u)>0}    &  =  \mathbf{1}_{x>t(\xi_1+u)>0}   e^{-\nu(\xi) t} \p_x f(x-t(\xi_1+u),\xi)  \\
  & +  \mathbf{1}_{x>t(\xi_1+u)>0}   \int_0^t \dd s e^{-\nu(\xi)(t-s)} \p_x Q(x-(\xi_1+u)(t-s),\xi).
\end{align*}

\unhide

\end{proof}

\ \\

\subsection{Weighted $C^1$ estimate of the linearized penalized problem}

Before we construct the derivative of the linearized penalized problem, we first establish an a-priori estimate in the following lemma:
\begin{lemma}[\textbf{A-priori weighted $C^1$ estimate}]\label{lemma:aprori_linear}
Let $g$ be the solution of the linearized problem \eqref{pen_prob}, suppose $\p_x g$ exists and satisfies
\begin{align*}
    (\xi_1+u)\p_x (\p_x g) + \mathcal{L}^p (\p_x g) = \p_x Q, \ x>0, \ \xi\in \mathbb{R}^3.
\end{align*}
If we further assume $\Vert w_{\tilde{\theta}}\alpha \p_x g\Vert_{\lfty}<\infty$, then we have
\begin{align}
  &\Vert w_{\tilde{\theta}}\alpha \p_x g \Vert_{L^\infty_{x,\xi}} \lesssim  \Vert w g\Vert_{L^\infty_{x,\xi}} + \Vert wg_b\Vert_{L^\infty_\xi} + \Vert w_{\tilde{\theta}}(\xi)Q(0,\xi)\Vert_{L^\infty_{\xi}}  \notag\\
 & + \underbrace{\sup_{x,\xi}\Big[w_{\tilde{\theta}}\alpha(x,\xi) \int^t_{\max\{0,\mathbf{1}_{\xi_1+u>0}(t-\frac{x}{\xi_1+u})\}} \dd s e^{-(\nu(\xi)-\gamma(\xi_1+u))(t-s)} |\p_x Q(x-(\xi_1+u)(t-s),\xi)| \Big]}_{\eqref{aprior_deri}_*}. \label{aprior_deri}
\end{align}

\end{lemma}

\begin{proof}
First we let $1\ll t$ be large and fixed, so that 
\begin{equation}\label{t_large}
\frac{1}{t} \ll 1, \ \ e^{-\nu_0 t/4}\ll 1 \text{ and }   t e^{-\nu_0 t/4}\ll 1 .    
\end{equation}

Recall the definition of $K^p$ in \eqref{Kp}. We denote 
\begin{equation}\label{nu_bar}
\bar{\nu}(\xi) := \nu(\xi) - \gamma(\xi_1+u)\geq \frac{\nu(\xi)}{2}\geq \frac{\nu_0(\xi)}{2}:= \frac{\nu_0 [1+|\xi|]}{2}.    
\end{equation}

By method of characteristic, we express $g$ as
\begin{align*}
    g(x,\xi) & = [\mathbf{1}_{\frac{x}{\xi_1+u}>t}+\mathbf{1}_{\xi_1+u<0}]e^{-\bar{\nu}(\xi) t} g(x-(\xi_1+u)t,\xi)  \\
   & \  +\mathbf{1}_{0<\frac{x}{\xi_1+u}\leq t} e^{-\bar{\nu}(\xi) \frac{x}{\xi_1+u}} g(0,\xi)\\
   &\  + \int^t_{\max\{0,\mathbf{1}_{\xi_1+u>0}(t-\frac{x}{\xi_1+u})\}} \dd s e^{-\bar{\nu}(\xi)(t-s)} K^p g(x-(\xi_1+u)(t-s),\xi) \\
   &\ + \int^t_{\max\{0,\mathbf{1}_{\xi_1+u>0}(t-\frac{x}{\xi_1+u})\}} \dd s e^{-\bar{\nu}(\xi)(t-s)} Q(x-(\xi_1+u)(t-s),\xi)  .
\end{align*}
From the assumption that $\p_x g$ exists, we take derivative to the above formula to have 
\begin{align}
  \p_x g(x,\xi)  & = [\mathbf{1}_{\frac{x}{\xi_1+u}>t}+\mathbf{1}_{\xi_1+u<0}]e^{-\bar{\nu}(\xi) t} \p_x g(x-(\xi_1+u)t,\xi) \label{p_x_1} \\
     & \ - \mathbf{1}_{0<\frac{x}{\xi_1+u}\leq t} \frac{\bar{\nu}(\xi)}{\xi_1+u} e^{-\bar{\nu}(\xi) \frac{x}{\xi_1+u}} g(0,\xi) \label{p_x_2}\\
     & \ -  \mathbf{1}_{0<\frac{x}{\xi_1+u}\leq t} \frac{1}{\xi_1+u} e^{-\bar{\nu}(\xi) \frac{x}{\xi_1+u} }K^p g(0,\xi) \label{p_x_3}\\
     & \ +  \int^t_{\max\{0,\mathbf{1}_{\xi_1+u>0}(t-\frac{x}{\xi_1+u})\}} \dd s e^{-\bar{\nu}(\xi)(t-s)} K^p \p_x g(x-(\xi_1+u)(t-s),\xi) \label{p_x_4} \\
     & \ -  \mathbf{1}_{0<\frac{x}{\xi_1+u}\leq t} \frac{1}{\xi_1+u} e^{-\bar{\nu}(\xi)\frac{x}{\xi_1+u}} Q(0,\xi) \label{p_x_5}\\
     & \ + \int^t_{\max\{0,\mathbf{1}_{\xi_1+u>0}(t-\frac{x}{\xi_1+u})\}} \dd s e^{-\bar{\nu}(\xi)(t-s)} \p_x Q(x-(\xi_1+u)(t-s),\xi)\label{p_x_6} .
\end{align}

We estimate~\eqref{p_x_1} - \eqref{p_x_6} term by term.

For~\eqref{p_x_1}, we apply~\eqref{nu_bar} and to have
\begin{align*}
  |\eqref{p_x_1}|  & \lesssim e^{-\nu_0 t/2} \frac{\Vert w_{\tilde{\theta}}\alpha \p_x g\Vert_{L^\infty_{x,\xi}}}{w_{\tilde{\theta}}(\xi)\alpha(x-(\xi_1+u)t,\xi)} \lesssim e^{-\nu_0t/4} \frac{\Vert w_{\tilde{\theta}}\alpha\p_x g\Vert_{L^\infty_{x,\xi}}}{w_{\tilde{\theta}}(\xi)\alpha(x,\xi)} .
\end{align*}
In the last inequality, we applied Lemma \ref{lemma:velocity_alpha} with $c=1/8$.

For~\eqref{p_x_2}, we use~\eqref{alpha_bdr} to have
\begin{align}
  |\eqref{p_x_2}|  & \lesssim \frac{1}{\alpha(0,\xi)} e^{- \frac{\nu_0 x}{2(\xi_1+u)}} \frac{[1+|\xi|]}{w(\xi)} \Vert  w  g_b\Vert_{L^\infty_{\xi}}   \notag\\
  & \lesssim  \frac{1}{w_{\tilde{\theta}}(\xi)\alpha(x,\xi)} e^{\frac{\nu_0 x}{4(\xi_1+u)}}e^{- \frac{\nu_0x}{2(\xi_1+u)}} \Vert w g_b\Vert_{L^\infty_{\xi}} \lesssim \frac{\Vert wg_b\Vert_{L^\infty_{\xi}}}{w_{\tilde{\theta}}(\xi)\alpha(x,\xi)}. \label{p_x_2_bdd}
\end{align}
In the second line we applied Lemma \ref{lemma:velocity_alpha} with $s=x/(\xi_1+u)$ and $c=\frac{1}{8}$.

For~\eqref{p_x_3}, we first consider the contribution of $Kg$. Applying Lemma \ref{lemma:k_theta}, we have
\begin{align*}
  &\frac{1}{\xi_1+u} e^{-\bar{\nu}(\xi)\frac{x}{\xi_1+u}} \frac{1}{w_{\tilde{\theta}}(\xi)}\int_{\mathbb{R}^3} \mathbf{k}(\xi,\xi') \frac{w_{\tilde{\theta}}(\xi)}{w_{\tilde{\theta}}(\xi')} w_{\tilde{\theta}}(\xi') g(0,\xi') \dd \xi'  \\
  & \lesssim  \frac{1}{\alpha(0,\xi)} e^{-\frac{\nu_0 x}{2(\xi_1+u)}} \Vert w g\Vert_{L^\infty_{x,\xi}}   \lesssim          \frac{1}{\alpha(x,\xi)} \Vert wg\Vert_{L^\infty_{x,\xi}},
\end{align*}
in the last inequality we applied the same computation as~\eqref{p_x_2_bdd}. 

For the contribution of $\gamma \prod_+ ((\xi_1+u)g)$ and $\gamma \mathbf{p}_u g$ in \eqref{p_x_3}, we have
\begin{align*}
 & \frac{1}{w_{\tilde{\theta}}(\xi)} \frac{e^{-\bar{\nu}(\xi)\frac{x}{\xi_1+u}}}{\xi_1+u} \gamma w_{\tilde{\theta}}(\xi) \big[\prod_+((\xi_1+u)g(0)) + \mathbf{p}_u g(0) \big]    \\
 &\lesssim  \frac{1}{w_{\tilde{\theta}}(\xi)} \frac{\gamma}{\xi_1+u} e^{-\bar{\nu}(\xi)\frac{x}{\xi_1+u}}  \Vert wg \Vert_{L^\infty_{x,\xi}}  \lesssim \frac{\Vert wg\Vert_{L^\infty_{x,\xi}}}{w_{\tilde{\theta}}(\xi)\alpha(x,\xi)}.   
\end{align*}
Here we have applied the same computation in \eqref{prod_bdd} and \eqref{p_g_bdd}.

Thus we conclude the estimate for \eqref{p_x_3} as
\begin{align}
  |\eqref{p_x_3}|  &  \lesssim      \frac{\Vert wg\Vert_{L^\infty_{x,\xi}}}{\alpha(x,\xi)}.    \notag
\end{align}

\eqref{p_x_5} is directly bounded as
\begin{equation*}
    |\eqref{p_x_5}|  \lesssim    \frac{1}{w_{\tilde{\theta}}(\xi)\alpha(0,\xi)} e^{-\frac{\nu_0 x}{2(\xi_1+u)}} \Vert w_{\tilde{\theta}} Q(0)\Vert_{L^\infty_\xi} \lesssim \frac{\Vert w_{\tilde{\theta}} Q(0)\Vert_{L^\infty_{\xi}}}{w_{\tilde{\theta}}(\xi)\alpha(x,\xi)}.
\end{equation*}

The contribution of \eqref{p_x_6} is already included in \eqref{aprior_deri}.

For~\eqref{p_x_4}, first we compute the contribution of $\gamma \prod_+ ((\xi_1+u)g)$ and $\gamma \mathbf{p}_u g$ as
\begin{align}
 & \gamma \int^t_{\max\{0,\mathbf{1}_{\xi_1+u>0}(t-\frac{x}{\xi_1+u})\}} \dd s e^{-\nu_0[1+|\xi|](t-s)/2}   \Big|\chi_+(\xi)\langle  (\xi_1+u) \chi_+\p_x g\rangle(x-(\xi_1+u)(t-s))   \notag\\
 &\quad \quad \quad \quad \quad +   \phi_u(\xi) \langle (\xi_1+u)\psi_u \p_x g\rangle(x-(\xi_1+u)(t-s))  \Big| \notag\\
 &\lesssim \frac{\gamma}{w_{\tilde{\theta}}(\xi)}\times \big[\Vert w_{\tilde{\theta}}\chi_+\Vert_{L^\infty_\xi} + \Vert w_{\tilde{\theta}}\phi_u\Vert_{L^\infty_\xi} \big] \Vert \alpha\p_x g\Vert_{L^\infty_{x,\xi}} \notag\\
 &\times \int^t_{\max\{0,\mathbf{1}_{\xi_1+u>0}(t-\frac{x}{\xi_1+u})\}} \dd s e^{-\nu_0 [1+|\xi|](t-s)/2} \int_{\mathbb{R}^3} \frac{w^{-1/2}(\xi')}{\alpha(x-(\xi_1+u)(t-s),\xi')} \dd \xi' \notag\\
 & \lesssim \frac{\gamma t\Vert  w_{\tilde{\theta}}\alpha\p_x g\Vert_{L^\infty_{x,\xi}}}{w_{\tilde{\theta}}(\xi) \alpha(x,\xi)}.   \label{K_other_bdd}
\end{align}
In the last line we have applied Lemma \ref{lemma:NLN_inner}. In the second line we used Lemma \ref{lemma:eigen_continuous}. In the third line, for the integration $\langle (\xi_1+u)\chi_+ \p_x g\rangle$ and $\langle (\xi_1+u)\psi \p_x g\rangle$, we used Lemma \ref{Lemma: w_psi_infty} to have
\begin{align*}
  &  \big|\langle (\xi_1+u)\chi_+ \p_x g\rangle + \langle (\xi_1+u)\psi_u \p_x g\rangle \big| = \int_{\mathbb{R}^3} [(\xi'_1 + u) \chi_+(\xi') \p_x g(\xi') + (\xi_1'+u) \psi_u(\xi') \p_x g(\xi')] \dd \xi'\\
  & \lesssim [\Vert w\psi_u\Vert_{L^\infty_\xi} + \Vert w \chi_+ \Vert_{L^\infty_\xi}] \Vert \alpha \p_x g\Vert_{L^\infty_{x,\xi}}\int_{\mathbb{R}^3} \frac{w^{-1/2}(\xi')}{\alpha(x-(\xi_1+u)(t-s),\xi')} \dd \xi'.
\end{align*}

We conclude that
\begin{align}
 & \eqref{p_x_1} + \eqref{p_x_2} + \eqref{p_x_3} + \eqref{p_x_5} + \eqref{p_x_6} + \eqref{p_x_4}_{\gamma \mathbf{p}_u g \text{ and }\gamma \prod_+((\xi_1+u)g)} \notag\\
 &\lesssim \frac{[e^{-\nu_0 t/4}+\gamma t]\Vert w_{\tilde{\theta}}\alpha \p_x g\Vert_{L^\infty_{x,\xi}} + \Vert wg_b\Vert_{L^\infty_\xi} + \Vert wg\Vert_{L^\infty_{x,\xi}} + \Vert w_{\tilde{\theta}}Q(0)\Vert_{L^\infty_\xi} + \eqref{aprior_deri}_*}{w_{\tilde{\theta}}(\xi)\alpha(x,\xi)}. \label{123_bdd}
\end{align}

For the contribution of 
\[K \p_x g(x-(\xi_1+u)(t-s),\xi)=\int_{\mathbb{R}^3} \mathbf{k}(\xi,\xi') \p_x g(x-(\xi_1+u)(t-s),\xi')\] 
in \eqref{p_x_4}, we expand $\p_x g(x-(\xi_1+u)(t-s),\xi')$ along $\xi'$ using~\eqref{p_x_1} - \eqref{p_x_6}.

In the expansion of $\p_x g(x-(\xi_1+u)(t-s),\xi')$,  the contribution of the $\eqref{p_x_1} - \eqref{p_x_6}$ except $Kg$ are bounded by \eqref{123_bdd} with replacing $(x,\xi)$ by $(x-(\xi_1+u)(t-s),\xi')$:
\begin{align}
    & \frac{1}{w_{\tilde{\theta}}(\xi)}\int^t_{\max\{0,\mathbf{1}_{\xi_1+u>0}(t-\frac{x}{\xi_1+u})\}} \dd s  e^{-\nu_0 [1+|\xi|] (t-s)/2} \int_{\mathbb{R}^3} \dd \xi' w_{\tilde{\theta}}(\xi)\mathbf{k}(\xi,\xi') \notag\\
    & \times \frac{[e^{-\nu_0 t/4}+\gamma t]\Vert \alpha w_{\tilde{\theta}}\p_x g\Vert_{L^\infty_{x,\xi}} + \Vert wg_b\Vert_{L^\infty_\xi} + \Vert wg\Vert_{L^\infty_{x,\xi}} + \Vert w_{\tilde{\theta}} Q(0)\Vert_{L^\infty_\xi} + \eqref{aprior_deri}_*}{w_{\tilde{\theta}}(\xi')\alpha(x-(\xi_1+u)(t-s),\xi')} \notag\\
    & \lesssim \frac{t\Big\{[e^{-\nu_0 t/4}+\gamma t]\Vert w_{\tilde{\theta}}\alpha \p_x g\Vert_{L^\infty_{x,\xi}} + \Vert wg_b\Vert_{L^\infty_\xi} + \Vert wg\Vert_{L^\infty_{x,\xi}} + \Vert w_{\tilde{\theta}} Q(0)\Vert_{L^\infty_\xi} + \eqref{aprior_deri}_* \Big\}}{w_{\tilde{\theta}}(\xi)\alpha(x,\xi)}.   \label{K_123456_bdd}
\end{align}
In the last line, first we applied \eqref{k_theta_bdd}, then we applied Lemma \ref{lemma:NLN}.

We focus on the contribution of the $Kg$ in~\eqref{p_x_4}, which induces a double Duhamel formula. Denote $y=x-(\xi_1+u)(t-s)$, this formula equals 
\begin{align}
    & \int^t_{\max\{0,\mathbf{1}_{\xi_1+u>0} (t-\frac{x}{\xi_1+u})\}} \dd s e^{-\bar{\nu}(\xi)(t-s)} \int_{\mathbb{R}^3} \dd \xi' \mathbf{k}(\xi,\xi') \notag\\
    & \times \int^s_{\max\{0,\mathbf{1}_{\xi_1'+u>0}(s-\frac{y}{\xi_1'+u})\}}\dd s' e^{-\bar{\nu}(\xi')(s-s')}  \int_{\mathbb{R}^3} \dd \xi'' \mathbf{k}(\xi',\xi'') \p_x g(y-(\xi'_1+u)(s-s'),\xi''). \label{kk}
\end{align}
We split the $s'$-integral into
\begin{align}
    & \underbrace{\mathbf{1}_{s-s'\leq \e}}_{\eqref{s'_integral}_1} + \underbrace{\mathbf{1}_{s-s'>\e}}_{\eqref{s'_integral}_2}. \label{s'_integral}
\end{align}

The contribution of \eqref{s'_integral}$_1$ in~\eqref{kk} is bounded as
\begin{align}
    &\frac{1}{w_{\tilde{\theta}}(\xi)}\int^t_{\max\{0,\mathbf{1}_{\xi_1+u>0} (t-\frac{x}{\xi_1+u})\}} \dd s e^{-\nu_0(\xi)(t-s)/2} \int_{\mathbb{R}^3} \dd \xi' \mathbf{k}(\xi,\xi')\frac{w_{\tilde{\theta}}(\xi)}{w_{\tilde{\theta}}(\xi')}  \notag\\
    &\times  \int^s_{s-\e} \dd s' e^{-\nu_0(\xi')(s-s')/2} \int_{\mathbb{R}^3} \dd \xi'' \mathbf{k}(\xi',\xi'') \frac{w_{\tilde{\theta}}(\xi')\Vert w_{\tilde{\theta}}\alpha\p_x g\Vert_{L^\infty_{x,\xi}}}{w_{\tilde{\theta}}(\xi'')\alpha(y-(\xi'_1+u)(s-s'),\xi'')} \notag\\
    & \lesssim \frac{1}{w_{\tilde{\theta}}(\xi)}\int^t_{\max\{0,\mathbf{1}_{\xi_1+u>0} (t-\frac{x}{\xi_1+u})\}} \dd s e^{-\nu_0(\xi)(t-s)/2} \int_{\mathbb{R}^3} \dd \xi' \mathbf{k}(\xi,\xi')\frac{w_{\tilde{\theta}}(\xi)}{w_{\tilde{\theta}}(\xi')} \frac{[\sqrt{\e}+\e \ln(t)]\Vert w_{\tilde{\theta}}\alpha \p_x g\Vert_{L^\infty_{x,\xi}}}{\alpha(y,\xi')} \notag\\
    & \lesssim \frac{t[\sqrt{\e}+\e \ln(t)]\Vert w_{\tilde{\theta}} \alpha \p_x g\Vert_{\lfty}}{w_{\tilde{\theta}}(\xi) \alpha(x,\xi)} \lesssim \frac{1}{t}\frac{\Vert w_{\tilde{\theta}}\alpha \p_x g\Vert_{L^\infty_{x,\xi}}}{w_{\tilde{\theta}}(\xi) \alpha(x,\xi)}. \label{kk_leq_e_bdd}
\end{align}
In the third line we applied~\eqref{nln_e} in Lemma \ref{lemma:NLN} with \eqref{k_theta_bdd} for the $\dd s'$ integral. In the last line, we applied \eqref{nln_large} with \eqref{k_theta_bdd} for the $\dd s$ integral. In the last inequality we take $\e = \frac{1}{t^4}$, which is small since $t \gg 1$, and that
\begin{equation}\label{e_t}
\sqrt{\e} = \frac{1}{t^2}, \ \ t[\sqrt{\e}+\e \ln(t)] \lesssim \frac{1}{t} + \frac{\ln t}{t^3} \lesssim \frac{1}{t} . 
\end{equation}

For the contribution of \eqref{s'_integral}$_2$ in~\eqref{kk}, without loss of generality, we assume $s>\e$. Otherwise, we bound \eqref{s'_integral}$_2$ by \eqref{kk_leq_e_bdd}. Then we observe the following chain rule:
\begin{align}
    & \frac{\p_{\xi'_1} [g(y-(\xi'_1+u)(s-s'),\xi'')]}{-(s-s')}    =  \p_x g(y-(\xi'_1+u)(s-s'),\xi''). \label{chain_rule}
\end{align}
In such case $s-s'\leq \e$, we only integrate over the $\xi'$ such that $\frac{y}{\xi_1'+u}>\e$, we denote
\begin{align*}
    & V:  = \{\xi'\in \mathbb{R}^3: \frac{y}{\xi'_1+u} > \e\} , \  \tilde{V}: = \{\xi'\in \mathbb{R}^3: \frac{y}{\xi'_1+u} = \e\}.
\end{align*}
We apply \eqref{chain_rule} and an integration by part for $\dd \xi'_1$ to compute this contribution as
\begin{align}
    & \int^t_{\max\{0,\mathbf{1}_{\xi_1+u>0} (t-\frac{x}{\xi_1+u})\}} \dd s e^{-\bar{\nu}(\xi)(t-s)} \int_{V} \dd \xi' \mathbf{k}(\xi,\xi')  \notag\\
    &\times \int^{s-\e}_{\max\{0,\mathbf{1}_{\xi_1'+u>0}(s-\frac{y}{\xi_1'+u})\}}\dd s' e^{-\bar{\nu}(\xi')(s-s')}  \int_{\mathbb{R}^3} \dd \xi'' \mathbf{k}(\xi',\xi'') \frac{\p_{\xi'_1} [g(y-(\xi'_1+u)(s-s'),\xi'')]}{-(s-s')} \notag\\
    & = \int^t_{\max\{0,\mathbf{1}_{\xi_1+u>0} (t-\frac{x}{\xi_1+u})\}} \dd s e^{-\bar{\nu}(\xi)(t-s)} \int_{V} \dd \xi' \int^{s-\e}_{\max\{0,\mathbf{1}_{\xi_1'+u>0}(s-\frac{y}{\xi_1'+u})\}}\dd s' \notag\\
    &\times    
  \p_{\xi'_1} [\mathbf{k}(\xi,\xi') \mathbf{k}(\xi',\xi'') e^{-\bar{\nu}(\xi')(s-s')} ]    \int_{\mathbb{R}^3} \dd \xi''  \frac{g(y-(\xi'_1+u)(s-s'),\xi'')}{-(s-s')}  \label{kk_geq_e_1}\\
    & +  \int^t_{\max\{0,\mathbf{1}_{\xi_1+u>0} (t-\frac{x}{\xi_1+u})\}} \dd s e^{-\bar{\nu}(\xi)(t-s)} \int_{V} \dd \xi' \mathbf{k}(\xi,\xi') \notag\\
    & \times \p_{\xi'_1} [\frac{y}{\xi_1'+u}] e^{-\frac{\bar{\nu}(\xi')y}{(\xi'_1+u)} } \int_{\mathbb{R}^3} \dd \xi'' \mathbf{k}(\xi',\xi'')  \frac{g(0,\xi'')}{y/(\xi_1'+u)} \label{kk_geq_e_2} \\
    & + \int^t_{\max\{0,\mathbf{1}_{\xi_1+u>0} (t-\frac{x}{\xi_1+u})\}}   \dd s e^{-\bar{\nu}(\xi)(t-s)} \int_{\tilde{V}} \dd \xi' \mathbf{k}(\xi,\xi') \notag \\
    &\times \int^{s-\e}_{s-\e}\dd s' e^{-\bar{\nu}(\xi')(s-s')}  \int_{\mathbb{R}^3} \dd \xi'' \mathbf{k}(\xi',\xi'') \frac{g(y-(\xi'_1+u)(s-s'),\xi'')}{-(s-s')}. \label{kk_geq_e_3}
\end{align}
For the last term, we apply $\Vert wg\Vert_{L^\infty_{x,\xi}}<\infty$ to have $\eqref{kk_geq_e_3} = 0$.

For $\p_{\xi'_1}[\mathbf{k}(\xi,\xi')\mathbf{k}(\xi',\xi'')]$ in~\eqref{kk_geq_e_1}, we use $2\tilde{\theta}<\theta$ to have
\begin{align}
    &    \frac{1}{w_{\tilde{\theta}}(\xi)}\int^t_{\max\{0,\mathbf{1}_{\xi_1+u>0} (t-\frac{x}{\xi_1+u})\}} \dd s e^{-\nu_0(\xi)(t-s)/2} \int_{\mathbb{R}^3} \dd \xi'   \int^s_{\max\{0,\mathbf{1}_{\xi_1'+u>0}(s-\frac{y}{\xi_1'+u})\}}\dd s' e^{-\nu_0(\xi')(s-s')/2} \notag \\
    &\times   \frac{w_{\tilde{\theta}}(\xi)}{w_{2\tilde{\theta}}(\xi')}    \int_{\mathbb{R}^3} \dd \xi''  \Big|\p_{\xi'_1} \mathbf{k}(\xi,\xi') \mathbf{k}(\xi',\xi'') + \mathbf{k}(\xi,\xi') \p_{\xi'_1}\mathbf{k}(\xi',\xi'') \Big|
       \frac{w_{2\tilde{\theta}}(\xi')}{w_{2\tilde{\theta}}(\xi'')} \frac{\Vert wg\Vert_{L^\infty_{x,\xi}}}{\e}  \notag\\
    &\lesssim \frac{\Vert wg\Vert_{L^\infty_{x,\xi}}}{\e w_{\tilde{\theta}}(\xi)} \lesssim t^4 \frac{\Vert wg\Vert_{L^\infty_{x,\xi}}}{w_{\tilde{\theta}}(\xi) \alpha(x,\xi)}. \label{kk_geq_e_1_bdd}
\end{align}
In the last inequality we applied~\eqref{alpha_bdd} and~\eqref{e_t}. To get the first inequality in the last line, we have applied \eqref{k_theta} and \eqref{nabla_k_theta} to have
\begin{align*}
    & \frac{w_{\tilde{\theta}}(\xi)}{w_{2\tilde{\theta}}(\xi')}      \Big|\p_{\xi'_1} \mathbf{k}(\xi,\xi') \mathbf{k}(\xi',\xi'') + \mathbf{k}(\xi,\xi') \p_{\xi'_1}\mathbf{k}(\xi',\xi'') \Big|    \frac{w_{2\tilde{\theta}}(\xi')}{w_{2\tilde{\theta}}(\xi'')}          \\
    & \lesssim \frac{[1+|\xi|']^2}{w_{\tilde{\theta}}(\xi')} \Big| \frac{e^{-C_\theta|\xi-\xi'|^2}}{|\xi-\xi'|^2} \frac{e^{-C_\theta|\xi'-\xi''|^2}}{|\xi'-\xi''|} + \frac{e^{-C_\theta|\xi-\xi'|^2}}{|\xi-\xi'|} \frac{e^{-C_\theta|\xi'-\xi''|^2}}{|\xi'-\xi''|^2} \Big| \in L^1(\dd \xi' \dd \xi'').
\end{align*}

For the other term $\p_{\xi'_1}e^{-\bar{\nu}(\xi')(s-s')}$ in~\eqref{kk_geq_e_1}, recall the definition of $\bar{\nu}$ in \eqref{nu_bar}, we apply \eqref{nablav nu} to bound it by \eqref{kk_geq_e_1_bdd} using the same computation.

Then we compute~\eqref{kk_geq_e_2}. The $\xi_1'$ derivative on $y/(\xi_1'+u)$ can be combined with the extra $1/(y/(\xi_1'+u))$, we apply~\eqref{alpha_bdr} and Lemma \ref{lemma:velocity_alpha} to have
\begin{align*}
    &  \Big|\p_{\xi_1'}[\frac{y}{\xi'_1+u}] \frac{\xi_1'+u}{y} \Big| =  \frac{1}{|\xi_1'+u|} \leq  \frac{1}{\alpha(0,\xi')}\lesssim \frac{e^{ \frac{\nu_0 y}{4(\xi_1'+u)}}}{\alpha(y,\xi')}.
\end{align*}
Then we proceed the computation as
\begin{align}
  |\eqref{kk_geq_e_2}|  &  \lesssim \frac{1}{w_{\tilde{\theta}}(\xi)}\int^t_{\max\{0,\mathbf{1}_{\xi_1+u>0} (t-\frac{x}{\xi_1+u})\}} \dd s e^{-\nu_0(\xi)(t-s)/2} \int_{\mathbb{R}^3} \dd \xi' \mathbf{k}(\xi,\xi') \frac{w_{\tilde{\theta}}(\xi)}{w_{\tilde{\theta}}(\xi')}\notag \\
  &\times e^{-\frac{\nu_0 y}{4(\xi_1'+u)}}  \frac{\Vert wg\Vert_{L^\infty_{x,\xi}}}{\alpha(y,\xi')} \int_{\mathbb{R}^3} \dd \xi'' \mathbf{k}(\xi',\xi'') \frac{w_{\tilde{\theta}}(\xi')}{w_{\tilde{\theta}}(\xi'')} \lesssim  \frac{t \Vert wg\Vert_{L^\infty_{x,\xi}}}{w_{\tilde{\theta}}(\xi) \alpha(x,\xi)}. \label{kk_geq_e_3_bdd}
\end{align}
In the last inequality we applied \eqref{nln_large} in Lemma \ref{lemma:NLN} with \eqref{k_theta_bdd}.

Collecting~\eqref{kk_leq_e_bdd}, \eqref{kk_geq_e_1_bdd} and~\eqref{kk_geq_e_3_bdd} we conclude
\begin{align*}
  |\eqref{kk}|  &  \lesssim \frac{\frac{1}{t}\Vert w_{\tilde{\theta}}\alpha\p_x  g\Vert_{L^\infty_{x,\xi}} + t^4[\Vert wg_b\Vert_{L^\infty_{\xi}} +\Vert wg\Vert_{L^\infty_{x,\xi}}]}{w_{\tilde{\theta}}(\xi)\alpha(x,\xi)}. 
\end{align*}
This, together with~\eqref{K_123456_bdd}, leads to the estimate 
\begin{align}
  \eqref{p_x_4}  & \lesssim  \frac{[\frac{1}{t}+te^{-\nu_0 t/4} + t^2\gamma]\Vert w_{\tilde{\theta}}\alpha\p_x g\Vert_{L^\infty_{x,\xi}} + t^4[\Vert wg_b\Vert_{L^\infty_\xi} + \Vert wg\Vert_{L^\infty_{x,\xi}} + \Vert w_{\tilde{\theta}}Q(0)\Vert_{L^\infty_\xi} + \eqref{aprior_deri}_*]}{w_{\tilde{\theta}}(\xi)\alpha(x,\xi)}  . \label{4_bdd}
\end{align}

\eqref{123_bdd} and~\eqref{4_bdd} leads to the estimate for $\p_x g(x,\xi)$:
\begin{align*}
  |\p_x g(x,\xi)|  & \lesssim  \eqref{4_bdd} .
\end{align*}
Applying~\eqref{t_large} to a fixed $t$, we choose a smaller $\gamma$ in \eqref{pen_lin_op} such that $t^2\gamma \ll 1$, we conclude the lemma by
\begin{align*}
  \Vert w_{\tilde{\theta}}\alpha \p_x g\Vert_{L^\infty_{x,\xi}}  & \lesssim t^4[\Vert wg_b\Vert_{L^\infty_\xi} + \Vert wg\Vert_{L^\infty_{x,\xi}} + \Vert w_{\tilde{\theta}}Q(0)\Vert_{L^\infty_\xi} + \eqref{aprior_deri}_*].
\end{align*}

\end{proof}

In the following lemma, we construct the derivative to the linearized penalized problem.
\begin{lemma}\label{lemma:linear_boltz_derivative}
Suppose $\p_x Q$ exists and satisfies $\eqref{aprior_deri}_*<\infty$, $\Vert w_{\tilde{\theta}}Q\Vert_{L^\infty_{x,\xi}}<\infty$, then the derivative of the unique solution to linearized penalized problem in \eqref{pen_prob} exists, and it is given by
\begin{align}
    (\xi_1 + u)\p_x (\p_x g) + \mathcal{L}^p (\p_x g) = \p_x Q. \notag
\end{align}
Moreover, the derivative satisfies \eqref{aprior_deri}.

\end{lemma}

\begin{proof}
We start by considering the derivative to the following equation:
\begin{equation}\label{eqn_g}
\begin{cases}
    & (\xi_1+u)\p_x g + [\nu(\xi) - \gamma (\xi_1+u)]g - \lambda K^p \tilde{g} = Q  ,\\
    & g(0,\xi) = g_b(\xi) , \ \ \xi_1+u>0.
\end{cases}
\end{equation}
\begin{align}
    & (\xi_1+u)\p_x (\p_x g) + [\nu(\xi) - \gamma (\xi_1+u)](\p_x g) - \lambda K^p (\p_x \tilde{g}) = \p_x Q.  \label{eqn_p_x_g}
\end{align}

We define a Banach space as
\begin{align}
   \mathcal{X}_2: = & \Big\{ g(x,\xi):    \Vert w_{\tilde{\theta}} \alpha \p_x g\Vert_{L^\infty_{x,\xi}}< \infty, \ \ \Vert w g\Vert_{L^\infty_{x,\xi}} < \infty,  \ \ \Vert g\Vert_{L^2_{x,\xi}} < \infty \Big \},  \notag
\end{align}
with associated norm
\begin{align}
  \Vert g\Vert_{\mathcal{X}_2}  &:= \Vert w_{\tilde{\theta}} \alpha \p_x g\Vert_{L^\infty_{x,\xi}}   +  \Vert g\Vert_{\mathcal{X}}  .  \notag
\end{align}
Here $\Vert g\Vert_{\mathcal{X}}$ is defined in \eqref{banach}.

For given $\tilde{g} \in \mathcal{X}_2 $, from Proposition \ref{prop:pen_continuous} there exists a unique solution $g$ to \eqref{eqn_g}. By Lemma \ref{lemma:weak_deri}, the derivative of $g$ exists and is given by \eqref{eqn_p_x_g}. 

We first prove that $\Vert w_{\tilde{\theta}} \alpha\p_x g \Vert_{L^\infty_{x,\xi}}<\infty$ so that we can apply the a priori estimate \eqref{aprior_deri}. From the proof of Lemma \ref{lemma:aprori_linear}, $\p_x g$ can be expressed by \eqref{p_x_1} - \eqref{p_x_6} with replacing $K^p g$ by $\lambda K^p (\tilde{g})$. \hide Then all given sources are bounded as $w_{\tilde{\theta}}(\xi)\alpha(x,\xi) \times [\eqref{p_x_2} + \cdots + \eqref{p_x_6}]_{Kg\to \lambda K\tilde{g}} < \infty$. \unhide

For \eqref{p_x_1}, when $\xi_1+u=0$, \eqref{p_x_1} can be absorbed by LHS. When $\xi_1+u\neq 0$, we have a uniform-in-$x$ bound for $\p_x g$ by the equation of $g$ in \eqref{pen_prob}:
\begin{align*}
    & \Vert w_{\tilde{\theta}}(\xi)\p_x g(x,\xi) \Vert_{L^\infty_x} \leq \frac{\Vert wg\Vert_{\lfty} + \Vert w_{\tilde{\theta}}Q\Vert_{\lfty} + \Vert w_{\tilde{\theta}} K^p (\tilde{g})\Vert_{\lfty}}{|\xi_1+u|}   \\
    & \lesssim \frac{\Vert wg\Vert_{\lfty} + \Vert w_{\tilde{\theta}}Q\Vert_{\lfty} + \Vert w\tilde{g}\Vert_{\lfty}}{|\xi_1+u|}.
\end{align*}
Here we have applied Lemma \ref{lemma:Kp}. From the uniform-in-$x$ bound, we can multiply $w_{\tilde{\theta}}(\xi)$ to \eqref{p_x_1} - \eqref{p_x_6}, and take $sup$ in $x$ to the LHS, so that $w_{\tilde{\theta}}(\xi)\times \eqref{p_x_1}$ can be absorbed by the LHS.

Since $\alpha(x,\xi)\leq 1$, we further multiply $\alpha(x,\xi)$ to \eqref{p_x_1} - \eqref{p_x_6}, and $\alpha(x,\xi)w_{\tilde{\theta}}(\xi)\times \eqref{p_x_1}$ can also be absorbed by the LHS. 

Without the contribution of $K^p g$ in \eqref{eqn_g}, we can apply the same argument as the a priori estimate in Lemma \ref{lemma:aprori_linear}, here we do not need to estimate \eqref{p_x_1}, \eqref{p_x_3} and \eqref{p_x_4}. Then we derive that \eqref{eqn_p_x_g} can be bounded by the same estimate as \eqref{aprior_deri}. Hence, for some $C_\alpha$,
\begin{align}
  \Vert w_{\tilde{\theta}} \alpha \p_x g\Vert_{L^\infty_{x,\xi}}  & \leq C_\alpha \Big[ \Vert w g\Vert_{L^\infty_{x,\xi}}  +  \Vert w_{\tilde{\theta}}Q(0)\Vert_{L^\infty_\xi} + \eqref{aprior_deri}_* + \eqref{kp_deri_bdd} \Big],    \label{p_x_g_0}
\end{align}
where \eqref{kp_deri_bdd} corresponds to the contribution of $\lambda K^p (\p_x \tilde{g})$:
\begin{align}
    & \lambda w_{\tilde{\theta}}(\xi) \Big[ K^p \tilde{g}(0,\xi) +\alpha(x,\xi)\int^t_{\max\{0,\mathbf{1}_{\xi_1+u>0}(t-\frac{x}{\xi_1+u})\}} \dd s e^{-\nu(\xi)(t-s)} K^p \p_x \tilde{g}(x-(\xi_1+u)(t-s),\xi)   \dd s\Big] \label{kp_deri_bdd}\\
    & \lesssim \lambda  \Vert w \tilde{g}\Vert_{L^\infty_{x,\xi}} + \lambda \gamma t\Vert w_{\tilde{\theta}} \alpha \p_x \tilde{g}\Vert_{L^\infty_{x,\xi}} \notag \\
    & + \lambda \Vert w_{\tilde{\theta}}\alpha \p_x \tilde{g}\Vert_{L^\infty_{x,\xi}} \alpha(x,\xi)\int^t_{\max\{0,\mathbf{1}_{\xi_1+u>0}(t-\frac{x}{\xi_1+u})\}} \dd s e^{-\nu(\xi)(t-s)} \int_{\mathbb{R}^3}\dd \xi' \frac{w_{\tilde{\theta}}(\xi)\mathbf{k}(\xi,\xi')}{w_{\tilde{\theta}}(\xi')\alpha(x-(\xi_1+u)(t-s),\xi')} \notag \\
    &\lesssim \lambda \gamma \Vert w \tilde{g}\Vert_{L^\infty_{x,\xi}} + \lambda \gamma t\Vert w_{\tilde{\theta}} \alpha \p_x \tilde{g}\Vert_{L^\infty_{x,\xi}} + \lambda t\Vert w_{\tilde{\theta}}\alpha \p_x \tilde{g}\Vert_{L^\infty_{x,\xi}}. \label{kp_deri_bdd2}
\end{align}
In the second line we have applied Lemma \ref{lemma:Kp} to $K^p \tilde{g}$ for the first term, for the second term, we applied the same computation as \eqref{K_other_bdd} to compute the contribution of $\gamma \prod_+ ((\xi_1+u)g) + \gamma \mathbf{p}_u ((\xi_1+u)g)$. In the last line we have used Lemma \ref{lemma:NLN} with \eqref{k_theta_bdd}.

Combining \eqref{p_x_g_0}, \eqref{kp_deri_bdd2}, $\Vert g\Vert_{\mathcal{X}}<\infty$ from \eqref{chi_bdd} and the assumption that $\eqref{aprior_deri}_*<\infty$, we conclude that $g \in \mathcal{X}_2$.

Next, for given $\tilde{g}_1,\tilde{g}_2$, we denote the associated solution as $g_1,g_2$. Then we apply \eqref{p_x_g_0}, \eqref{kp_deri_bdd2} and \eqref{contraction_X} to have
\begin{align}
  &\Vert g_1-g_2\Vert_{\mathcal{X}_2}   = \Vert w_{\tilde{\theta}}\alpha \p_x (g_1-g_2)\Vert_{L^\infty_{x,\xi}}   +   \Vert g_1-g_2\Vert_{\mathcal{X}} \notag \\
  & \leq C_\alpha \Vert w (g_1-g_2)\Vert_{L^\infty_{x,\xi}}+ [\lambda \gamma t+ \lambda t ] C_\alpha \Vert w_{\tilde{\theta}}\alpha \p_x (\tilde{g}_1-\tilde{g}_2)\Vert_{L_{x,\xi}^\infty} + C_\alpha \lambda \gamma\Vert w(\tilde{g}_1-\tilde{g}_2) \Vert_{L^\infty_{x,\xi}}+ C \lambda \Vert \tilde{g}_1-\tilde{g}_2\Vert_{\mathcal{X}} \notag \\
  & \leq  C_\alpha \lambda [\gamma t+ t] \Vert w_{\tilde{\theta}} \alpha \p_x (\tilde{g}_1 - \tilde{g}_2)\Vert_{L^\infty_{x,\xi}} + C_\alpha \Vert g_1-g_2\Vert_{\mathcal{X}} + C_\alpha \lambda[\gamma + C] \Vert \tilde{g}_1-\tilde{g}_2 \Vert_{\mathcal{X}} \notag \\
  & \leq C_\alpha \lambda [\gamma t+ t] \Vert w_{\tilde{\theta}} \alpha \p_x (\tilde{g}_1 - \tilde{g}_2)\Vert_{L^\infty_{x,\xi}} + C_\alpha \lambda[\gamma + 2C] \Vert \tilde{g}_1-\tilde{g}_2 \Vert_{\mathcal{X}} .\notag
\end{align}
Then we choose $\lambda$ to be sufficiently small such that all coefficients are bounded as
\begin{align*}
    & C_\alpha \lambda [\gamma t+ t]< 1 \ \ \ C_\alpha \lambda[\gamma + 2C]< 1.
\end{align*}
By Banach fixed point theorem, there exists a unique $g\in \mathcal{X}_2$ satisfying
\begin{equation}
\begin{cases}
    & (\xi_1+u)\p_x g + [\nu(\xi) - \gamma (\xi_1+u)]g - \lambda K^p g = Q , \notag\\
    & g(0,\xi) = g_b(\xi) , \ \ \xi_1+u>0, \notag\\
    & (\xi_1+u)\p_x (\p_x g) + [\nu(\xi) - \gamma (\xi_1+u)](\p_x g) - \lambda K^p (\p_x g) = \p_x Q.  \label{linear_p_x_g}
\end{cases}
\end{equation}

From the well-posedness of \eqref{linear_p_x_g}, next we consider the derivative to the solution of the following problem,
\begin{equation}
\begin{cases}
    & (\xi_1+u)\p_x g + [\nu(\xi) - \gamma (\xi_1+u)]g - \lambda K^p g = \lambda K^p (\tilde{g}) + Q , \notag\\
    & g(0,\xi) = g_b(\xi) , \ \ \xi_1+u>0, \notag\\
    & (\xi_1+u)\p_x (\p_x g) + [\nu(\xi) - \gamma (\xi_1+u)](\p_x g) - \lambda K^p (\p_x g) = \lambda K^p(\p_x \tilde{g}) +  \p_x Q.  \notag
\end{cases}
\end{equation}
With given source $\tilde{g}\in \mathcal{X}_2$, from previous argument we know that $\p_x g$ exists and $g\in \mathcal{X}_2$. It is straightforward to show that one can obtain the same uniform in $\lambda$ a-priori estimate \eqref{aprior_deri}. Thus \eqref{p_x_g_0} still holds. Then we apply the same fixed-point argument to show that there exists a unique $g$ satisfying
\begin{equation}
\begin{cases}
    & (\xi_1+u)\p_x g + [\nu(\xi) - \gamma (\xi_1+u)]g - 2\lambda K^p g =  Q  \notag\\
    & g(0,\xi) = g_b(\xi) , \ \ \xi_1+u>0, \notag\\
    & (\xi_1+u)\p_x (\p_x g) + [\nu(\xi) - \gamma (\xi_1+u)](\p_x g) - 2\lambda K^p (\p_x g) =   \p_x Q.  \notag
\end{cases}
\end{equation}
Step by step, we construct solution with coefficient $1$ for $K^p (\p_x g)$, then we conclude the lemma.

\end{proof}

\ \\

\subsection{Weighted $C^1$ estimate of nonlinear penalized problem}

Next, we construct the derivative to the nonlinear problem \eqref{pen_nonlin_prob}.

\begin{proposition}
The derivative to the solution of \eqref{pen_nonlin_prob} in Proposition \ref{prop:nonlinear_continuous} exists. Moreover, we have the weighted $C^1$ estimate
\begin{align}
 \Vert \alpha\p_x h\Vert_{L^\infty_{x,\xi}} + \Vert w_{\tilde{\theta}}\alpha \p_x g\Vert_{L^\infty_{x,\xi}}  & \leq C\Vert w g\Vert_{L^\infty_{x,\xi}} +   C\Vert w g_b\Vert_{L^\infty_{\xi}} + C\Vert h\Vert_{L^\infty_{x}}.  \label{penalized_C1}
\end{align}

\end{proposition}

\begin{proof}
We denote a Banach space as
\begin{equation}\label{banach_nonlinear_deri}
\begin{split}
  \mathcal{X}_3  & := \Big\{ (g,h): \Vert w_{\tilde{\theta}} \alpha \p_x g\Vert_{L_{x,\xi}^\infty}  + \Vert \alpha \p_x h\Vert_{L^\infty_{x,\xi}}  +  \Vert (g,h)\Vert_{\mathcal{X}_1} \leq 2C\e      \Big\}   ,
\end{split}
\end{equation}
with the associated norm defined as
\begin{equation*}
\begin{split}
  \Vert (g,h)\Vert_{\mathcal{X}_3}  &  = \Vert w_{\tilde{\theta}}\alpha \p_x g\Vert_{L^\infty_{x,\xi}} + \Vert \alpha \p_x h\Vert_{L^\infty_{x,\xi}} + \Vert (g,h)\Vert_{\mathcal{X}_1}.
\end{split}
\end{equation*}
Here, $\mathcal{X}_1$ is defined in \eqref{nonlinear_banach}. Clearly we have $\Vert (g,h)\Vert_{\mathcal{X}_1}\leq \Vert (g,h)\Vert_{\mathcal{X}_3}$.

Given $(\tilde{g}, \tilde{h})\in \mathcal{X}_3$, we consider the following linear problem with given sources as $\tilde{g},\tilde{h}$:
\begin{equation*}
\begin{cases}
    &(\xi_1+u)\p_x g + \mathcal{L}^p g = e^{-\gamma x} (\mathbf{I}-\mathbf{P}_u)\Gamma(\tilde{g}-\tilde{h}\phi_u, \tilde{g}-\tilde{h}\phi_u)  , \\
    &h(x) = -e^{-\gamma x} \int_0^\infty e^{(\tau_u - 2\gamma)z} \langle \psi_u \Gamma(\tilde{g}-\tilde{h}\phi_u,\tilde{g}-\tilde{h}\phi_u)\rangle(x+z) \dd z,   \\
    & g(0,\xi) = f_b(\xi) + h(0)\phi_u(\xi), \ \ \xi_1+u>0, \\
    & (\xi_1+u) \p_x (\p_x g) + \mathcal{L}^p (\p_x g) = \p_x \Big[e^{-\gamma x}(\mathbf{I}-\mathbf{P}_u)\Gamma(\tilde{g}-\tilde{h}\phi_u, \tilde{g}-\tilde{h}\phi_u) \Big], \\
    & \p_x h(x) = -\p_x \Big[ e^{-\gamma x} \int_0^\infty e^{(\tau_u - 2\gamma)z} \langle \psi_u \Gamma(\tilde{g}-\tilde{h}\phi_u,\tilde{g}-\tilde{h}\phi_u)\rangle(x+z) \dd z \Big].
\end{cases}
\end{equation*}
$h(x)$ can also be expressed as
\begin{align*}
   h(x) &  = -e^{-\gamma x} \int_{x}^\infty e^{(\tau_u - 2\gamma)(z-x)} \langle \psi_u \Gamma(\tilde{g}-\tilde{h}\phi_u, \tilde{g}-\tilde{h}\phi_u) \rangle(z) \dd z \\
   & = -e^{(\gamma - \tau_u)x} \int_x^\infty  e^{(\tau_u-2\gamma)z} \langle \psi_u \Gamma(\tilde{g}-\tilde{h}\phi_u, \tilde{g}-\tilde{h}\phi_u) \rangle (z) \dd z.
\end{align*}
Then we take derivative to have 
\begin{align}
    & \p_x h(x)  = (\gamma - \tau_u)h(x) + e^{-\gamma x}\langle \psi_u  \Gamma(\tilde{g}-\tilde{h}\phi_u, \tilde{g}-\tilde{h}\phi_u)\rangle(x)  \label{h_eqn}  .
\end{align}
From the equation of $h$ in \eqref{h_eqn}, we apply \eqref{alpha_bdd} to have
\begin{align}
  \Vert \alpha \p_x h\Vert_{L^\infty_{x,\xi}}  & \leq \Vert  \p_x h \Vert_{L^\infty_x } \lesssim \Vert h\Vert_{L^\infty_{x}} + \Vert (\tilde{g},\tilde{h})\Vert_{\mathcal{X}_1}^2 \lesssim \Vert (\tilde{g},\tilde{h})\Vert_{\mathcal{X}_1}^2 . \label{p_x_h_bdd}
\end{align}
In the second inequality, we applied \eqref{h_bdd} to the source term $\Gamma$. In the last inequality, we used $\Vert h\Vert_{L^\infty_x} \lesssim \Vert (\tilde{g},\tilde{h})\Vert_{\mathcal{X}_1}^2$ from \eqref{h_bdd}.

The well-posedness of $\p_x g$ is guaranteed by Lemma \ref{lemma:linear_boltz_derivative}, and $\p_x g$ satisfies \eqref{aprior_deri}. The boundary term in \eqref{aprior_deri} can be controlled as
\begin{align}
    & \Vert wg_b\Vert_{L^\infty_\xi} + \Vert w_{\tilde{\theta}}(\xi)Q(0,\xi)\Vert_{L^\infty_\xi} 
    \lesssim \Vert wf_b\Vert_{L^\infty_\xi} + \Vert w\phi_u\Vert_{L^\infty_\xi}\Vert h\Vert_{L^\infty_x} + \Vert (\tilde{g},\tilde{h})\Vert_{\mathcal{X}_1}^2 \lesssim \e + \Vert (\tilde{g},\tilde{h})\Vert_{\mathcal{X}_1}^2. \label{p_x_bdr_term}
\end{align}
Here we applied \eqref{gamma_bdd} and \eqref{P_u_gamma_infty} with $w_{\tilde{\theta}}(\xi) \lesssim \frac{w(\xi)}{1+|\xi|}$ to $Q(0) = (\mathbf{I}-\mathbf{P}_u)\Gamma(\tilde{g}-\tilde{h}\phi_u,\tilde{g}-\tilde{h}\phi_u)(0)$ in the first inequality. In the second inequality, $\e$ comes from the assumption $\Vert wf_b\Vert_{L^\infty}\leq \e$ in Proposition \ref{prop:nonlinear_continuous}, and we used \eqref{h_bdd} to $h$.

Then we evaluate $\eqref{aprior_deri}_*$ with $Q=e^{-\gamma x} (\mathbf{I}-\mathbf{P}_u) \Gamma(\tilde{g}-\tilde{h}\phi_u,\tilde{g}-\tilde{h}\phi_u)$. Note that
\begin{align}
  \p_x Q  &
 = -\gamma e^{-\gamma x} (\mathbf{I}-\mathbf{P}_u)\Gamma(\tilde{g}-\tilde{h}\phi_u, \tilde{g} - \tilde{h}\phi_u) + e^{-\gamma x}(\mathbf{I}-\mathbf{P}_u)\p_x \Gamma(\tilde{g}-\tilde{h}\phi_u, \tilde{g} - \tilde{h}\phi_u).    \label{p_x_Q}    \end{align}

The contribution of the first term on RHS of \eqref{p_x_Q} is bounded as
\begin{align}
    &w_{\tilde{\theta}}(\xi)\int^t_{\max\{0,\mathbf{1}_{\xi_1+u>0}(t-\frac{x}{\xi_1+u})\}} \dd s e^{-\nu(\xi)(t-s)} (\mathbf{I}-\mathbf{P}_u)\Gamma(\tilde{g}-\tilde{h}\phi_u,\tilde{g}-\tilde{h}\phi_u)(x-(\xi_1+u)(t-s),\xi) \notag  \\
    &\lesssim   \big\Vert \frac{w}{[1+|\xi|]} (\mathbf{I}-\mathbf{P}_u)\Gamma\big\Vert_{L^\infty_{x,\xi}} \int^t_{\max\{0,\mathbf{1}_{\xi_1+u>0}(t-\frac{x}{\xi_1+u})\}} \dd s e^{-\nu(\xi)(t-s)}   \lesssim \Vert (\tilde{g},\tilde{h})\Vert_{\mathcal{X}_1}^2. \label{p_x_expo_gamma_bdd}
\end{align}
Here we have applied the same computation in \eqref{g_bdd} with $w_{\tilde{\theta}}(\xi)\lesssim \frac{w(\xi)}{[1+|\xi|]}$.

For $\eqref{aprior_deri}_*$, the contribution of $I\p_x \Gamma$ in the second term of RHS of \eqref{p_x_Q} is bounded as
\begin{align*}
    & w_{\tilde{\theta}}(\xi)\alpha(x,\xi)\int^t_{\max\{0,\mathbf{1}_{\xi_1+u>0}(t-\frac{x}{\xi_1+u})\}} \dd s e^{-\bar{\nu}(\xi)(t-s)} | \p_x \Gamma( \tilde{g} -  \tilde{h}\phi_u,  \tilde{g} -  \tilde{h}\phi_u)(x-(\xi_1+u)(t-s),\xi) |   .
\end{align*}
We apply \eqref{Gamma_est}. The contribution of the first term in \eqref{Gamma_est} is bounded as
\begin{align}
    &  \int^t_{\max\{0,\mathbf{1}_{\xi_1+u>0}(t-\frac{x}{\xi_1+u})\}} \dd s e^{-\bar{\nu}(\xi)(t-s)} [1+|\xi|] \Vert w(\tilde{g}-\phi_u \tilde{h})\Vert_{L^\infty_{x,\xi}} \notag \\
    &\times w_{\tilde{\theta}}(\xi)\alpha(x,\xi)|[\p_x \tilde{g}-\p_x \tilde{h} \phi_u](x-(\xi_1+u)(t-s),\xi)| \notag \\
    &\leq  \Vert (\tilde{g},\tilde{h})\Vert_{\mathcal{X}_1} [\Vert w_{\tilde{\theta}}\alpha \p_x \tilde{g}\Vert_{L^\infty_{x,\xi}}+\Vert \alpha \p_x \tilde{h}\Vert_{L^\infty_{x}}] \int^t_{\max\{0,\mathbf{1}_{\xi_1+u>0}(t-\frac{x}{\xi_1+u})\}} \dd s e^{-\nu_0(\xi)(t-s)/2} [1+|\xi|]  e^{\nu_0(t-s)/4} \notag \\
    &\leq \Vert (\tilde{g},\tilde{h})\Vert_{\mathcal{X}_1} [\Vert w_{\tilde{\theta}} \alpha \p_x \tilde{g}\Vert_{L^\infty_{x,\xi}} + \Vert \alpha \p_x \tilde{h}\Vert_{L^\infty_{x}}] \int^t_{\max\{0,\mathbf{1}_{\xi_1+u>0}(t-\frac{x}{\xi_1+u})\}} \dd s e^{-\nu_0(\xi)(t-s)/4} [1+|\xi|] \notag\\
    &\lesssim \Vert (\tilde{g},\tilde{h})\Vert_{\mathcal{X}_1} [\Vert w_{\tilde{\theta}}\alpha \p_x \tilde{g}\Vert_{L^\infty_{x,\xi}} + \Vert \alpha \p_x \tilde{h}\Vert_{L^\infty_{x}}]. \label{p_x_gamma_bdd}
\end{align}
In the third line we have used $\Vert w_{\tilde{\theta}}(\xi)[1+|\xi|]\phi_u\Vert_{L^\infty_\xi} \lesssim \Vert w\phi_u\Vert_{L^\infty_\xi}\lesssim 1$ and applied Lemma \ref{lemma:velocity_alpha}.

The contribution of the second and third term in \eqref{Gamma_est} are bounded as
\begin{align}
    & w_{\tilde{\theta}}(\xi) \alpha(x,\xi)\Vert w(\tilde{g}-\phi_u \tilde{h})\Vert_{L^\infty_{x,\xi}} \int^t_{\max\{0,\mathbf{1}_{\xi_1+u>0}(t-\frac{x}{\xi_1+u})\}} \dd s e^{-\bar{\nu}(\xi)(t-s)}  \notag \\
    &\times \int_{\mathbb{R}^3} \dd \xi' \mathbf{k}(\xi,\xi') |\p_x [\tilde{g}-\tilde{h}\phi_u](x-(\xi_1+u)(t-s),\xi') | \notag \\
    &\leq  \Vert (\tilde{g},\tilde{h})\Vert_{\mathcal{X}_1} [\Vert w_{\tilde{\theta}}\alpha \p_x \tilde{g}\Vert_{L^\infty_{x,\xi}}+\Vert \alpha \p_x \tilde{h}\Vert_{L^\infty_{x}}] \alpha(x,\xi)\int^t_{\max\{0,\mathbf{1}_{\xi_1+u>0}(t-\frac{x}{\xi_1+u})\}} \dd s e^{-\nu_0(\xi)(t-s)/2} \notag \\
    & \times \int_{\mathbb{R}^3} \frac{w_{\tilde{\theta}}(\xi) \mathbf{k}(\xi,\xi')}{w_{\tilde{\theta}}(\xi')\alpha(x-(\xi_1+u)(t-s),\xi')} \dd \xi' \notag \\
    &\lesssim t \Vert (\tilde{g},\tilde{h})\Vert_{\mathcal{X}_1} [\Vert w_{\tilde{\theta}}\alpha \p_x \tilde{g}\Vert_{L^\infty_{x,\xi}} + \Vert \alpha \p_x \tilde{h}\Vert_{L^\infty_{x}}]  \label{p_x_gamma_bdd_2}.
\end{align}
In the fourth line we have used Lemma \ref{lemma:eigen_continuous} to have $\phi_u(\xi')\leq \frac{\Vert w\phi_u\Vert_{L^\infty_\xi}}{w_{\tilde{\theta}}(\xi')}$. In the last line applied Lemma \ref{lemma:NLN} with \eqref{k_theta_bdd}.

For $\eqref{aprior_deri}_*$, the contribution of $\mathbf{P}_u \p_x \Gamma$ in the second term of the RHS of \eqref{p_x_Q} is controlled as
\begin{align*}
    &   \alpha(x,\xi) \int^t_{\max\{0,\mathbf{1}_{\xi_1+u>0}(t-\frac{x}{\xi_1+u})\}} \dd s e^{-\bar{\nu}(\xi)(t-s)} \Vert w_{\tilde{\theta}}(\xi_1+u)\phi_u\Vert_{L^\infty_\xi} \Vert w\psi_u\Vert_{L^\infty_\xi} \\
    &\times \int_{\mathbb{R}^3} w^{-1}(\xi') \big|\p_x \Gamma(\tilde{g}-\tilde{h}\phi_u,\tilde{g}-\tilde{h}\phi_u)(x-(\xi_1+u)(t-s),\xi') \big| \dd \xi'.
\end{align*}

Again we apply \eqref{Gamma_est}. The contribution of the first term in \eqref{Gamma_est} is bounded as
\begin{align}
    &  \alpha(x,\xi)\Vert w\phi_u\Vert_{L^\infty_\xi} \Vert w\psi_u\Vert_{L^\infty_\xi} \int^t_{\max\{0,\mathbf{1}_{\xi_1+u>0}(t-\frac{x}{\xi_1+u})\}} \dd s e^{-\nu_0(\xi)(t-s)/2}  \notag\\
    & \times \int_{\mathbb{R}^3} w^{-1}(\xi') [1+|\xi|']  \Vert w(\tilde{g}-\tilde{h}\phi_u)\Vert_{L^\infty_{x,\xi}} [\p_x \tilde{g}-\p_x \tilde{h} \phi_u](x-(\xi_1+u)(t-s),\xi') \dd \xi'  \notag\\
    & \lesssim \alpha(x,\xi)\Vert (\tilde{g},\tilde{h})\Vert_{\mathcal{X}_1} [\Vert \alpha \p_x \tilde{g}\Vert_{\lfty} + \Vert \alpha \p_x \tilde{h}\Vert_{\lfty}] \notag\\
    &\times \int^t_{\max\{0,\mathbf{1}_{\xi_1+u>0}(t-\frac{x}{\xi_1+u})\}} \dd s e^{-\nu_0(\xi)(t-s)/2} \int_{\mathbb{R}^3} \frac{w^{-1/2}(\xi')}{\alpha(x-(\xi_1+u)(t-s),\xi')} \dd \xi' \notag \\
    &\lesssim t\Vert (\tilde{g},\tilde{h})\Vert_{\mathcal{X}_1} [\Vert w_{\tilde{\theta}}\alpha \p_x \tilde{g}\Vert_{\lfty} + \Vert \alpha \p_x \tilde{h}\Vert_{\lfty}].  \label{p_x_gamma_P_1}
\end{align}
In the last line we have used Lemma \ref{lemma:NLN_inner}.

The contribution of the second and third term in \eqref{Gamma_est} are bounded as
\begin{align}
    &    \alpha(x,\xi)\Vert (\tilde{g},\tilde{h})\Vert_{\mathcal{X}_1} \int^t_{\max\{0,\mathbf{1}_{\xi_1+u>0}(t-\frac{x}{\xi_1+u})\}} \dd s e^{-\nu_0(\xi)(t-s)/2} \notag \\
    &   \times \int_{\mathbb{R}^3} \dd \xi' w^{-1}(\xi') \int_{\mathbb{R}^3} \dd \xi'' \mathbf{k}(\xi',\xi'') [|\p_x \tilde{g}| + |\p_x \tilde{h}| ](x-(\xi_1+u)(t-s),\xi'') \notag \\
    & \lesssim   \alpha(x,\xi)\Vert (\tilde{g},\tilde{h})\Vert_{\mathcal{X}_1} [\Vert \alpha \p_x \tilde{g}\Vert_{\lfty} + \Vert \alpha \p_x \tilde{h}\Vert_{\lfty}] \notag \\
    &\times     \int^t_{\max\{0,\mathbf{1}_{\xi_1+u>0}(t-\frac{x}{\xi_1+u})\}} \dd s e^{-\nu_0(\xi)(t-s)/2} \int_{\mathbb{R}^3} w^{-1}(\xi') \dd \xi' \int_{\mathbb{R}^3} \frac{\mathbf{k}(\xi',\xi'')}{\alpha(x-(\xi_1+u)(t-s),\xi'')} \dd \xi''  \notag \\
    & \lesssim t\Vert (\tilde{g},\tilde{h})\Vert_{\mathcal{X}_1} [\Vert w_{\tilde{\theta}}\alpha \p_x \tilde{g}\Vert_{\lfty} + \Vert \alpha \p_x \tilde{h}\Vert_{\lfty}]. \label{p_x_gamma_P_2}
\end{align}
In the last line we have applied \eqref{NLN_two} in Lemma \ref{lemma:NLN_inner}.

Collecting \eqref{p_x_expo_gamma_bdd}, \eqref{p_x_gamma_bdd}, \eqref{p_x_gamma_bdd_2}, \eqref{p_x_gamma_P_1} and \eqref{p_x_gamma_P_2}, we conclude that the contribution of $\eqref{aprior_deri}_*$ is bounded as
\begin{align}
    \eqref{aprior_deri}_* \mathbf{1}_{Q=e^{-\gamma x}(\mathbf{I}-\mathbf{P}_u)\Gamma(\tilde{g}-\tilde{h}\phi_u,\tilde{g}-\tilde{h}\phi_u)} \lesssim t \Vert (\tilde{g},\tilde{h})\Vert_{\mathcal{X}_3}^2. \label{deri_gamma_bdd}
\end{align}

It has been proved in \eqref{g_h_bdd} that given $\Vert (\tilde{g},\tilde{h})\Vert_{\mathcal{X}_1}\leq 2C\e$, then $\Vert (g,h)\Vert_{\mathcal{X}_1} \leq 2C \e$. This, combining with the estimate of $\p_x h$ in \eqref{p_x_h_bdd}, the estimate for $\p_x g$ in \eqref{aprior_deri}, \eqref{p_x_bdr_term} and \eqref{deri_gamma_bdd}, implies that for some $C_1=C_1(t)>0$,
\begin{align*}
 \Vert (g,h)\Vert_{\mathcal{X}_3}   &  \leq  \Vert (g,h)\Vert_{\mathcal{X}_1} + \eqref{p_x_h_bdd} + \eqref{p_x_bdr_term} + \eqref{deri_gamma_bdd} \\
    & \leq C_1 \e + C_1 \Vert (\tilde{g},\tilde{h})\Vert_{\mathcal{X}_3}^2 \leq C_1 \e + 4C_1 C^2 \e^2 \leq 2C\e.
\end{align*}
Here we have taken $C$ in \eqref{banach_nonlinear_deri} to be $C=C_1$ and let $\e$ be small enough such that $C_1C^2\e^2 = C_1^3 \e^2  \leq C\e$. Therefore, we conclude that $(g,h)\in\mathcal{X}_3$.

Next we prove the contraction property. Given $(\tilde{g}_1,\tilde{h}_1),(\tilde{g}_2,\tilde{h}_2)\in \mathcal{X}_3$, we denote the corresponding solutions as $(g_1,h_1),(g_2,h_2)$. Then $(g_1-g_2,h_1-h_2)$ satisfies \eqref{nonlinear_difference}, and the derivative satisfies
\begin{equation*}
\begin{cases}
    &  (\xi_1+u)\p_x (\p_x (g_1-g_2)) + \mathcal{L}^p (\p_x (g_1-g_2)) = \p_x \big[e^{-\gamma x}(\mathbf{I}-\mathbf{P}_u)\Sigma \big], \\
    & \p_x[h_1-h_2](x) = -\p_x \Big[e^{-\gamma x}\int_0^\infty e^{(\tau_u-2\gamma)z} \langle \psi_u \Sigma \rangle (x+z) \dd z \Big],    
\end{cases} 
\end{equation*}
here $\Sigma$ is defined in \eqref{nonlinear_difference}.

To estimate $h_1-h_2$, we apply the same computation as \eqref{p_x_h_bdd}, with \eqref{h_12_bdd}, we obtain
\begin{align}
    & \Vert \alpha \p_x (h_1-h_2)\Vert_{L^\infty_{x,\xi}}\leq \Vert \p_x (h_1-h_2)\Vert_{L^\infty_{x,\xi}} \lesssim \Vert h_1-h_2\Vert_{L^\infty_x} + \Vert e^{-\gamma x} \langle \psi \Sigma\rangle\Vert_{L^\infty_{x,\xi}} \notag  \\
    &\lesssim \big\Vert \frac{w}{[1+|\xi|]}\Sigma \big\Vert_{L^\infty_{x,\xi}} \lesssim \e \Vert (\tilde{g}_1-\tilde{g}_2,\tilde{h}_1-\tilde{h}_2)\Vert_{\mathcal{X}_1} .\label{p_x_h_difference_bdd}
\end{align}

To estimate $\p_x (g_1-g_2)$, we apply \eqref{aprior_deri}. First we compute the boundary term with $Q(0)=(\mathbf{I}-\mathbf{P}_u)\Sigma(0)$ as
\begin{align}
    & \Vert w(g_1-g_2)_b\Vert_{L^\infty_\xi} + \Vert Q(0)\Vert_{L^\infty_\xi}\lesssim  \Vert w\phi_u\Vert_{L^\infty_\xi} \Vert h_1-h_2\Vert_{L^\infty_x} + \Vert (\mathbf{I}-\mathbf{P}_u)\Sigma(0)\Vert_{L^\infty_\xi} \notag\\
    & \lesssim  \e \Vert (\tilde{g}_1-\tilde{g}_2 , \tilde{h}_1-\tilde{h}_2)\Vert_{\mathcal{X}_1} + \big\Vert \frac{w}{[1+|\xi|]}\Sigma \big\Vert_{L^\infty_{x,\xi}} + \big\Vert \frac{w}{[1+|\xi|]}\mathbf{P}_u\Sigma \big\Vert_{L^\infty_{x,\xi}} \lesssim \e \Vert (\tilde{g}_1-\tilde{g}_2,\tilde{h}_1-\tilde{h}_2)\Vert_{\mathcal{X}_1}. \label{p_x_diff_bdr}
\end{align}
Here we applied the same estimate for $h_1-h_2$ in \eqref{h_12_bdd} and the same estimate for $\Sigma$ in \eqref{g_12_bdd}.

For $\eqref{aprior_deri}_*$, we express the derivative as
\begin{equation}\label{deri_Sigma}
-\gamma e^{-\gamma x}(\mathbf{I}-\mathbf{P}_u)\Sigma + e^{-\gamma x} (\mathbf{I}-\mathbf{P}_u) (\p_x \Sigma).
\end{equation}
We apply the same estimate as \eqref{p_x_expo_gamma_bdd} to the first term in \eqref{deri_Sigma}, and apply the same estimate as \eqref{deri_gamma_bdd} to the second term in \eqref{deri_Sigma}, this yields
\begin{align}
    &\eqref{aprior_deri}_* \mathbf{1}_{Q=e^{-\gamma x}(\mathbf{I}-\mathbf{P}_u)\Sigma}  \lesssim  [\Vert (\tilde{g}_1,\tilde{h}_1) \Vert_{\mathcal{X}_1} + \Vert (\tilde{g}_2,\tilde{h}_2)\Vert_{\mathcal{X}_1}] \Vert (\tilde{g}_1-\tilde{g}_2 , \tilde{h}_1-\tilde{h}_2)\Vert_{\mathcal{X}_1} \notag\\
  & +  t [\Vert (\tilde{g}_1,\tilde{h}_1) \Vert_{\mathcal{X}_3} + \Vert (\tilde{g}_2,\tilde{h}_2)\Vert_{\mathcal{X}_3}]\Vert (\tilde{g}_1-\tilde{g}_2,\tilde{h}_1-\tilde{h}_2)\Vert_{\mathcal{X}_3}   \notag \\
 & \lesssim t\e \Vert (\tilde{g}_1-\tilde{g}_2,\tilde{h}_1-\tilde{h}_2)\Vert_{\mathcal{X}_3}.    \label{p_x_Sigma_bdd}
\end{align}

Finally, we collect \eqref{p_x_h_difference_bdd}, \eqref{aprior_deri}, \eqref{p_x_diff_bdr}, \eqref{p_x_Sigma_bdd} and \eqref{chi_1_contraction} to conclude that
\begin{align*}
    & \Vert (g_1-g_2,h_1-h_2)\Vert_{\mathcal{X}_3} = \Vert w_{\tilde{\theta}}\alpha \p_x (g_1-g_2)\Vert_{L^\infty_{x,\xi}} + \Vert \alpha \p_x (h_1-h_2)\Vert_{L^\infty_x} + \Vert (g_1-g_2,h_1-h_2) \Vert_{\mathcal{X}_1} \\
    & \lesssim \Vert w(g_1-g_2)\Vert_{L^\infty_{x,\xi}}  + t\e \Vert (\tilde{g}_1-\tilde{g}_2,\tilde{h}_1-\tilde{h}_2)\Vert_{\mathcal{X}_3} \\
& \leq C_2\e \Vert (\tilde{g}_1-\tilde{g}_2,\tilde{h}_1-\tilde{h}_2)\Vert_{\mathcal{X}_3}.
\end{align*}
Here $C_2$ depends on $t$. With small enough $\e$ such that $C_2(t)\e < 1$, we conclude the proposition with Banach fixed point theorem.

\end{proof}

\ \\

\subsection{Proof of Theorem \ref{thm:weight_C1}}
The unique solution to \eqref{equation_f} is constructed as
\[f(x,\xi) = e^{-\gamma x}g(x,\xi) - e^{-\gamma x}h(x)\phi_u(\xi).\]
With $h(x)$ and $g(x)$ satisfying \eqref{penalized_C1}, we take $x$-derivative and have
\begin{align*}
   \Vert w_{\tilde{\theta}}\alpha \p_x f \Vert_{L^\infty_{\xi}} & \lesssim e^{-\gamma x} [\Vert wg\Vert_{L^\infty_{x,\xi}} + \Vert w \phi_u\Vert_{L^\infty_\xi} \Vert h\Vert_{L^\infty_x} ] \\
   & + e^{-\gamma x} [\Vert w_{\tilde{\theta}} \alpha \p_x g \Vert_{L^\infty_{x,\xi}} + \Vert w\phi_u\Vert_{L^\infty_\xi} \Vert  \alpha \p_x h\Vert_{L^\infty_x}] \lesssim e^{-\gamma x}\e.
\end{align*}
We conclude the theorem with some $\bar{C}>0$.

In summary, we have
\begin{align}\label{weighted_C1_2}
    & \Vert e^{\gamma x} w_{\tilde{\theta}} \alpha \p_x f\Vert_{L^\infty_{x,\xi}} \lesssim \e.
\end{align}
This verifies \eqref{weighted_C1}.

Then we prove \eqref{W1p}. For $\gamma_0 < \gamma$, \eqref{weighted_C1_2} leads to 
\begin{align}
    & \Vert w_{\tilde{\theta}/2} e^{\gamma_0 x}\p_x f\Vert_{L^{p}_{x,\xi}} =\Big(\int_{0}^\infty \int_{\mathbb{R}^3}  
 e^{p\tilde{\theta}|\xi|^2/2} e^{p\gamma_0 x} |\p_x f(x,\xi)|^p \dd \xi \dd x\Big)^{1/p}  \notag \\
    & \lesssim \Vert e^{\gamma x} w_{\tilde{\theta}} \alpha \p_x f\Vert_{L^\infty_{x,\xi}} \Big(\int_{0}^\infty \int_{\mathbb{R}^3}  \frac{e^{-(\gamma-\gamma_0) px}w^{-p}_{\tilde{\theta}/2}(\xi)}{\alpha^p(x,\xi)}     \dd \xi \dd x\Big)^{1/p}. \label{W1p_first_bdd}
\end{align}
From the definition of $\alpha(x,\xi)$ in \eqref{kinetic_weight}, when $\xi_1 \geq 1$ or $x\geq 1$, we have $\alpha(x,\xi)\gtrsim 1$. Then we proceed the computation as
\begin{align}
  \eqref{W1p_first_bdd}  & \lesssim \e \Big( \int_{\mathbb{R}^2}\dd \xi_2\dd \xi_3 w^{-p}_{\tilde{\theta}/2}(\xi_2)w^{-p}_{\tilde{\theta}/2}(\xi_3) \int_{0}^\infty  \int_{\mathbb{R}}  \frac{e^{-(\gamma-\gamma_0) px}w^{-p}_{\tilde{\theta}/2}(\xi_1)}{\alpha^p(x,\xi)}     \dd \xi_1 \dd x\Big)^{1/p} \notag\\
  & \lesssim \e \Big(\int_{|\xi_1|\geq 1 \text{ or }x\geq 1} e^{-(\gamma-\gamma_0) px} w_{\tilde{\theta}/2}^{-p}(\xi_1) \dd \xi_1 \dd x + \int_{0}^1 \int_{0}^1  \frac{1}{\big[|\xi_1+u|^2 + (c\nu_0)^2 x^2\big]^{p/2}}     \dd \xi_1 \dd x\Big)^{1/p}  \notag \\
  &\lesssim \e \Big(1+\int_{0}^1 \int_{0}^1  \frac{1}{\big[|\xi_1|^2 +  x^2\big]^{p/2}}     \dd \xi_1 \dd x\Big)^{1/p} .\label{W1p_second_bdd}
\end{align}
In the last inequality we apply the change of variable $\xi_1+u\to u$ and $c\nu_0 x \to x$. Note that the integral domain becomes $(x,\xi_1)\in [0,1]\times [0,1]\subset \{(x,\xi_1)|x^2+\xi_1^2 \leq 4\}$, then we apply the polar coordinate $x=r\cos\theta, \ \xi_1 = r\sin\theta$ to have
\begin{align*}
  \eqref{W1p_second_bdd}  & \lesssim \e  \Big(\int_0^{2\pi}\int_0^2   \frac{1}{r^p(|\cos^p \theta| + |\sin^p\theta| )} r \dd r  \dd \theta   \Big)^{1/p}  \\
  &\lesssim   \e \Big(\int_0^{2\pi} \frac{1}{|\cos^p \theta| + |\sin^p\theta|} \dd \theta\Big)^{1/p} \lesssim \e.
\end{align*}
In the last line, we have used $p<2$. We conclude the $W^{1,p}$ estimate \eqref{W1p}.

To prove \eqref{local_H1}, similar to the computation in \eqref{W1p_second_bdd}, we only focus on the integration over $[\delta,1]\times [0,1]$, we have
\begin{align*}
    & \int_{\delta}^\infty \int_{\mathbb{R}^3} w_{\tilde{\theta}}(\xi) e^{2\gamma_0 x} |\p_x f|^2 \dd \xi \dd x \\
    & \lesssim \e + \e\int_{0}^1 \int_{\delta}^1 \frac{1}{|\xi_1|^2 + x^2} \dd \xi_1 \dd x = \e + \e\int_{\delta}^1\frac{1}{x}\arctan(\frac{1}{x}) \dd x \lesssim C_\delta \e.   
\end{align*}
We conclude the theorem.

\ \\

\appendix

\section{Definition of \eqref{extra_assumption}}\label{appendix}
In the appendix we give the definition of the addition assumptions \eqref{extra_assumption}. To be specific, $Y_1[u]$ and $Y_2[u]$ are defined in \eqref{def_Y}, and $\mathfrak{R}_{u,\gamma}$ is defined in \eqref{def_R}.

The original version of the linearized penalized operator \eqref{pen_lin_op} has extra coefficients to be determined:
\begin{equation}\label{origin_operator}
\mathcal{L}^p g = \mathcal{L}p + \alpha \prod_+((\xi_1+u)g) + \beta \mathbf{p}_u g - \gamma (\xi_1+u)g.
\end{equation}
In order to remove the extra penalization terms, \cite{golse} proposed the following lemma:
\begin{lemma}[Lemma 4.3 in \cite{golse}]\label{lemma:remove}
Let
\begin{align*}
    \mathcal{A} := \begin{bmatrix}
       \alpha & 0  & -u \beta \langle \psi_u X_+\rangle \\
       0 & 0 &  -\beta  \langle \phi_u X_0 \rangle  \\
       \alpha \langle \psi_u X_+ \rangle & \frac{1}{u}\tau_u  & \tau_u - \beta \langle \psi_u \phi_u\rangle
     \end{bmatrix}
\end{align*}
For $|u|\ll 1$, $\mathcal{A}$ has three different eigenvalues. Let $(l_1,l_2,l_3)$ be a real basis of left eigenvectors of $\mathcal{A}$, define
\begin{equation}\label{def_Y}
\begin{split}
   Y_1[u](\xi) := & (X_+(\xi),X_0(\xi),\psi_u(\xi)) \cdot l_1(u) \\
   Y_2[u](\xi) := & (X_+(\xi),X_0(\xi),\psi_u(\xi)) \cdot l_2(u) .
\end{split}
\end{equation}
Then for $g$ satisfying \eqref{pen_prob} with $\mathcal{L}^p$ given by \eqref{origin_operator}, one has
\begin{align*}
    \langle (\xi_1+u)X_+ g\rangle = \langle (\xi_1+u)\psi_u g\rangle = 0 \Leftrightarrow \begin{cases}
        \langle (\xi_1+u)Y_1[u]g\rangle|_{x=0} = 0 \\
         \langle (\xi_1+u)Y_2[u]g\rangle|_{x=0} = 0 \end{cases}.
\end{align*}

\end{lemma}

Taking $\alpha = \beta = 2\gamma$ as in \eqref{pen_lin_op}, we obtain a unique solution $(g,h)$ to the nonlinear penalized problem \eqref{pen_nonlin_prob} as in Proposition \ref{prop:nonlinear_wellpose} and Proposition \ref{prop:nonlinear_continuous}. Define
\begin{equation}\label{def_R}
\mathfrak{R}_{u,\gamma}[f_b](\xi) :=  g(0,\xi) , \ \ \xi\in \mathbb{R}^3.
\end{equation}
Then removing the penalization in the nonlinear penalized problem \eqref{pen_nonlin_prob} is equivalent to imposing the condition in Lemma \ref{lemma:remove}:
\begin{align*}
    \langle (\xi_1+u)Y_1[u] \mathfrak{R}_{u,\gamma}[f_b] \rangle = \langle (\xi_1+u)Y_2[u] \mathfrak{R}_{u,\gamma}[f_b]\rangle = 0,
\end{align*}
which is exactly \eqref{extra_assumption}.

\ \\

\noindent \textbf{Acknowledgement.} 
The author thanks Professor Chanwoo Kim for suggesting the problem and engaging in stimulating discussion. HC is supported by NSF 2047681 and GRF grant (2130715) from RGC of Hong Kong. HC thanks the host from the Chinese University of Hong Kong, and thanks Professor Kung-Chien Wu for helpful discussion.

\bibliographystyle{siam}
\bibliography{citation}

\begin{thebibliography}{10}

\bibitem{bardos1986milne}
{\sc C.~Bardos, R.~E. Caflisch, and B.~Nicolaenko}, {\em {The Milne and Kramers
  problems for the Boltzmann equation of a hard sphere gas}}, Communications on
  pure and applied mathematics, 39 (1986), pp.~323--352.

\bibitem{bardos2006}
{\sc C.~Bardos, F.~Golse, and Y.~Sone}, {\em {Half-space problems for the
  Boltzmann equation: a survey}}, Journal of statistical physics, 124 (2006),
  pp.~275--300.

\bibitem{golse}
{\sc N.~Bernhoff and F.~Golse}, {\em {On the boundary layer equations with
  phase transition in the kinetic theory of gases}}, Archive for Rational
  Mechanics and Analysis, 240 (2021), pp.~51--98.

\bibitem{CKJ}
{\sc Y.~Cao, J.~Jang, and C.~Kim}, {\em {Passage from the Boltzmann equation
  with Diffuse Boundary to the Incompressible Euler equation with Heat
  Convection}}, arXiv preprint arXiv:2104.02169,  (2021).

\bibitem{CKL}
{\sc Y.~Cao, C.~Kim, and D.~Lee}, {\em {Global Strong Solutions of the
  Vlasov--Poisson--Boltzmann System in Bounded Domains}}, Archive for Rational
  Mechanics and Analysis,  (2019), pp.~1--104.

\bibitem{chen2}
{\sc H.~Chen}, {\em {Regularity of Boltzmann Equation with Cercignani--Lampis
  Boundary in Convex Domain}}, SIAM Journal on Mathematical Analysis, 54
  (2022), pp.~3316--3378.

\bibitem{chen2024global}
{\sc H.~Chen, R.~Duan, and J.~Zhang}, {\em Global dynamics of isothermal
  rarefied gas flows in an infinite layer}, arXiv preprint arXiv:2411.17068,
  (2024).

\bibitem{CK}
{\sc H.~Chen and C.~Kim}, {\em Regularity of stationary {B}oltzmann equation in
  convex domains}, Archive for Rational Mechanics and Analysis, 244 (2022),
  pp.~1099--1222.

\bibitem{CK_2023}
\leavevmode\vrule height 2pt depth -1.6pt width 23pt, {\em {Gradient Decay in
  the Boltzmann theory of Non-isothermal boundary}}, Archive for Rational
  Mechanics and Analysis, 248 (2024), p.~14.

\bibitem{Ikun}
{\sc I.-K. Chen, C.-H. Hsia, and D.~Kawagoe}, {\em {Regularity for diffuse
  reflection boundary problem to the stationary linearized Boltzmann equation
  in a convex domain}}, in Annales de l'Institut Henri Poincar{\'e} C, Analyse
  non lin{\'e}aire, vol.~36, Elsevier, 2019, pp.~745--782.

\bibitem{chen2023geometric}
{\sc I.-K. Chen, C.-H. Hsia, D.~Kawagoe, and J.-K. Su}, {\em Geometric effects
  on $w^{1,p}$ regularity of the stationary linearized boltzmann equation},
  arXiv preprint arXiv:2311.12387,  (2023).

\bibitem{coron1988}
{\sc F.~Coron, F.~Golse, and C.~Sulem}, {\em {A classification of well-posed
  kinetic layer problems}}, Communications on Pure and Applied Mathematics, 41
  (1988), pp.~409--435.

\bibitem{ghost}
{\sc R.~Esposito, Y.~Guo, R.~Marra, and L.~Wu}, {\em {Ghost effect from
  Boltzmann theory}}, arXiv preprint arXiv:2301.09427,  (2023).

\bibitem{R}
{\sc R.~T. Glassey}, {\em {The Cauchy problem in kinetic theory}}, SIAM, 1996.

\bibitem{golse2008}
{\sc F.~Golse}, {\em {Analysis of the boundary layer equation in the kinetic
  theory of gases}}, Bulletin of the Institute of Mathematics, Academia Sinica
  (New Series), 3 (2008), pp.~211--242.

\bibitem{golse1988}
{\sc F.~Golse, B.~Perthame, and C.~Sulem}, {\em {On a boundary layer problem
  for the nonlinear Boltzmann equation}}, Archive for Rational Mechanics and
  Analysis, 103 (1988), pp.~81--96.

\bibitem{G}
{\sc Y.~Guo}, {\em {Decay and continuity of the Boltzmann equation in bounded
  domains}}, Archive for rational mechanics and analysis, 197 (2010),
  pp.~713--809.

\bibitem{GKTT2}
{\sc Y.~Guo, C.~Kim, D.~Tonon, and A.~Trescases}, {\em {BV-regularity of the
  Boltzmann equation in non-convex domains}}, Archive for Rational Mechanics
  and Analysis, 220 (2016), pp.~1045--1093.

\bibitem{GKTT}
\leavevmode\vrule height 2pt depth -1.6pt width 23pt, {\em {Regularity of the
  Boltzmann equation in convex domains}}, Inventiones mathematicae, 207 (2017),
  pp.~115--290.

\bibitem{huang_diffuse}
{\sc F.~Huang and Y.~Wang}, {\em {Boundary Layer Solution of the Boltzmann
  Equation for Diffusive Reflection Boundary Conditions in Half-Space}}, SIAM
  Journal on Mathematical Analysis, 54 (2022), pp.~3480--3534.

\bibitem{huang_specular}
{\sc F.-m. Huang, Z.-h. Jiang, and Y.~Wang}, {\em {Boundary layer solution of
  the Boltzmann equation for specular boundary condition}}, Acta Mathematicae
  Applicatae Sinica, English Series, 39 (2023), pp.~65--94.

\bibitem{CJ1}
{\sc J.~Jang and C.~Kim}, {\em {Incompressible Euler limit from Boltzmann
  equation with Diffuse Boundary Condition for Analytic data}}, arXiv preprint
  arXiv:2005.12192,  (2020).

\bibitem{jiang2024knudsen}
{\sc N.~Jiang and Y.-L. Luo}, {\em Knudsen boundary layer equations for full
  ranges of cutoff collision kernels: Maxwell reflection boundary with all
  accommodation coefficients in [0, 1]}, arXiv preprint arXiv:2407.02852,
  (2024).

\bibitem{jiang2025knudsen}
{\sc N.~Jiang, Y.-L. Luo, Y.~Wu, and T.~Yang}, {\em Knudsen boundary layer
  equations with incoming boundary condition: full range of cutoff collision
  kernels and mach numbers of the far field}, arXiv preprint arXiv:2501.04035,
  (2025).

\bibitem{K}
{\sc C.~Kim}, {\em {Formation and propagation of discontinuity for Boltzmann
  equation in non-convex domains}}, Communications in mathematical physics, 308
  (2011), pp.~641--701.

\bibitem{KL}
{\sc C.~Kim and D.~Lee}, {\em {The Boltzmann equation with specular boundary
  condition in convex domains}}, Communications on Pure and Applied
  Mathematics, 71 (2018), pp.~411--504.

\bibitem{liu2013}
{\sc T.-P. Liu and S.-H. Yu}, {\em {Invariant manifolds for steady Boltzmann
  flows and applications}}, Archive for Rational Mechanics and Analysis, 209
  (2013), pp.~869--997.

\bibitem{sone2007}
{\sc Y.~Sone}, {\em {Molecular gas dynamics: theory, techniques, and
  applications}}, Springer, 2007.

\bibitem{sone2002}
{\sc Y.~Sone and Y.~Sone}, {\em {Kinetic theory and fluid dynamics}}, Springer,
  2002.

\bibitem{ukai2003}
{\sc S.~Ukai, T.~Yang, and S.-H. Yu}, {\em {Nonlinear boundary layers of the
  Boltzmann equation: I. Existence}}, Communications in mathematical physics,
  236 (2003), pp.~373--393.

\bibitem{ukai2004nonlinear}
\leavevmode\vrule height 2pt depth -1.6pt width 23pt, {\em {Nonlinear stability
  of boundary layers of the Boltzmann equation, I. The case $M^\infty<- 1$}},
  Communications in mathematical physics, 244 (2004), pp.~99--109.

\bibitem{wu2015}
{\sc L.~Wu and Y.~Guo}, {\em Geometric correction for diffusive expansion of
  steady neutron transport equation}, Communications in Mathematical Physics,
  336 (2015), pp.~1473--1553.

\bibitem{wu}
{\sc L.~Wu and Z.~Ouyang}, {\em {Asymptotic analysis of Boltzmann equation in
  bounded domains}}, arXiv preprint arXiv:2008.10507,  (2020).

\end{thebibliography}

\end{document}